\documentclass[12pt,reqno]{amsart}
\usepackage[margin=1in]{geometry}
\usepackage{amsmath,amssymb,amsthm,graphicx,amsxtra, setspace}
\usepackage[utf8]{inputenc}
\usepackage{mathrsfs}
\usepackage{hyperref}
\usepackage{upgreek}
\usepackage{mathtools}
\usepackage{xcolor}
\usepackage[mathcal]{euscript}
\allowdisplaybreaks

\usepackage[pagewise]{lineno}

\DeclareMathAlphabet{\mathpzc}{OT1}{pzc}{m}{it}

\usepackage[cyr]{aeguill}

\colorlet{darkblue}{blue!50!black}

\hypersetup{
	colorlinks,%
	citecolor=blue,%
	filecolor=red,%
	linkcolor=red,%
	urlcolor=blue,%
	pdfnewwindow=true,%
	pdfstartview={FitH}
}

\newtheorem{theorem}{Theorem}[section]
\newtheorem{lemma}[theorem]{Lemma}
\newtheorem{proposition}[theorem]{Proposition}

\newtheorem{example}[theorem]{Example}
\newtheorem{remark}[theorem]{Remark}

\newtheorem{hypothesis}[theorem]{Hypothesis}

\allowdisplaybreaks

\let\originalleft\left
\let\originalright\right
\renewcommand{\left}{\mathopen{}\mathclose\bgroup\originalleft}
\renewcommand{\right}{\aftergroup\egroup\originalright}

\let\emptyset\varnothing

\renewcommand{\d}{\/\mathrm{d}\/}

\def\w{\textbf{W}^{\varepsilon}_{{\theta}^{\varepsilon}}}

\def\L{\mathbb{L}}
\def\A{\mathrm{A}}
\def\I{\mathrm{I}}

\def\C{\mathrm{C}}
\def\f{\boldsymbol{f}}
\def\J{\mathrm{J}}
\def\B{\mathrm{B}}
\def\D{\mathrm{D}}
\def\y{\boldsymbol{y}}

\def\x{\boldsymbol{x}}

\def\g{\boldsymbol{g}}

\def\z{\boldsymbol{z}}
\def\v{\boldsymbol{v}}
\def\V{\mathbb{v}}
\def\w{\boldsymbol{w}}
\def\W{\mathrm{W}}

\def\N{\mathbb{N}}

\def\V{\mathbb{V}}
\def\wi{\widetilde}

\def\P{\mathrm{P}}

\def\H{\mathbb{H}}

\newcommand{\R}{\mathbb{R}}

\renewcommand{\d}{\/\mathrm{d}\/}

\newcommand{\Addresses}{{
		\footnote{
			
			\noindent \textsuperscript{1,2,3}Department of Mathematics, Indian Institute of Technology Roorkee-IIT Roorkee,
			Haridwar Highway, Roorkee, Uttarakhand 247667, INDIA.\par\nopagebreak
			\noindent  \textit{e-mail:} \texttt{Manil T. Mohan: maniltmohan@ma.iitr.ac.in, maniltmohan@gmail.com.}
			
			\textit{e-mail:} \texttt{Kush Kinra: kkinra@ma.iitr.ac.in.}
			
			\textit{e-mail:} \texttt{Sagar Gautam: sagar\_g@ma.iitr.ac.in.}
			
			\noindent \textsuperscript{*}Corresponding author.
			
			\textit{Key words:} convective Brinkman-Forchheimer equations, monotone operators, strong solution, stabilization, feedback control, time optimal control.
			
			Mathematics Subject Classification (2020): Primary 49J20, 49N35, 93D15; Secondary 35Q35, 76D03.

}}}

\begin{document}
	
	\title[CBF equations with potential]{2D and 3D convective Brinkman-Forchheimer equations perturbed by a subdifferential and applications to control problems
		\Addresses}
	\author[S. Gautam, K. Kinra and M. T. Mohan]
	{Sagar Gautam\textsuperscript{1}, Kush Kinra\textsuperscript{2} and Manil T. Mohan\textsuperscript{3*}}
	
	\maketitle

		\begin{abstract}
		The  following convective Brinkman-Forchheimer (CBF)  equations (or damped Navier-Stokes  equations) with potential 
		\begin{equation*}
		\frac{\partial \boldsymbol{y}}{\partial t}-\mu \Delta\boldsymbol{y}+(\boldsymbol{y}\cdot\nabla)\boldsymbol{y}+\alpha\boldsymbol{y}+\beta|\boldsymbol{y}|^{r-1}\boldsymbol{y}+\nabla p+\Psi(\boldsymbol{y})\ni\boldsymbol{g},\ \nabla\cdot\boldsymbol{y}=0, 
		\end{equation*}
	 in a $d$-dimensional torus  is considered in this work, where $d\in\{2,3\}$, $\mu,\alpha,\beta>0$ and $r\in[1,\infty)$. For $d=2$ with $r\in[1,\infty)$ and $d=3$ with $r\in[3,\infty)$ ($2\beta\mu\geq 1$ for $d=r=3$), we establish the existence of \textsf{\emph{a unique global strong solution}} for the above multi-valued  problem with the help of the \textsf{\emph{abstract theory of $m$-accretive operators}}. 
	 Moreover, we demonstrate that the same results hold \textsf{\emph{local in time}} for the case $d=3$ with $r\in[1,3)$ and $d=r=3$ with $2\beta\mu<1$. We explored the $m$-accretivity of the nonlinear as well as multi-valued operators, Yosida approximations and their properties, and several higher order energy estimates in the proofs.  For $r\in[1,3]$, we {quantize (modify)} the Navier-Stokes  nonlinearity $(\boldsymbol{y}\cdot\nabla)\boldsymbol{y}$ to establish the existence and uniqueness results, while for $r\in[3,\infty)$ ($2\beta\mu\geq1$ for $r=3$), we handle the Navier-Stokes nonlinearity by the nonlinear damping term $\beta|\boldsymbol{y}|^{r-1}\boldsymbol{y}$. Finally, we  discuss the applications of the above developed theory  in feedback control problems like  flow invariance, time optimal control and stabilization. 
	\end{abstract}

	\section{Introduction}\label{sec1}\setcounter{equation}{0}
	\subsection{The model} 
Let $\mathbb{T}^{d}=\big(\R/\mathbb{Z}\big)^{d}$ be a $d$-dimensional torus ($d=2,3$).  The convective Brinkman-Forchheimer (CBF) equations describe the motion of incompressible fluid flows in a saturated porous medium (\cite{DAN,DAN1}). {A porous medium or a porous material is defined as a solid (often called a matrix) permeated by an interconnected network of pores (voids) filled with a fluid (e.g., air or water). 
We consider a porous medium that is fully saturated with a fluid (\cite{TCY}).} With control applications in mind, we consider the following  CBF equations with potential (perturbed by a subdifferential, see Hypothesis \ref{hyp1} below): 
	\begin{equation}\label{1}
		\left\{
		\begin{aligned}
			\frac{\partial \y}{\partial t}-\mu \Delta\y+(\y\cdot\nabla)\y+\alpha\y+\beta|\y|^{r-1}\y+\nabla p+\Psi(\y)&\ni\g, \ \text{ in } \ \mathbb{T}^{d}\times(0,\infty), \\ \nabla\cdot\y&=0, \ \text{ in } \ \mathbb{T}^{d}\times[0,\infty), \\
			\y(0)&=\y_0 \ \text{in} \ \mathbb{T}^{d},
		\end{aligned}
		\right.
	\end{equation}
	where $\y(x,t):\mathbb{T}^{d}\times(0,\infty)\to\R^d$ represents the velocity field at time $t$ and position $x$, $p(x,t):\mathbb{T}^{d}\times(0,\infty)\to\R$ denotes the pressure field, $\g(x,t):\mathbb{T}^{d}\times(0,\infty)\to\R^d$ is an external forcing and $\Psi(\cdot)\subset\mathbb{L}^2(\mathbb{T}^d)\times\mathbb{L}^2(\mathbb{T}^d)$ is a multi-valued map. Moreover, $\y(\cdot,\cdot)$, $p(\cdot,\cdot)$ and $\g(\cdot,\cdot)$ satisfies the following periodic boundary conditions:
	\begin{align}\label{2}
		\y(x+e_{i},\cdot) = \y(x,\cdot), \ p(x+e_{i},\cdot) = p(x,\cdot) \ \text{and} \ \g(x+e_{i},\cdot) = \g(x,\cdot),
	\end{align}
for every $x\in\R^{d}$ and $i=1,\ldots,d,$ where $\{e_{1},\dots,e_{d}\}$ is the canonical basis of $\R^{d}.$ The constant $\mu>0$ denotes the \emph{Brinkman coefficient} (effective viscosity), the positive constants $\alpha$ and $\beta$ represent the \emph{Darcy} (permeability of porous medium) and \emph{Forchheimer} (proportional to the porosity of the material) coefficients, respectively. {The term $\beta|\y|^{r-1}\y$ with $r\geq1$ is known as the absorption or damping term and $r$ is known as the absorption exponent (cf. \cite{SNA}). The exponent $r=3$ is called the \emph{critical exponent}.} The critical homogeneous CBF equations (system \eqref{1} without potential, $r=3$ and $\g=\mathbf{0}$)  have the same scaling as Navier-Stokes equations (NSE)  only when $\alpha=0$ (\cite{KWH}).  We refer the case $r<3$ as \emph{subcritical} and $r>3$ as \emph{supercritical} (or fast growing nonlinearities). The model is accurate when the flow velocity is too large for Darcy's law to be valid, and apart from that the porosity is not too small (\cite{MTT}).  If one considers system \eqref{1} without potential and if $\alpha=\beta=0$, then we obtain the classical NSE, and if $\alpha, \beta>0$, then it can be considered as damped NSE. 
For 3D case, the authors in \cite{MTT} introduced the extra term $\beta_1|\y|^{r_1-1}\y$ in system \eqref{1} to model a pumping, when $\beta_1<0$ by opposition to the damping modeled through the term $\beta|\y|^{r-1}\y$ when $\beta>0$. For $\beta>0$ and $\beta_1\in\R$, the existence of weak solutions were obtained by assuming $r>r_1\geq 1$, and the continuous dependence on the data as well as the existence of strong solutions were established for $r>3$. As we are working on the torus $\mathbb{T}^d$  and $r_1<r$, by modifying some calculations accordingly, similar results of this paper holds true for the model proposed in \cite{MTT} {as well}. 

\subsection{Literature survey}
In the literature, CBF equations are also known as tamed NSE or NSE modified with an absorption term, cf. \cite{SNA,MRXZ} etc., and references therein. The damping $\alpha\y+\beta|\y|^{r-1}\y$ arises from the resistance to the motion of the flow, which describes several physical phenomena such as drag or friction effects, porous media flow, some dissipative mechanisms, cf. \cite{ZCQJ,KWH,MTT,ZZXW} etc., and references therein.  The continuous data assimilation problem for 3D CBF model is described in \cite{MTT}.  The global solvability results for CBF model (for fast growing nonlinearities in 3D)  can be accessed from \cite{SNA,KWH,KT2,MTT,MT1}, etc.  
Similar to 3D NSE, the existence of a unique global (in time) weak solution of 3D CBF equations with $r\in[1,3)$ (for any $\beta,\mu>0$) and $r=3$ (for $2\beta\mu<1$) is also an open problem.

The theory of monotone operators is an important tool in the study of nonlinear operator equations, we refer the readers to \cite{VB1,VB2,ZdSmP, Hu,JMSS}, etc., for more details. When the operator has some kind of monotonicity properties, then one can pass {to} the limit in the Galerkin and Faedo-Galerkin approximations of the original equation, with a-priori estimates that  are in general weaker than those necessary in the compactness methods (\cite{BPWFSS}). In particular, monotone operators are  suitable tools for studying variational inequalities (\cite{VB2}). Local and global solvability of NSE and Boussinesq equations with potential (or perturbed by a subdifferential) is established in \cite{AIL,AIL3}, respectively.

{
In a standard Faedo-Galerkin discretization scheme under a Gelfand triplet $\V\subset\H\equiv\H'\subset\V'$ (all are Hilbert spaces), we approximate the infinite-dimensional operator equation by a finite-dimensional one, whose well-posedness can be ensured by some classical existence and uniqueness results or some fixed point arguments.  Then we find some uniform bounds of these approximate solutions (a-priori estimates) and use the Banach-Alaoglu compactness argument to find a subsequence which converges weakly. Using some compactness results like Aubin-Lions, one can extract a further subsequence which converges strongly, so that the   weak limit provides us the solution of the original equation (cf. \cite{Te}). Whereas, the monotone operators is a class of nonlinear operators which provide a broad framework for the analysis of infinite-dimensional problems. Of special interest are the so-called maximal monotone operators (or $m$-accretive operators on Hilbert spaces identified with its own dual), which exhibit remarkable surjectivity properties. Further, for inclusion problems, once the operator becomes maximal, then one can use the Yosida approximation scheme, which is an important regularization technique for solving the evolution systems (\cite{VB1,OPHB}, etc.). In this way, one can avoid the tedious Faedo-Galerkin approximation scheme also.}

Control of ordinary/partial differential equations associated with fluid flow motions have numerous applications in science, engineering and technology. Behavior and control  of turbulent flows are some of the most difficult problems in fluid mechanics. 
By control of turbulent flows, we meant  to determine an optimal action which \emph{minimizes the turbulence} inside the flow, (cf. \cite{F.T.,Fu,G,OpVf}).  
One of the interesting control problem is the flow invariance preserving feedback controllers for fluid flows. The authors in \cite{VBSS} developed a procedure to design feedback controllers that ensure the resultant dynamics of turbulence preserve some prescribed physical constraints such as enstrophy, helicity, etc. Flow invariance of controlled flux sets with respect to NSE  is discussed in \cite{VBNH}. The existence problem of the variational inequality for Stokes equations and NSE with constraints of obstacle type is considered in  \cite{TFu,MGNK}, respectively. An another interesting feedback control problem is the time optimal control problem, where one finds a control of bang-bang type to reach a fixed state from an arbitrary state in minimal time (cf. \cite{VB2,wwxz}).  As far as the time optimal control of fluid flow models are concerned,  the time optimal control problem for 2D NSE, Boussinesq equations, 3D Navier-Stokes-Voigt equations, 2D CBF equations with $r\in[1,3],$ and 3D NS-$\alpha$ model is considered in \cite{CTATMN,TiOp1,SlGw,TiOp2,DsLt},  respectively. 
Stabilization of NSE is dealt to stabilize the equilibrium solution of NSE by using finite-dimensional feedback controllers having support either in interior or on the boundary of the domain (cf. \cite{VB6, VB7,VBLT}, etc.). The internal stabilizability of NSE (with slip and non-slip Dirichlet boundary conditions) is developed in \cite{VBL, VBT}. 
The author in \cite{AIL2} discussed the feedback stabilization of NSE preserving the  invariance of  a given convex set.  

\subsection{Difficulties, approaches and novelties}
The main concern for considering the CBF equations \eqref{1} in a $d$-dimensional torus is as follows. In the  torus $\mathbb{T}^d$, the Helmholtz-Hodge projection $\mathcal{P}$ and $\Delta$ {commutes} (\cite[Theorem 2.22]{JCR4}). So, the equality  (\cite[Lemma 2.1]{KWH})
\begin{align}\label{3}
	&\int_{\mathbb{T}^{d}}(-\Delta \boldsymbol{y}(x))\cdot|\boldsymbol{y}(x)|^{r-1}\boldsymbol{y}(x)\d x\nonumber\\&=\int_{\mathbb{T}^{d}}|\nabla \boldsymbol{y}(x)|^2|\boldsymbol{y}(x)|^{r-1}\d x+4\left[\frac{r-1}{(r+1)^2}\right]\int_{\mathbb{T}^{d}}|\nabla|\boldsymbol{y}(x)|^{\frac{r+1}{2}}|^2\d x,
\end{align}
is quite useful in obtaining regularity results. It is also noticed in the literature that the above equality may not be useful in  domains other than the whole domain or a $d$-dimensional torus (see \cite{KT2,MT1},  etc. for a detailed discussion). Recently, the authors in  \cite{DsSz} addressed this regularity problem for Dirichlet's boundary conditions and the well-posedness of CBF equations with potential in bounded domains will be a future work. 

The main difficulty with nonlinear terms arises when we multiply them by a generalized function  to get $m$-accretivity.  In the literature for  NSE with potential (cf. \cite{AIL,AIL2}) or for feedback control problems (cf. \cite{VBSS,TiOp1}), $\mathbb{V}$-quantization of the nonlinear term $(\y\cdot\nabla)\y$ is used to obtain the $m$-accretivity of the operators. {The terminologies ``quantize'' and ``quantization'' have been taken from the work \cite{VBSS}, which says that ``The intermediate $m$-accretive construction of the nonlinearity used in this paper is in fact a form of quantization conceptually similar to constructive quantum field theory \cite{qu2,qu1}.''} Whereas, for  the supercritical CBF  equations \eqref{1} (that is, for $r>3$),  one can handle the NSE nonlinearity  $(\y\cdot\nabla)\y$ by the Forchheimer  nonlinearity  $|\y|^{r-1}\y$ (see steps \eqref{2.30}, \eqref{373}, etc. below). Along with this fact, the monotonicity of the nonlinear term $|\y|^{r-1}\y$ helps to obtain the $m$-accretivity of the operators  without  using a quantization technique (see Proposition \ref{prop33} below). The same results hold true for $r=3$ with $2\beta\mu\geq 1$ also without quantization, but for $r\in[1,3]$ ($2\beta\mu<1$ for $r=3$), we need an $\wi\L^4$-quantization technique (Appendix  \ref{Appen.}). 


The condition $(\mathrm{A}\y,\Phi_{\lambda}(\y))\geq -\gamma(1+\|\y\|_{\H}^{2})$  is considered in Hypothesis \ref{hyp1} (H.3) for the case $d=3$ with $r\in[5,\infty)$. This is required to handle the term $|(\mathcal{C}(\y_{\lambda}),\Phi_{\lambda}(\y_{\lambda}))|$ in Proposition \ref{prop3.3}, while taking the inner product with $\Phi_{\lambda}(\y_{\lambda})$ for the Yoisida approximated stationary problem. The Sobolev embedding $\V\subset\wi\L^p$ for any $p\in[1,\infty)$ helps us to resolve this problem in 2D, whereas in 3D, the embedding is true only for $p\in[2,6]$.  Moreover, for the supercritical case,  by choosing
$
	\mathcal{F}^{1}(\cdot) = \mu(1-\delta_1)\mathrm{A}+\beta(1-\delta_2)\mathcal{C}(\cdot)\ \text{ and }\ 
	\mathcal{F}^{2}(\cdot) = \mu\delta_1\mathrm{A}+\mathcal{B}(\cdot)+\beta\delta_2\mathcal{C}(\cdot)+\kappa\mathrm{I},$
 for some $\delta_1, \delta_2\in(0,1)$ and $\kappa\geq\varrho=\frac{r-3}{2\mu(r-1)}\left(\frac{2}{\beta\mu (r-1)}\right)^{\frac{2}{r-3}},$ we used the well-known  perturbation theorem for nonlinear $m$-accretive operators (\cite[Theorem 3.5, Chapter II]{VB1}) to show that the operator $\mathcal{F}^{1}+\mathcal{F}^{2}=\mu\A+\mathcal{B}(\cdot)+\beta\mathcal{C}(\cdot)+\kappa\I$  with the domain $\D(\A)$  is $m$-accretive in $\H$.
For NSE with potential, the author in \cite{AIL} proved a result similar to Theorem \ref{thm1.2} by assuming that $\y_0\in\V\cap\D(\Phi)$ and $\f\in\mathrm{L}^2(0,T;\H)$. Under the same assumptions, we are able to prove Theorem \ref{thm1.2}  for the case $d=2,3$ and $r\in[1,3]$ only (Appendix \ref{Appen.}). For $d=2,3$ and $r\in(3,\infty)$, we need $\y_0\in\D(\A)\cap\D(\Phi)$ and $\f\in\mathrm{W}^{1,1}(0,T;\H)$ to control the term $\int_0^T \|\Phi_\lambda(\y_{\lambda}(t))\|_{\H}^2 \d t$ (Step IV, Proposition \ref{prop4.1}). In order to do this, we first obtain the regularity estimates  for 	$\left\|\frac{\d^+\y_\lambda(\cdot)}{\d t}\right\|_{\H}$  and $\int_0^T \left\|\frac{\d(\nabla\y_\lambda(t))}{\d t}\right\|_{\H}^2 \d t$ (Step II, Proposition \ref{prop4.1}) by taking  the difference of Yosida approximated CBF equations (see \eqref{3.29} below) at $t+h$ and $t$ for $h>0$ and $t\in[0,T]$ and then using the monotonicity of $\Phi_\lambda(\cdot)$. Due to the lack of Gateaux derivative of $\Phi_{\lambda}(\y_{\lambda}(\cdot))$, one cannot differentiate the equation \eqref{3.29}  and get the required estimates by taking inner product with $\frac{\d\y_\lambda(\cdot)}{\d t}$ in the resulting equation. This kind of difficulty is not appearing in the case of NSE. 


Flow {invariance} preserving feedback controllers for 2D as well 3D NSE with normal cone as potential were considered in \cite{VBSS}. The results obtained in the work \cite{VBSS} were global for $d=2$ and local for $d=3$. But the presence of the damping term $|\y|^{r-1}\y$ helps us to obtain global results in 3D as well  for supercritical CBF equations.   The author in \cite{TiOp2} discussed the time optimal control problem for 2D CBF equations with $r\in[1,3]$ by using a $\V$-quantization and $m$-accretivity of the nonlinear operators. Hypothesis \ref{hyp1} and Proposition \ref{thm1.1} help us to study the time optimal control problem of CBF equations for $d=2,3$ with $r>3$ also. Moreover, the author in \cite{AIL2} examined the feedback stabilization of 2D and 3D NSE preserving the invariance of a given convex set by deducing the existence of  weak solutions (uniqueness only in 2D)  for the NSE system perturbed by a subdifferential. Whereas, for the CBF equations \eqref{1p4}, one can address similar problems  for $d=2,3$ with $r>3$ by establishing  uniqueness results also.
\subsection{Outline of the paper} 
The rest of the paper is organized as follows: The next section  is devoted for the functional settings, definition and properties of linear, bilinear and nonlinear operators. Moreover, we have also stated our main result of this work (see Theorem \ref{thm1.2}) in the same section. In Section \ref{sec3}, we prove $m$-accretivity of single and multi-valued operators (see Propositions \ref{prop33}-\ref{prop3.3}). In Section \ref{sec44}, we derive higher order energy estimates in order to prove Theorem \ref{thm1.2} (see Proposition \ref{prop4.1}). Then, we prove  some convergence results  using the Banach-Alaoglu theorem and Aubin-Lions compactness  lemma (see Proposition \ref{soln}). Lastly, we conclude the proof of Theorem \ref{thm1.2} by using Proposition \ref{prop3.5}, and the uniqueness result is provided in Proposition \ref{unique}. In Section \ref{sec5}, we discuss three applications of Theorem \ref{thm1.2}, namely,  flow invariance preserving feedback controllers, time optimal control problem and feedback stabilization.

\section{Functional Settings and Preliminaries}\label{sec2}\setcounter{equation}{0}

In this section, we provide the necessary functional setting needed to obtain the results of this work. We consider the problem \eqref{1} on a $d$-dimensional torus $\mathbb{T}^{d}=\big(\R/\mathbb{Z}\big)^{d}$, with periodic boundary conditions \eqref{2}.

	\subsection{Function spaces} Let \ $\C_{\mathrm{p}}^{\infty}(\mathbb{T}^d;\R^d)$ denote the space of all infinitely differentiable  functions ($\mathbb{R}^d$-valued) satisfying periodic boundary conditions \eqref{2}.  As observed in \cite{KWH} that the absorption term $\beta|\y|^{r-1}\y$ does not preserve zero average condition like in the case of NSE, so we cannot use the well-known Poincar\'e inequality and we have to deal with  full $\H^1$-norm. The Sobolev space  $\H_{\mathrm{p}}^s(\mathbb{T}^d):=\mathrm{H}_{\mathrm{p}}^s(\mathbb{T}^d;\mathbb{R}^d)$ is the completion of $\C_{\mathrm{p}}^{\infty}(\mathbb{T}^d;\R^d)$  with respect to the $\H^s$-norm $$\|\y\|_{\H^s_{\mathrm{p}}}:=\left(\sum_{0\leq|\alpha|\leq s}\|\D^{\alpha}\y\|_{\mathbb{L}^2(\mathbb{T}^d)}^2\right)^{1/2}.$$ The Sobolev space of periodic functions $\H_{\mathrm{p}}^s(\mathbb{T}^d)$ is the same as (\cite{JCR}) $$\left\{\y:\y(x)=\sum_{k\in\mathbb{Z}^d}\y_{k} \mathrm{e}^{2\pi i k\cdot x},\ \overline{\y}_{k}=\y_{-k},\ \|\y\|_{\H^s_f}:=\left(\sum_{k\in\mathbb{Z}^d}(1+|k|^{2s})|\y_{k}|^2\right)^{\frac{1}{2}}<\infty\right\}.$$ From \cite[Proposition 5.38]{JCR}, we infer that the norms $\|\cdot\|_{\dot{\H}^s_p}$ and $\|\cdot\|_{\dot{\H}^s_f}$ are equivalent. Let us define 
\begin{align*} 
	\mathcal{V}&:=\{\y\in\C_{\mathrm{p}}^{\infty}(\mathbb{T}^d;\R^d):\nabla\cdot\y=0\}.
\end{align*}
The spaces $\H$ and $\widetilde{\L}^{p}$ are the closure of $\mathcal{V}$ in the Lebesgue spaces $\mathrm{L}^2(\mathbb{T}^d;\R^d)$ and $\mathrm{L}^p(\mathbb{T}^d;\R^d)$ for $p\in(2,\infty)$, respectively. The space $\V$ is the closure of $\mathcal{V}$ in the Sobolev space $\mathrm{H}^1(\mathbb{T}^d;\R^d)$. We characterize the spaces $\H$, $\widetilde{\L}^p$ and $\V$ with the norms  $$\|\y\|_{\H}^2:=\int_{\mathbb{T}^d}|\y(x)|^2\d x,\quad \|\y\|_{\widetilde{\L}^p}^p:=\int_{\mathbb{T}^d}|\y(x)|^p\d x\ \text{ and }\ \|\y\|_{\V}^2:=\int_{\mathbb{T}^d}\left(|\y(x)|^2+|\nabla\y(x)|^2\right)\d x,$$ respectively.
Let $(\cdot,\cdot)$ denote the inner product in the Hilbert space $\H$ and $\langle \cdot,\cdot\rangle $ represent the induced duality between the spaces $\V$  and its dual $\V'$ as well as $\widetilde{\L}^p$ and its dual $\widetilde{\L}^{p'}$, where $\frac{1}{p}+\frac{1}{p'}=1$. Note that $\H$ can be identified with its own dual $\H'$. The sum space $\V'+\widetilde{\L}^{p'}$ is well defined (see \cite[Subsection 2.1]{FKS}). Furthermore, we have 
\begin{align*}
(\V'+\widetilde{\L}^{p'})'=	\V\cap\widetilde{\L}^p \  \text{and} \ (\V\cap\widetilde{\L}^p)'=\V'+\widetilde{\L}^{p'},
\end{align*} 
where $\|\y\|_{\V\cap\wi\L^{p}}=\max\{\|\y\|_{\V},\|\y\|_{\wi\L^p}\},$ which is equivalent to the norms  $\|\y\|_{\V}+\|\y\|_{\widetilde{\L}^{p}}$  and $\sqrt{\|\y\|_{\V}^2+\|\y\|_{\widetilde{\L}^{p}}^2}$, and  
\begin{align*}
	\|\y\|_{\V'+\widetilde{\L}^{p'}}&=\inf\{\|\y_1\|_{\V'}+\|\y_2\|_{\wi\L^{p'}}:\y=\y_1+\y_2, \y_1\in\V' \ \text{and} \ \y_2\in\wi\L^{p'}\}\nonumber\\&=
	\sup\left\{\frac{|\langle\y_1+\y_2,\f\rangle|}{\|\f\|_{\V\cap\widetilde{\L}^p}}:\boldsymbol{0}\neq\f\in\V\cap\widetilde{\L}^p\right\}.
\end{align*}
Note that $\V\cap\widetilde{\L}^p$ and $\V'+\widetilde{\L}^{p'}$ are Banach spaces. Moreover, we have the continuous embedding $\V\cap\widetilde{\L}^p\hookrightarrow\V\hookrightarrow\H\cong\H^{\prime}\hookrightarrow\V'\hookrightarrow\V'+\widetilde{\L}^{p'}$. 
\subsection{Linear operator}\label{liop}
Let $\mathcal{P}_p: \L^p(\mathbb{T}^d) \to\wi\L^p,$ $p\in[1,\infty)$ be the Helmholtz-Hodge (or Leray) projection  (cf.  \cite{JBPCK,DFHM}, etc.).	Note that $\mathcal{P}_p$ is a bounded linear operator and for $p=2$,  $\mathcal{P}:=\mathcal{P}_2$ is an orthogonal projection (\cite[Section 2.1]{JCR4}). We define the Stokes operator 
\begin{equation*}
	\A\y:=-\mathcal{P}\Delta\y,\;\y\in\D(\A):=\V\cap\H^{2}_\mathrm{p}(\mathbb{T}^d).
\end{equation*}
Note that $\D(\A)$ can also be written as $\D(\A)=\big\{\y\in\H^{2}_\mathrm{p}(\mathbb{T}^d):\nabla\cdot\y=0\big\}$.  It should be noted that $\mathcal{P}$ and $\Delta$ commutes in a {torus} (\cite[Lemma 2.9]{JCR4}). For the Fourier expansion $\y(x)=\sum\limits_{k\in\mathbb{Z}^d} e^{2\pi i k\cdot x} \y_{k} ,$ we have by using Parseval's identity
\begin{align*}
	\|\y\|_{\H}^2=\sum\limits_{k\in\mathbb{Z}^d} |\y_{k}|^2 \  \text{and} \ \|\A\y\|_{\H}^2=(2\pi)^4\sum_{k\in\mathbb{Z}^d}|k|^{4}|\y_{k}|^2.
\end{align*}
Therefore, we have 
\begin{align*}
	\|\y\|_{\H^2_\mathrm{p}(\mathbb{T}^d)}^2=\sum_{k\in\mathbb{Z}^d}(1+|k|^{4})|\y_{k}|^2=\|\y\|_{\H}^2+\frac{1}{(2\pi)^4}\|\A\y\|_{\H}^2\leq\|\y\|_{\H}^2+\|\A\y\|_{\H}^2. 
\end{align*}
Moreover, by the definition of $\|\cdot\|_{\H^2_\mathrm{p}(\mathbb{T}^d)}$, we have $	\|\y\|_{\H^2_\mathrm{p}(\mathbb{T}^d)}^2\geq\|\y\|_{\H}^2+\|\A\y\|_{\H}^2$ and hence it is immediate that $\D(\I+\A)=\H^2_\mathrm{p}(\mathbb{T}^d)$.

\subsection{Bilinear operator}
Let us define the \textsl{trilinear form} $b(\cdot,\cdot,\cdot):\V\times\V\times\V\to\R$ by $$b(\y,\z,\w)=\int_{\mathbb{T}^d}(\y(x)\cdot\nabla)\z(x)\cdot\w(x)\d x=\sum_{i,j=1}^d\int_{\mathbb{T}^d}\y_i(x)\frac{\partial \z_j(x)}{\partial x_i}\w_j(x)\d x.$$ If $\y, \z$ are such that the linear map $b(\y, \z, \cdot) $ is continuous on $\V$, the corresponding element of $\V'$ is denoted by $\mathcal{B}(\y, \z)$. We also denote $\mathcal{B}(\y) = \mathcal{B}(\y, \y)=\mathcal{P}[(\y\cdot\nabla)\y]$.
An integration by parts yields 
\begin{equation}\label{b0}
	\left\{
	\begin{aligned}
		b(\y,\z,\w) &=  -b(\y,\w,\z),\ \text{ for all }\ \y,\z,\w\in \V,\\
		b(\y,\z,\z) &= 0,\ \text{ for all }\ \y,\z \in\V.
	\end{aligned}
	\right.\end{equation}


%

\subsection{Nonlinear operator}
Let us now consider the operator $\mathcal{C}(\y):=\mathcal{P}(|\y|^{r-1}\y)$. It is immediate that $\langle\mathcal{C}(\y),\y\rangle =\|\y\|_{\widetilde{\L}^{r+1}}^{r+1}$ and the map $\mathcal{C}(\cdot):\V\cap\widetilde{\L}^{r+1}\to\V'+\widetilde{\L}^{\frac{r+1}{r}}$ is Gateaux differentiable with Gateaux derivative 
	\begin{align}\label{29}
	\mathcal{C}'(\y)\z&=\left\{\begin{array}{cl}\mathcal{P}(\z),&\text{ for }r=1,\\ \left\{\begin{array}{cc}\mathcal{P}(|\y|^{r-1}\z)+(r-1)\mathcal{P}\left(\frac{\y}{|\y|^{3-r}}(\y\cdot\z)\right),&\text{ if }\y\neq \mathbf{0},\\\mathbf{0},&\text{ if }\y=\mathbf{0},\end{array}\right.&\text{ for } 1<r<3,\\ \mathcal{P}(|\y|^{r-1}\z)+(r-1)\mathcal{P}(\y|\y|^{r-3}(\y\cdot\z)), &\text{ for }r\geq 3,\end{array}\right.
\end{align}
for all $\y,\z\in\widetilde{\L}^{r+1}$. 
From  {\cite[pp. 104]{pL}} (also see \cite[Subsection 2.4]{MT2}), we have 
\begin{align}\label{2.23}
	\langle\mathcal{C}(\y)-\mathcal{C}(\z),\y-\z\rangle&\geq \frac{1}{2}\||\y|^{\frac{r-1}{2}}(\y-\z)\|_{\H}^2+\frac{1}{2}\||\z|^{\frac{r-1}{2}}(\y-\z)\|_{\H}^2\nonumber\\&\geq \frac{1}{2^{r-1}}\|\y-\z\|_{\wi\L^{r+1}}^{r+1}\geq 0,
\end{align}
for $r\geq 1$. This shows that $\mathcal{C}(\cdot)$ is monotone operator. 

\begin{remark}\label{e.2222}	
 In periodic domain (cf. \cite[Subsection 3.5]{MT2}), we have 
\begin{align}\label{P.11}
	\|\y\|_{\widetilde{\L}^{3(r+1)}}^{r+1}\leq C\int_{\mathbb{T}^d}|\nabla \y(x)|^2|\y(x)|^{r-1}\d x,
\end{align}
for $d=3$ and $r\geq 1$. Also, from \cite[Lemma 2.2]{LSM}, we obtain 
\begin{align}\label{P.22}
	\|\y\|_{\widetilde{\L}^{p(r+1)}}^{r+1}=\||\y|^{\frac{r+1}{2}}\|_{\wi\L^{2p}(\mathbb{T}^d)}^{2}\leq C \int_ {\mathbb{T}^d}|\nabla|\y|^{\frac{r+1}{2}}|^{2}\d x\ {\leq}\  C\int_{\mathbb{T}^d}|\nabla \y(x)|^2|\y(x)|^{r-1}\d x,
\end{align}
for $d=2$ and for all $p\in[2,\infty).$	
\end{remark}

\subsection{Main results} 
The main objective of this work is to establish the solvability results of the inclusion problem \eqref{1} and discuss their applications in the context of control problems. Since $\alpha$ is not playing a major role in this work, so we fix $\alpha=0$ in the rest of the paper.  Let us state the main results of this work for the problem \eqref{1} in an abstract framework (see \eqref{1p4} below). We will prove these results in the subsequent sections. Let us denote $\f=\mathcal{P}\g$ and $\Phi(\cdot)=\mathcal{P}\Psi(\cdot)$, where $\mathcal{P}$ is the Helmholtz-Hodge (or Leray) projection. The following assumption is imposed on $\Phi(\cdot)$ to achieve our goals which is similar to the work \cite{AIL}. 
\begin{hypothesis}\label{hyp1}
	Let $\Phi$ be a maximal monotone operator on $\H\times\H$ satisfying the following hypothesis:
	\begin{enumerate}
		\item[(H.1)] $\Phi = \partial\varphi$, where $\varphi: \H\to\overline{\R}:=\R\cup\{+\infty\}$ is a lower semicontinuous proper convex function.
		\item[(H.2)] $\boldsymbol{0}\in\mathrm{D}(\Phi)$.
		\item[(H.3)] {For some $\gamma\geq 0$ and $\varsigma\in(0,\frac{1}{\mu})$ (when  $d=2$ with  $r\in[1,\infty)$ and $d=3$  with  $r\in[1,5)$), and for some $\gamma\geq0$ and $\varsigma=0$ (when $d=3$ with $r\in[5,\infty)$),  the following inequality hold:
		\begin{align*}
			&(\mathrm{A}\y,\Phi_{\lambda}(\y))\geq -\gamma(1+\|\y\|_{\H}^{2})-\varsigma\|\Phi_{\lambda}(\y)\|_{\H}^{2},
		\end{align*}}
	\end{enumerate}
	for all $\lambda>0$ and $\y\in\mathrm{D}(\mathrm{A})$, where $ \Phi_{\lambda} = \frac{1}{\lambda}(\mathrm{I}-(\mathrm{I}+\lambda\Phi)^{-1}):\H\to\H$ is the Yosida approximation of $\Phi$. 
\end{hypothesis}
We observe from the above expression that  $\Phi_\lambda(\y)\in\Phi((\mathrm{I}+\lambda\Phi)^{-1})(\y)),$
for every $\y\in\H$ and $\lambda>0.$ Let us discuss one example which satisfies the Hypothesis \ref{hyp1}.
\begin{example}\label{exm}
{	Let $\mathcal{K}$ be a closed and convex subset of $\H$ satisfying
	\begin{align}\label{inv}
		(\I+\lambda\A)^{-1}\mathcal{K}\subset\mathcal{K}.
	\end{align}
 We consider the indicator function $\I_{\mathcal{K}}:\H\to\overline{\R}$ (\cite{VB2}) by 
	\begin{align*}
		\I_{\mathcal{K}}(\x)=
		\begin{cases}
			\boldsymbol{0},  &\text{if} \ \x\in\mathcal{K},\\
			+\infty, &\text{if} \  \x\notin\mathcal{K},
		\end{cases}
	\end{align*}
	whose subdifferential is given by 
	\begin{align*}
		\partial\I_{\mathcal{K}}(\x)=
		\begin{cases}
			\emptyset, &\text{if} \ \x\notin\mathcal{K},\\
			\{\boldsymbol{0}\}, &\text{if} \  \x\in \mathrm{int}(\mathcal{K}),\\
			N_{\mathcal{K}}(\x), &\text{if} \ \x\in\mathrm{bdy}(\mathcal{K}),
		\end{cases}
	\end{align*}
	where $\mathrm{int}(\mathcal{K})$ and $\mathrm{bdy}(\mathcal{K})$ denote the interior and boundary  of $\mathcal{K}$, respectively. Then a regularization of $\I_{\mathcal{K}}$ is given by (see \cite[Chapter 2, Theorem 2.2]{VB2}) $$(\I_{\mathcal{K}})_{\lambda}(\x)=\frac{1}{2\lambda}\|\x-\P_{\mathcal{K}}(\x)\|_{\H}^2,$$
	and its Gateaux derivative $$(\partial\I_{\mathcal{K}})_{\lambda}(\x)=\frac{1}{\lambda}(\x-\P_{\mathcal{K}}(\x)),$$ 
	where $\P_{\mathcal{K}}:\L^2(\mathbb{T}^d)\to\mathcal{K}$ is the projection operator of $\x$ onto $\mathcal{K}$ which is equal to the resolvent  $(\I+\lambda\partial\I_{\mathcal{K}})^{-1}.$ It implies that the above derivative is equal to the Yosida approximation of $\partial\I_{\mathcal{K}},$ that is, $$(\partial\I_{\mathcal{K}})_{\lambda}(\x)=\frac{1}{\lambda}(\x-(\I+\lambda\partial\I_{\mathcal{K}})^{-1}(\x)),  \  \text{ for all } \  \x\in\H.$$
	Then from \cite[Theorem 2.1, Chapter 2]{VB2}, we observe that the multi-valued operator $\Phi:=N_{\mathcal{K}}$ is a maximal monotone operator with $\boldsymbol{0}\in\D(\Phi)=\D(\partial\I_{\mathcal{K}})=\mathcal{K}.$ Also, from \cite[Chapter 1, pp. 30]{VB2} since $\A$ is single-vlaued maximal monotone operator in $\H$, thus from \cite[Chapter IV, Proposition 1.1, part(iv)]{VB1}, we have 
	\begin{align*}
		(\A\y,(\partial\I_{\mathcal{K}})_{\lambda}(\y))\geq0,  \  \text{ for all } \ \y\in\D(\A),\  \lambda>0.
	\end{align*}
	Thus the multi-valued operator $N_{\mathcal{K}}=\partial\I_{\mathcal{K}}$ satisfies all the assumptions (H1)-(H3)  of Hypothesis \ref{hyp1} (see Subsections \ref{TO}-\ref{St} for further examples).}
\end{example}

\begin{theorem}\label{thm1.2}
	Let $T>0$ and assume that $\Phi\subset\H\times\H$ satisfies Hypothesis \ref{hyp1}. Let $\y_0\in\D(\A)\cap\D(\Phi)$ and $\f \in \mathrm{W}^{1,1}(0,T;\H)$. For $d=2,3$ with $r\in[3,\infty)$ and $d=r=3$ with $2\beta\mu\geq1$, there exists a unique strong solution 
	\begin{align}\label{regu}
		\y\in{\W^{1,\infty}(0,T;\H)}\cap \C([0,T];\V)\cap\mathrm{L}^{\infty}(0,T;\D(\A))\cap\mathrm{L}^{r+1}(0,T;\wi\L^{3(r+1)})\cap\mathrm{W}^{1,2}(0,T;\V),
	\end{align}   
 such that in $\H$
	\begin{equation}\label{1p7}
		\left\{
		\begin{aligned}
			\frac{\d \y(t)}{\d t}+\mu\A\y(t)+\mathcal{B}(\y(t))+\beta\mathcal{C}(\y(t))+\Phi(\y(t))&\ni \f(t), \ \text{ a.e. } \ t\in[0,T], \\
			\y(0)&=\y_0. 
		\end{aligned}
		\right.
	\end{equation}
\end{theorem}

\begin{remark}
	1.) Theorem \ref{thm1.2} holds true if one replaces $(1+\|\y\|_{\H}^{2})$ by $(1+\|\y\|_{\V}^{2})$ in $\mathrm{(H.3)}$. \\
	2.) Using  condition $\mathrm{(H.1)}$, the system \eqref{1p7} can be considered as \textsl{CBF equations perturbed by a subdifferential.}
\end{remark}

\section{Abstract results}\label{sec3}
In this section, we first prove the $m$-accretivity of a single-valued and multi-valued operator. Then, we show abstract result (Proposition \ref{thm1.1}) by using the abstract theory available in \cite{VB1,VB2}.
\begin{proposition}\label{prop33}
For $d=2,3$ with $r>3$, define the operator $\mathcal{G}(\cdot):\mathrm{D}(\mathcal{G})\to\H$ by
	\begin{align*}
		\mathcal{G}(\cdot)=\mu\A+\B(\cdot)+\beta\mathcal{C}(\cdot), 
	\end{align*}
where $\mathrm{D}(\mathcal{G})=\{\y\in\V\cap\wi\L^{r+1}:\A\y\in\H\}.$ Then  $\mathcal{G}+\kappa\mathrm{I}$ is $m$-accretive in $\H\times\H$ for some $\kappa\geq\varrho,$ where 
\begin{align}
	\label{215}\varrho=\frac{r-3}{2\mu(r-1)}\left(\frac{2}{\beta\mu(r-1)}\right)^{\frac{2}{r-3}}.
\end{align}
\end{proposition}

\begin{proof}
 We shall first show that $\mathcal{G}+\kappa\mathrm{I}$ is a monotone operator for $\kappa\geq\varrho>0$. Then we will show that $\mathcal{G}+\kappa\mathrm{I}$ is coercive and demicontinuous, which imply  the $m$-accretivity of the operator $\mathcal{G}+\kappa\I$.  The proof is divided into {the}  following four steps:
\vskip 2mm
\noindent
\textbf{Step I:} \textsl{$\mathcal{G}+\kappa\mathrm{I}$ is monotone for some $\kappa>0$.} 	We estimate $	\langle\A\y-\A\z,\y-\z\rangle $ by	using an integration by parts as
\begin{align}\label{ae}
	\langle\A\y-\A\z,\y-\z\rangle =\|\nabla(\y-\z)\|^2_{\H}.
\end{align}
Note that $\langle\mathcal{B}(\y,\y-\z),\y-\z\rangle=0$ which implies along with H\"older's and Young's inequalities that
\begin{align}\label{2..28}
	|\langle \mathcal{B}(\y)-\mathcal{B}(\z),\y-\z\rangle| &=|\langle\mathcal{B}(\y-\z,\y-\z),\z\rangle|\leq \frac{\mu }{2}\|\nabla(\y-\z)\|_{\H}^2+\frac{1}{2\mu }\|\z(\y-\z)\|_{\H}^2.
\end{align} 
We take the term $\|\z(\y-\z)\|_{\H}^2$ from \eqref{2..28} and use H\"older's and Young's inequalities to estimate it as (see \cite{KWH} also)
\begin{align}\label{2..29}
	\int_{\mathbb{T}^d}|\z(x)|^2|\y(x)-\z(x)|^2\d x &=\int_{\mathbb{T}^d}|\z(x)|^2|\y(x)-\z(x)|^{\frac{4}{r-1}}|\y(x)-\z(x)|^{\frac{2(r-3)}{r-1}}\d x\nonumber\\&\leq\left(\int_{\mathbb{T}^d}|\z(x)|^{r-1}|\y(x)-\z(x)|^2\d x\right)^{\frac{2}{r-1}}\left(\int_{\mathbb{T}^d}|\y(x)-\z(x)|^2\d x \right)^{\frac{r-3}{r-1}}\nonumber\\&\leq{\beta\mu} \||\z|^{\frac{r-1}{2}}(\y-\z)\|_{\H}^2+\frac{r-3}{r-1}\left[\frac{2}{\beta\mu (r-1)}\right]^{\frac{2}{r-3}}\|\y-\z\|_{\H}^2,
\end{align}
for $r>3$. Using \eqref{2..29} in \eqref{2..28}, we find 
\begin{align}\label{2.30}
	|\langle\mathcal{B}(\y)-\mathcal{B}(\z),\y-\z\rangle|\leq\frac{\mu}{2} \|\nabla(\y-\z)\|_{\H}^2+\frac{\beta}{2}\||\z|^{\frac{r-1}{2}}(\y-\z)\|_{\H}^2+\varrho\|\y-\z\|_{\H}^2,
\end{align}
where $\varrho=\frac{r-3}{2\mu(r-1)}\left[\frac{2}{\beta\mu (r-1)}\right]^{\frac{2}{r-3}}.$ From \eqref{2.23}, we easily have 
\begin{align}\label{2.27}
	\beta	\langle\mathcal{C}(\y)-\mathcal{C}(\z),\y-\z\rangle \geq \frac{\beta}{2}\||\z|^{\frac{r-1}{2}}(\y-\z)\|_{\H}^2. 
\end{align}
Combining \eqref{ae} and \eqref{2.30}-\eqref{2.27}, we conclude that 
\begin{align}
	\langle(\mathcal{G}+\kappa\I)(\y)-(\mathcal{G}+\kappa\I)(\z),\y-\z\rangle&\geq\frac{\mu}{2} \|\nabla(\y-\z)\|_{\H}^2+(\kappa-\varrho)\|\y-\z\|_{\H}^2\geq\frac{\mu}{2} \|\nabla(\y-\z)\|_{\H}^2,
\end{align}
for $\kappa\geq\varrho$ and $r>3$. Thus $\mathcal{G}+\kappa\I$ is monotone.
\vskip 2mm
\noindent
\textbf{Step II:} \textsl{$\mathcal{G}+\kappa\I$ is demicontinuous.} Let us take a sequence $\y^n\to \y$ in $\V\cap\widetilde{\L}^{r+1},$ so that $\|\y^n-\y\|_{\V}+\|\y^n-\y\|_{\wi\L^{r+1}}\to 0$ as $n\to\infty$. For any $\z\in\V\cap\widetilde{\L}^{r+1}$, we consider 
\begin{align}\label{214}
	&\langle(\mathcal{F}+\kappa\I)(\y^n)-(\mathcal{F}+\kappa\I)(\y),\z\rangle\nonumber\\ &=\mu \langle \A(\y^n)-\A(\y),\z\rangle+\langle\B(\y^n)-\mathcal{B}(\y),\z\rangle-\beta\langle \mathcal{C}(\y^n)-\mathcal{C}(\y),\z\rangle+\kappa(\y_{n}-\y,\z).
\end{align} 
Note that 
\begin{align*}
		|\mu\langle \mathrm{A}(\y_{n}-\y),\z\rangle+\kappa(\y_{n}-\y,\z)|\leq\mu\|\nabla(\y_n-\y)\|_{\H}\|\nabla\z\|_{\H}+\kappa\|\y_n-\y\|_{\H}\|\z\|_{\H}\to 0,
\end{align*}
 as  $n\to\infty$, since $\y^n\to \y$ strongly in $\V\cap\wi\L^{r+1}$. We estimate the term $	|\langle\B(\y^n)-\mathcal{B}(\y),\z\rangle|$ by using the H\"older's and interpolation inequalities 
\begin{align*}
	|\langle\B(\y^n)-\B(\y),\z\rangle|&=|\langle\B(\y^n,\y^n-\y),\z\rangle+\langle\B(\y^n-\y,\y),\z\rangle|
	\nonumber\\&\leq\left(\|\y^n\|_{\widetilde{\L}^{\frac{2(r+1)}{r-1}}}+\|\y\|_{\widetilde{\L}^{\frac{2(r+1)}{r-1}}}\right)\|\y^n-\y\|_{\widetilde{\L}^{r+1}}\|\nabla\z\|_{\H}\nonumber\\&\leq \left(\|\y^n\|_{\H}^{\frac{r-3}{r-1}}\|\y^n\|_{\widetilde{\L}^{r+1}}^{\frac{2}{r-1}}+\|\y\|_{\H}^{\frac{r-3}{r-1}}\|\y\|_{\widetilde{\L}^{r+1}}^{\frac{2}{r-1}}\right)\|\y^n-\y\|_{\widetilde{\L}^{r+1}}\|\nabla\z\|_{\H}\nonumber\\& \to 0, \ \text{ as } \ n\to\infty, 
\end{align*}
since $\y^n\to\y$ strongly in $\V\cap\wi\L^{r+1}$ and $\y^n,\y\in\V\cap\wi\L^{r+1}$. We estimate the term $|\langle \mathcal{C}(\y^n)-\mathcal{C}(\y),\z\rangle|$ using the Taylor's formula (\cite[Theorem 7.9.1]{PGC}) as 
\begin{align*}
	|\langle \mathcal{C}(\y^n)-\mathcal{C}(\y),\z\rangle|&\leq \sup_{0<\theta<1} r\|(\y^n-\y)|\theta\y^n+(1-\theta)\y|^{r-1}\|_{\widetilde{\L}^{\frac{r+1}{r}}}\|\z\|_{\widetilde{\L}^{r+1}}\nonumber\\&\leq r\|\y^n-\y\|_{\widetilde{\L}^{r+1}}\left(\|\y^n\|_{\widetilde{\L}^{r+1}}+\|\y\|_{\widetilde{\L}^{r+1}}\right)^{r-1}\|\z\|_{\widetilde{\L}^{r+1}}\to 0 \text{ as } n\to\infty,
\end{align*}
since $\y_n\to\y$ strongly in $\V\cap{\widetilde{\L}^{r+1}}$ and $\y_n, \y\in\V\cap{\widetilde{\L}^{r+1}}$. From the above convergences, it is immediate that $\langle(\mathcal{G}+\kappa\I)(\y^n)-(\mathcal{G}+\kappa\I)(\y),\z\rangle \to 0$, for all $\z\in \V\cap\widetilde{\L}^{r+1}$.
Therefore, the operator $\mathcal{G}+\kappa\I:\V\cap\widetilde{\L}^{r+1}\to \V'+\widetilde{\L}^{\frac{r+1}{r}}$ is demicontinuous and hence it is hemicontinuous. 
\vskip 2mm
\noindent
\textbf{Step III:} \textsl{$\mathcal{G}+\kappa\I$ is coercive.} We consider
\begin{align*}
	\frac{\langle(\mathcal{G}+\kappa\I)(\y),\y\rangle}{\|\y\|_{\V\cap\wi\L^{r+1}}}&=\frac{\mu\|\nabla\y\|_{\H}^{2}+\beta\|\y\|_{\wi\L^{r+1}}^{r+1}+\kappa\|\y\|_{\H}^{2}}{\sqrt{\|\y\|_{\V}^{2}+\|\y\|_{\wi\L^{r+1}}^2}}\nonumber\geq\frac{\min\{\mu,\beta,\kappa\}\left(\|\y\|_{\V}^{2}+\|\y\|^2_{\wi\L^{r+1}}-1\right)}{\sqrt{\|\y\|_{\V}^{2}+\|\y\|_{\wi\L^{r+1}}^2}},
\end{align*}
where we have used the fact that $x^{2}\leq x^{r+1}+1,$ for $x\geq0$ and $r\geq1.$ Thus, we have 
\begin{align*}
	\lim\limits_{\|\y\|_{\V\cap\wi\L^{r+1}}\to\infty} \frac{\langle(\mathcal{G}+\kappa\I)(\y),\y\rangle}{\|\y\|_{\V\cap\wi\L^{r+1}}}=\infty, 
\end{align*}
and it  shows that the operator $\mathcal{G}+\kappa\I$ is coercive.
\vskip 2mm
\noindent
\textbf{Step IV:} \textsl{$\mathcal{F}(\cdot):=\mathcal{G}(\cdot)+\kappa\I$ is $m$-accretive in $\H\times\H$}. Let us define an operator 
$$\mathcal{F}(\y):=\mu\A\y+\mathcal{B}(\y)+\beta\mathcal{C}(\y)+\kappa\y,$$ where $\D(\mathcal{F})=\{\y\in\V\cap\wi\L^{r+1}:\mu\A\y+\mathcal{B}(\y)+\beta\mathcal{C}(\y)\in\H\}.$ Note that the space $\V\cap\widetilde{\L}^{r+1}$ is reflexive. Since $\mathcal{G}+\kappa\I$ is monotone, hemicontinuous and coercive from $\V\cap\widetilde{\L}^{r+1}$ to $\V'+\widetilde{\L}^{\frac{r+1}{r}}$, then by an application of \cite[Example 2.3.7]{OPHB}, we obtain that  $\mathcal{G}+\kappa\I$ is maximal monotone in $\H$ with domain $\mathrm{D}(\mathcal{F})\supseteq\mathrm{D}(\mathrm{A}).$ In fact, we shall prove that $\mathcal{F}$ is $m$-accretive for $\kappa$ sufficiently large with $\mathrm{D}(\mathcal{F})=\mathrm{D}(\mathrm{A}).$ Let us consider the operators for some $\delta_1, \delta_2\in(0,1)$ as
\begin{align}
	\mathcal{F}^{1}(\cdot) &= \mu(1-\delta_1)\mathrm{A}+\beta(1-\delta_2)\mathcal{C}(\cdot),\label{3.3.2}\\
	\mathcal{F}^{2}(\cdot) &= \mu\delta_1\mathrm{A}+\mathcal{B}(\cdot)+\beta\delta_2\mathcal{C}(\cdot)+\kappa\mathrm{I},\label{3.3.33}
\end{align}
where $\D(\mathcal{F}^{1})=\{\y\in\V\cap\wi\L^{r+1}:\mathcal{F}^1(\cdot)\in\H\}$ and $\D(\mathcal{F}^{2}) = \{\y\in\V\cap\wi\L^{r+1}:\mathcal{F}^{2}(\cdot)\in\H\}.$ 
Taking the inner product with $\y$ in \eqref{3.3.2}, we obtain
\begin{align*}
	\mu(1-\delta_1)\|\nabla\y\|_{\H}^2+\beta(1-\delta_2)\|\y\|_{\wi\L^{r+1}}^{r+1}\leq(\mathcal{F}^{1}(\y),\y)\leq \|\mathcal{F}^{1}(\y)\|_{\H}\|\y\|_{\H},
\end{align*}
so that 
\begin{align}\label{3.3.5}
	\|\nabla\y\|_{\H}^2\leq\frac{1}{\mu(1-\delta_1)}\|\mathcal{F}^{1}(\y)\|_{\H}\|\y\|_{\H}.
\end{align}
	Taking the inner product with $\A\y$ in \eqref{3.3.2} and using equality \eqref{3}, we get
\begin{align*}
	\mu(1-\delta_1)\|\A\y\|_{\H}^{2}+\beta(1-\delta_2) \left[\||\y|^{\frac{r-1}{2}}\nabla\y\|_{\H}^2+\frac{4(r-1)}{(r+1)^2}\||\nabla|\y|^{\frac{r+1}{2}}|\|_{\H}^2\right]=(\mathcal{F}^{1}(\y),\A\y). 
\end{align*}
Therefore, we have 
\begin{align}\label{3.147}
	\|\A\y\|_{\H}\leq\frac{1}{\mu(1-\delta_1)}\|\mathcal{F}^{1}(\y)\|_{\H}\ \text{ which implies }\ \D(\mathcal{F}^{1})\subseteq\D(\A). 
\end{align}
Moreover, by using Sobolev embedding $\H^2_{\mathrm{p}}(\mathbb{T}^d)\hookrightarrow\wi\L^{2r}$, we infer  
\begin{align*}
	\|\mathcal{F}^{1}(\y)\|_{\H}&\leq\mu(1-\delta_1)\|\mathrm{A}\y\|_{\H} +C\beta(1-\delta_2)\|\y\|_{\wi\L^{2r}}^r\nonumber\\&\leq\mu(1-\delta_1)\|\mathrm{A}\y\|_{\H}+C \beta(1-\delta_2)\|\y\|_{\H^2_{\mathrm{p}}(\mathbb{T}^d)}^r\nonumber\\&\leq\mu(1-\delta_1) \|\mathrm{A}\y\|_{\H}+C\beta(1-\delta_2)\|\A\y\|_{\H}^r+C\beta(1-\delta_2)\|\y\|_{\H}^r, 
\end{align*}
which gives $\D(\mathcal{F}^{1})\supseteq\D(\A)$ and therefore $\D(\A)=\D(\mathcal{F}^1).$
Similarly, taking the inner product with $\mathcal{C}(\y)$ in \eqref{3.3.2}, we find
\begin{align}\label{3.149}
	\|\mathcal{C}(\y)\|_{\H}\leq\frac{1}{\beta(1-\delta_2)}\|\mathcal{F}^{1}(\y)\|_{\H}.
\end{align}
For $r>3$, similar to \eqref{2..29}, we estimate $\|\mathcal{B}(\y)\|_{\H}$ using H\"older's inequality as follows:
\begin{align}\label{3.3.4}
	\|\mathcal{B}(\y)\|_{\H}^2 \leq
\||\y|^\frac{r-1}{2}\nabla\y\|_{\H}^{\frac{4}{r-1}}\|\nabla\y\|_{\H}^\frac{2(r-3)}{r-1}.  
\end{align}
Note that $(\mathcal{C}(\y),\A\y)=\int_{\mathbb{T}^{d}}(-\Delta \boldsymbol{y}(x))\cdot|\boldsymbol{y}(x)|^{r-1}\boldsymbol{y}(x)\d x$. Using the estimate \eqref{3.3.5} and the equality \eqref{3} in \eqref{3.3.4}, we find
\begin{align*}
    \|\mathcal{B}(\y)\|_{\H}^2\leq\left[(\mathcal{C}(\y),\A\y)\right]^\frac{2}{r-1}\left[\frac{1}{\mu(1-\delta_1)}\|\mathcal{F}^{1}(\y)\|_{\H}\|\y\|_{\H}\right]^\frac{r-3}{r-1}.
\end{align*}
Therefore, we estimate $\|\mathcal{B}(\y)\|_{\H}$ as 
\begin{align}\label{3p36}
	 \|\mathcal{B}(\y)\|_{\H}\leq\|\mathcal{C}(\y)\|_{\H}^\frac{1}{r-1}\|\A\y\|_{\H}^\frac{1}{r-1}\left[\frac{1}{\mu(1-\delta_1)}\|\mathcal{F}^{1}(\y)\|_{\H}\|\y\|_{\H}\right]^\frac{r-3}{2(r-1)}.
\end{align}
Using the estimates \eqref{3.147}-\eqref{3.149} in \eqref{3p36},  then using Young's inequality, we get
\begin{align}\label{3.3.6}
 \|\mathcal{B}(\y)\|_{\H}&\leq\left[\frac{\|\mathcal{F}^1(\y)\|_{\H}^2}{\beta\mu(1-\delta_1)(1-\delta_2)}\right]^\frac{1}{r-1}\left[\frac{\|\mathcal{F}^{1}(\y)\|_{\H}\|\y\|_{\H}}{\mu(1-\delta_1)}\right]^\frac{r-3}{2(r-1)}\nonumber\\&=
 \frac{1}{\sqrt{\mu(1-\delta_1)}}\left[\frac{1}{\beta(1-\delta_2)}\right]^\frac{1}{r-1}\|\mathcal{F}^1(\y)\|_{\H}^\frac{r+1}{2(r-1)}\|\y\|_{\H}^\frac{r-3}{2(r-1)}\nonumber\\&\leq \frac{\delta_1}{1-\delta_1}\|\mathcal{F}^{1}(\y)\|_{\H}+C_{\delta_1,\delta_2,\mu,\beta}\|\y\|_{\H}, 
\end{align}
where  $C_{\delta_1,\delta_2,\mu,\beta}=\frac{r-3}{2(r-1)}\left(\frac{1-\delta_1}{\mu^{\frac{r-1}{2}}\beta(1-\delta_2)}\right)^{\frac{2}{r-3}}\left(\frac{r+1}{2\delta_1(r-1)}\right)^{\frac{r+1}{r-3}}$.
Now using the estimates \eqref{3.147}-\eqref{3.149} and \eqref{3.3.6} in \eqref{3.3.33}, we deduce 
\begin{align*}
	\|\mathcal{F}^{2}(\y)\|_{\H}&\leq\mu\delta_1\|\A\y\|_{\H}+\beta\delta_2 \|\mathcal{C}(\y)\|_{\H}+\|\mathcal{B}(\y)\|_{\H}+\kappa\|\y\|_{\H}\nonumber\\&\leq	
	\left[\frac{2\delta_1}{1-\delta_1}+\frac{\delta_2}{1-\delta_2}\right]\|\mathcal{F}^1(\y)\|_{\H}+(C_{\delta_1,\delta_2,\mu,\beta}+\kappa)\|\y\|_{\H}. 
\end{align*}
Let us choose $\delta_1$ and $\delta_2$ in such a way that $\rho=\frac{2\delta_1}{1-\delta_1}+\frac{\delta_2}{1-\delta_2}<1,$ for example, one can choose $\delta_1=\frac{1}{9}$, $\delta_2=\frac{1}{5}$, so that $\rho=\frac{1}{2}.$ Then by the well-known  perturbation theorem for nonlinear $m$-accretive operators (\cite[Chapter II, Theorem 3.5]{VB1}), we conclude that the operator $\mathcal{F}^{1}+\mathcal{F}^{2}$  with the domain $\D(\A)$  is $m$-accretive in $\H$. Since  $\mathcal{F}^{1}+\mathcal{F}^{2}=\mathcal{G}+\kappa\mathrm{I}$ , the operator  $\mathcal{G}+\kappa\mathrm{I}$ is $m$-accretive in $\H$. 
\end{proof}

\begin{remark}
	1. For $d=2$ with $r\in[1,\infty)$ and $d=3$ with $r\in[1,5]$, Sobolev's embedding yields $\V\subset\wi\L^{r+1}$, so that $\V\cap\wi\L^{r+1}=\V$. 
	
	2. For $d=r=3$ and $2\beta\mu\geq 1$, one can obtain global monotonicity of the operator $\mathcal{G}(\cdot):\V\to\V'$ in the following way:
	
	We estimate $|\langle\mathcal{B}(\y-\z,\y-\z),\z\rangle|$ using H\"older's and Young's inequalities as 
	\begin{align}\label{ae.}
		|\langle\mathcal{B}(\y-\z,\y-\z),\z\rangle|\leq\|\z(\y-\z)\|_{\H}\|\nabla(\y-\z)\|_{\H} \leq\mu \|\nabla(\y-\z)\|_{\H}^2+\frac{1}{4\mu }\|\z(\y-\z)\|_{\H}^2.
	\end{align} 
Combining \eqref{ae}, \eqref{2.27} and \eqref{ae.}, we obtain 
\begin{align}\label{gm}
	\langle\mathcal{G}(\y)-\mathcal{G}(\z),\y-\z\rangle\geq\frac{1}{2}\left(\beta-\frac{1}{2\mu }\right)\|\z(\y-\z)\|_{\H}^2\geq 0,
\end{align}
provided $2\beta\mu \geq 1$. Moreover, other properties like demicontinuity and coercivity can be proved in similar way as $r>3$ case (see the proof of Proposition \ref{prop33}).
\end{remark}


\begin{proposition}\label{prop3.3}
	Let $\Phi\subset\H\times\H$ be a maximal monotone operator satisfying Hypothesis \ref{hyp1}. Define the multi-valued operator  $\mathfrak{A}:\mathrm{D}(\mathfrak{A})\to\H$ by 
	\begin{equation*}
		\mathfrak{A}(\cdot) = \mu\mathrm{A} +\mathcal{B}(\cdot)+\beta\mathcal{C}(\cdot)+\Phi(\cdot)+\kappa\I,
	\end{equation*}
	with the domain $\mathrm{D}(\mathfrak{A})=\{\y\in\H: \varnothing \neq \mathfrak{A}(\y)\subset\H\}$. Then $\mathrm{D}(\mathfrak{A}) = \mathrm{D}(\mathrm{A})\cap\mathrm{D}(\Phi)$ and $\mathfrak{A}$ is a maximal monotone operator in $\H\times\H,$ where $\kappa$ is as in Proposition \ref{prop33}. 
	
Furthermore, the following estimates holds 	
\begin{align}\label{3.3838}
	\|\A\w\|_{\H}^2\leq C(1+\|\w\|_{\H}^2+\|\mu\mathrm{A}\w+\mathcal{B}(\w)+\beta\mathcal{C}(\w)+\Phi_\lambda(\w)\|_{\H}^2)^{\vartheta},
\end{align}
for every $\w\in\D(\A), \lambda>0$ and 
\begin{align}\label{3.3939}
\|\A\w\|_{\H}^2\leq C(1+\|\w\|_{\H}^2+\|\mu\mathrm{A}\w+\mathcal{B}(\w)+\beta\mathcal{C}(\w)+\xi\|_{\H}^2)^{\vartheta},	
\end{align}
for every $\w\in\D(\A)\cap\D(\Phi)$ and $\xi\in\Phi(\w),$ where 
\begin{align}\label{vartheta}
	\vartheta=
	\begin{cases}
		r,  &\text{when} \ d=2 \  \text{with} \ r\in(3,\infty),\\
		\frac{r+3}{5-r}, &\text{when} \ d=3 \ \text{with} \ r\in(3,5),\\
			3,  &\text{when} \ d=r=3 \  \text{with} \ 2\beta\mu\geq1,\\
		1,  &\text{when} \ d=3 \ \text{with} \ r\in[5,\infty).
	\end{cases}
\end{align}

\end{proposition}
\begin{proof}

	It has been shown in Proposition \ref{prop33} that  the operator $\mathcal{F}(\cdot)=\mu\mathrm{A} +\mathcal{B}(\cdot)+\beta\mathcal{C}(\cdot)+\kappa\mathrm{I}$ is maximal monotone with domain $\mathrm{D}(\mathcal{F})=\mathrm{D}(\mathrm{A})$ in $\H\times\H.$ Note that  $\mathfrak{A}=\mathcal{F}+\Phi$ implies   $\mathrm{D}(\mathrm{A})\cap\mathrm{D}(\Phi)\subseteq\mathrm{D}(\mathfrak{A})$ and since  $\mathfrak{A}$  is the sum of two monotone operators, it is monotone. In order to prove $\mathfrak{A}$ is maximal monotone, we need to show that 
	\begin{align}\label{316}
		\mathrm{R}(\mathrm{I}+\mathfrak{A})=\H.
	\end{align} 
	\vskip 0.2 cm
	\noindent 
	\textbf{Step I:} \textsl{Well-posedness of the Yosida approximated problem.} 	Let $\f\in\H$ be arbitrary but fixed. We approximate the inclusion problem 
	\begin{equation}
		\y+\mu\mathrm{A}\y+\mathcal{B}(\y)+\beta\mathcal{C}(\y)+\Phi(\y)+\kappa\y\ni \f,
	\end{equation}
	by the equation 
	\begin{equation}\label{Y1}
		\y_{\lambda}+\mu\mathrm{A}\y_{\lambda}+\mathcal{B}(\y_{\lambda})+\beta\mathcal{C}(\y_{\lambda})+\Phi_{\lambda}(\y_{\lambda})+\kappa\y_{\lambda}=\f,
	\end{equation}
where $\Phi_{\lambda}$  is the Yosida approximation of $\Phi$. 
	By the properties of Yosida approximation, $\Phi_{\lambda}$ is demicontinuous and monotone (see \cite[Chapter 2, Proposition 1.3]{VB1}). Therefore the sum $\mathcal{F}(\cdot)+\Phi_{\lambda}(\cdot)$ is maximal monotone (see \cite[Chapter 2, Corollary 1.1]{VB2}). This guarantees the existence of a solution $\y_{\lambda}\in\mathrm{D}(\mathrm{A})$ for \eqref{Y1}.
	Let $\widetilde{\kappa}=\kappa+1. $ Then \eqref{Y1} can be written as 
	\begin{equation}\label{Y2}
		\mu\mathrm{A}\y_{\lambda}+\mathcal{B}(\y_{\lambda})+\beta\mathcal{C}(\y_{\lambda})+\Phi_{\lambda}(\y_{\lambda})+\widetilde{\kappa}\y_{\lambda}=\f. 	
	\end{equation} 
We shall now prove the uniqueness. Let $\y_{\lambda}$ and $\z_{\lambda}$ be two solutions of the equation \eqref{Y2} with the same data $\f$ and let $\w_{\lambda}=\y_{\lambda}-\z_{\lambda}.$ Then we have 
\begin{equation}\label{Z1}
	\mu\mathrm{A}\w_{\lambda}+\mathcal{B}(\y_{\lambda})-\mathcal{B}(\z_{\lambda})+\beta(\mathcal{C}(\y_{\lambda})-\mathcal{C}(\z_{\lambda}))+\Phi_{\lambda}(\y_{\lambda})-\Phi_{\lambda}(\z_{\lambda})+\widetilde{\kappa}\w_{\lambda}=\mathbf{0}.
\end{equation}   	
Taking the inner product with $\w_{\lambda}$ in \eqref{Z1}, we get
\begin{align}\label{Z6}
	\mu\|\nabla\w_{\lambda}\|_{\H}^{2}&+(\Phi_{\lambda}(\y_{\lambda})-\Phi_{\lambda}(\z_{\lambda}),\w_{\lambda})+\widetilde{\kappa}\|\w_{\lambda}\|_{\H}^{2}\nonumber\\&=-(\mathcal{B}(\y_{\lambda})-\mathcal{B}(\z_{\lambda}),\w_{\lambda})-\beta(\mathcal{C}(\y_{\lambda})-\mathcal{C}(\z_{\lambda}),\w_{\lambda}).
\end{align}
By similar calculations as in \eqref{2.30} and \eqref{2.23}, we obtain 
\begin{align}\label{Z6..}
	-(\mathcal{B}(\y_{\lambda})-\mathcal{B}(\z_{\lambda})-\beta(\mathcal{C}(\y_{\lambda})+\mathcal{C}(\z_{\lambda})),\w_{\lambda})\leq
\frac{\mu}{2}\|\nabla\w_{\lambda}\|_{\H}^{2}+\varrho\|\w_{\lambda}\|_{\H}^{2},
\end{align}
 By \cite[Chapter 2, Proposition 1.3, part (i)]{VB2},  we know that $\Phi_{\lambda}$ is  monotone, so that $(\Phi_{\lambda}(\y_{\lambda})-\Phi_{\lambda}(\z_{\lambda}),\w_{\lambda})\geq 0$ for any $\lambda>0$. Therefore, we conclude from \eqref{Z6} that 
\begin{align*}
	\frac{\mu}{2}\|\nabla\w_{\lambda}\|_{\H}^{2}+\left(\widetilde{\kappa}-\varrho\right)\|\w_{\lambda}\|_{\H}^{2}\leq 0.
\end{align*}
Since $\varrho<\widetilde{\kappa}$, we get $\w_{\lambda}=\boldsymbol{0}$ and thus $\y_{\lambda}=\z_{\lambda}.$ 
\vskip 0.2 cm
\noindent 
	\textbf{Step II:} \textsl{Uniform bounds for $\y_{\lambda}$.}  Let us   take the inner product with $\y_{\lambda}$ in \eqref{Y2} to get 
	\begin{equation}\label{P1}
		\mu\|\nabla\y_{\lambda}\|_{\H}^{2} +\beta(\mathcal{C}(\y_{\lambda}),\y_{\lambda})+(\Phi_{\lambda}(\y_{\lambda}),\y_{\lambda})+\widetilde{\kappa}\|\y_{\lambda}\|_{\H}^{2} = (\f,\y_{\lambda}),
	\end{equation}
since $(\mathcal{B}(\y_{\lambda}),\y_{\lambda})=0$. 	As the operator $\Phi_{\lambda}$ is monotone with $\boldsymbol{0}\in\mathrm{D}(\Phi_{\lambda})=\H$,  we infer 
	\begin{equation}
		(\Phi_{\lambda}(\y_{\lambda}),\y_{\lambda})\geq(\Phi_{\lambda}(\boldsymbol{0}),\y_{\lambda}).
	\end{equation}
	By applying Young's inequality and by \cite[Chapter 2, Proposition 1.3, part (ii)]{VB2}, we have
	\begin{equation}\label{M1}
		-(\Phi_{\lambda}(\boldsymbol{0}),\y_{\lambda})\leq\|\Phi_{\lambda}(\boldsymbol{0})\|_{\H}\|\y_{\lambda}\|_{\H}\leq\frac{1}{\widetilde{\kappa}}\|\Phi(\boldsymbol{0})\|_{\H}^{2}+\frac{\widetilde{\kappa}}{4}\|\y_{\lambda}\|_{\H}^{2}.	
	\end{equation}
	Then equation \eqref{P1} yields 
	\begin{equation}\label{e1.2}
		\|\y_{\lambda}\|_{\H}^{2}+ \|\nabla\y_{\lambda}\|_{\H}^{2} +	\|\y_{\lambda}\|_{\widetilde{\L}^{r+1}}^{r+1} \leq C(1+\|\f\|_{\H}^{2}), \  \text{ for all } \ \lambda>0,
	\end{equation}
	where the constant $C=C(\mu,\beta,\widetilde{\kappa},\|\Phi(\boldsymbol{0})\|_{\H})$ does not depend on $\lambda$. Taking the  inner product of \eqref{Y2} with $\mathrm{A\y_{\lambda}},$ we get
	\begin{equation}\label{e1.3}
		\mu\|\mathrm{A}\y_{\lambda}\|_{\H}^{2} +(\mathcal{B}(\y_{\lambda}),\mathrm{A}\y_{\lambda})+\beta(\mathcal{C}(\y_{\lambda}),\mathrm{A}\y_{\lambda})+(\Phi_{\lambda}(\y_{\lambda}),\mathrm{A}\y_{\lambda})+\widetilde{\kappa}\|\nabla\y_{\lambda}\|_{\H}^{2} = (\f,\mathrm{A}\y_{\lambda}).
	\end{equation}
By \cite[ {Lemma 3.1}, pp. 404]{RT1}, we have $(\mathcal{B}(\y_{\lambda}),\mathrm{A}\y_{\lambda}) =0$ for $d=2.$
For $d=3$, we consider the cases $r>3$ and $r=3$ with $2\beta\mu\geq1$ separately. 
	\vskip 2mm
	\noindent
	\textbf{Case I:} \textsl{$r>3.$}
	From Cauchy-Schwarz and Young's inequalities, we obtain
	\begin{align}\label{f1}
		(\f,\mathrm{A}\y_{\lambda})\leq\|\f\|_{\H}\|\A\y\|_{\H}\leq\frac{\mu}{4}\|\A\y\|_{\H}^2+\frac{1}{\mu}\|\f\|_{\H}^2.
	\end{align}
	We estimate $|(\mathcal{B}(\y_{\lambda}),\A\y_{\lambda})|$ using H\"older's and Young's inequalities as  (see \eqref{2..29} for more details)
	\begin{align}\label{373}
		|(\mathcal{B}(\y_{\lambda}),\A\y_{\lambda})|&\leq\||\y_{\lambda}||\nabla\y_{\lambda}|\|_{\H}\|\A\y_{\lambda}\|_{\H}\nonumber\\&\leq\frac{\mu}{2}\|\A\y_{\lambda}\|_{\H}^2+\frac{\beta}{2}\||\y_{\lambda}|^{\frac{r-1}{2}}|\nabla\y_{\lambda}|\|_{\H}^{2}+\varrho\|\nabla\y_\lambda\|_{\H}^2
	\end{align}
	Using the condition (H.3) of Hypothesis \ref{hyp1},  estimates \eqref{e1.2}, \eqref{f1}-\eqref{373} and equality \eqref{3} in  \eqref{e1.3}, it yields for all $\lambda>0$ 
	\begin{align}\label{phi1}
	&	\frac{\mu}{4}\|\mathrm{A\y_{\lambda}}\|_{\H}^{2}+\frac{\beta}{2}\||\y_{\lambda}|^{\frac{r-1}{2}}|\nabla\y_{\lambda}|\|_{\H}^{2}+4\beta\left[\frac{r-1}{(r+1)^2}\right]\||\nabla|\y_{\lambda}|^{\frac{r+1}{2}}|\|_{\H}^{2}\nonumber\\&\leq C(1+\|\f\|_{\H}^{2})+
			\begin{cases}
				\varsigma\|\Phi_{\lambda}(\y_{\lambda})\|_{\H}^{2}, &\text{ for } d=2 \text{ with } r\in(3,\infty)\text{ and } d=3 \text{ with } r\in(3,5),\\
				0,   &\text{ for } d=3 \text{ with } r\in[5,\infty).
			\end{cases}
	\end{align}
Note that we complete the proof of energy estimates for $d=3$ with $r\in[5,\infty).$ But, for the other cases, we need the uniform bound of $\varsigma\|\Phi_{\lambda}(\y_{\lambda})\|^2_{\H}$ which will be shown in the next step.
	\vskip 2mm
\noindent
\textbf{Case II:} \textsl{$r=3$ with $2\beta\mu\geq1.$}
From the equality \eqref{3}, and Cauchy-Schwarz and Young's inequalities, one can obtain
\begin{align}\label{r3e}
\frac{3\mu}{8}\|\A\y_\lambda\|_{\H}^2+\left(\beta-\frac{1}{2\mu}\right)\||\y_\lambda||\nabla\y_\lambda|\|_{\H}^2+\frac{\beta}{2}\||\nabla|\y_{\lambda}|^2|\|_{\H}^{2}\leq C(1+\|\f\|_{\H}^{2})+\varsigma\|\Phi_{\lambda}(\y_{\lambda})\|_{\H}^{2}. 
\end{align}

\vskip 2mm
\noindent
\textbf{Step III:} \textsl{Uniform bounds for  $\|\Phi_{\lambda}(\y_{\lambda})\|_{\H}^{2}.$}
Taking the inner product with $\Phi_{\lambda}(\y_{\lambda})$ in \eqref{Y2}, we have
	\begin{align}\label{329}
		\|\Phi_{\lambda}(\y_{\lambda})\|_{\H}^{2}&=(\f,\Phi_{\lambda}(\y_{\lambda}))-(\mathcal{B}(\y_{\lambda}),\Phi_{\lambda}(\y_{\lambda}))-\beta(\mathcal{C}(\y_{\lambda}),\Phi_{\lambda}(\y_{\lambda}))-\mu(\mathrm{A\y_{\lambda}},\Phi_{\lambda}(\y_{\lambda}))\nonumber\\&\quad-\widetilde{\kappa}(\y_{\lambda},\Phi_{\lambda}(\y_{\lambda})) . 
	\end{align} Similar to \eqref{M1}, we have 
	\begin{align}\label{m}
		(\y_{\lambda},\Phi_{\lambda}(\y_{\lambda}))\geq-\frac{1}{2}(\|\Phi(\boldsymbol{0})\|_{\H}^{2}+\|\y_{\lambda}\|_{\H}^{2}).
	\end{align}
	We calculate $|(\mathcal{B}(\y_{\lambda}),\Phi_{\lambda}(\y_{\lambda})|$ using  \eqref{e1.2}, Agmon's and Young's inequalities  as 
	\begin{align}\label{e1.4}
			|(\mathcal{B}(\y_{\lambda}),\Phi_{\lambda}(\y_{\lambda})|&=|b(\y_{\lambda},\y_{\lambda},\Phi_{\lambda}(\y_{\lambda}))|\nonumber\\&\leq C	\|\y_{\lambda}\|_{\H}^{1-\frac{d}{4}}\|\nabla\y_{\lambda}\|_{\H}\|\y_{\lambda}\|_{\H^2_\mathrm{p}}^{\frac{d}{4}}\|\Phi_{\lambda}(\y_{\lambda})\|_{\H}\nonumber\\&\leq C
			\|\y_{\lambda}\|_{\H}^{1-\frac{d}{4}}\|\nabla\y_{\lambda}\|_{\H}(\|\A\y_{\lambda}\|_{\H}^{\frac{d}{4}}+\|\y_\lambda\|_{\H}^{\frac{d}{4}})\|\Phi_{\lambda}(\y_{\lambda})\|_{\H}\nonumber\\&= C
			\|\y_{\lambda}\|_{\H}^{1-\frac{d}{4}}\|\nabla\y_{\lambda}\|_{\H}\|\A\y_{\lambda}\|_{\H}^{\frac{d}{4}}\|\Phi_{\lambda}(\y_{\lambda})\|_{\H}+\|\y_\lambda\|_{\H}\|\nabla\y_{\lambda}\|_{\H}\|\Phi_{\lambda}(\y_{\lambda})\|_{\H}\nonumber\\&\leq
			\frac{1-\mu\varsigma}{8}\|\Phi_{\lambda}(\y_{\lambda})\|_{\H}^{2}+\frac{\mu(1-\mu\varsigma)}{8\varsigma}\|\mathrm{A\y_{\lambda}}\|_{\H}^{2}+C(1+\|\f\|_{\H}^2)^3.
	\end{align}
By using the Cauchy-Schwarz, interpolation and Young's inequalities, \eqref{e1.2} and \eqref{phi1}, we obtain
\begin{align}\label{3.5555}
&|(\mathcal{C}(\y_{\lambda}),\Phi_{\lambda}(\y_{\lambda}))|\nonumber\\&\leq\|\mathcal{C}(\y_{\lambda})\|_{\H}\|\Phi_{\lambda}(\y_{\lambda})\|_{\H}\leq\|\y_{\lambda}\|_{\wi\L^{2r}}^r\|\Phi_{\lambda}(\y_{\lambda})\|_{\H}\nonumber\\&\leq
\|\y_{\lambda}\|_{\widetilde{\L}^{r+1}}^{\frac{r+3}{4}}\|\y_{\lambda}\|_{\widetilde{\L}^{3(r+1)}}^{\frac{3(r-1)}{4}}\|\Phi_{\lambda}(\y_{\lambda})\|_{\H}\nonumber\\&\leq
C\|\y_{\lambda}\|_{\widetilde{\L}^{r+1}}^{\frac{r+3}{4}}\||\y_{\lambda}|^\frac{r-1}{2}\nabla\y_{\lambda}\|_{\H}^\frac{3(r-1)}{2(r+1)}\|\Phi_{\lambda}(\y_{\lambda})\|_{\H}\nonumber\\&\leq
C(1+\|\f\|_{\H}^2)^\frac{r+3}{4(r+1)}\left[\frac{2\varsigma}{\beta}\|\Phi_{\lambda}(\y_{\lambda})\|_{\H}^{2}+\frac{2C}{\beta}(1+\|\f\|_{\H}^{2})\right]^\frac{3(r-1)}{4(r+1)}\|\Phi_{\lambda}(\y_{\lambda})\|_{\H} 
\nonumber\\&\leq C \left[\|\Phi_{\lambda}(\y_{\lambda})\|_{\H}^{\frac{5r-1}{2(r+1)}}(1+\|\f\|_{\H}^2)^\frac{r+3}{4(r+1)}+\|\Phi_{\lambda}(\y_{\lambda})\|_{\H}(1+\|\f\|_{\H}^2)^\frac{r}{r+1}\right]\nonumber\\&\leq
\frac{1-\mu\varsigma}{8\beta}\|\Phi_{\lambda}(\y_{\lambda})\|_{\H}^{2}+C(1+\|\f\|_{\H}^2)^\frac{r+3}{5-r} +C(1+\|\f\|_{\H}^2)^\frac{2r}{r+1}\nonumber\\&\leq
\frac{1-\mu\varsigma}{8\beta}\|\Phi_{\lambda}(\y_{\lambda})\|_{\H}^{2}+C(1+\|\f\|_{\H}^2)^\frac{r+3}{5-r}, \qquad\quad\ \ \ \mbox{for $d=3$ with $r\in(3,5)$,}
\end{align}
where we have used the fact  that $\frac{r+3}{5-r}>\frac{2r}{r+1}$.  Using the Sobolev embedding, we deduce 
\begin{align}\label{3.6666}
&|(\mathcal{C}(\y_{\lambda}),\Phi_{\lambda}(\y_{\lambda}))|\nonumber\\&\leq\|\mathcal{C}(\y_{\lambda})\|_{\H}\|\Phi_{\lambda}(\y_{\lambda})\|_{\H}\leq\|\y_{\lambda}\|_{\wi\L^{2r}}^r\|\Phi_{\lambda}(\y_{\lambda})\|_{\H}\nonumber\\&\leq 
\|\y_\lambda\|_{\V}^r\|\Phi_{\lambda}(\y_{\lambda})\|_{\H}\leq C(1+\|\f\|_{\H}^2)^\frac{r}{2}\|\Phi_{\lambda}(\y_{\lambda})\|_{\H}\nonumber\\&\leq
\frac{1-\mu\varsigma}{8\beta}\|\Phi_{\lambda}(\y_{\lambda})\|_{\H}^{2}+C(1+\|\f\|_{\H}^2)^r, \qquad\quad\ \ \ \mbox{for $d=2$ with $r\in(3,\infty)$.}
\end{align}
 Also, by the Cauchy-Schwarz and Young's inequalities, we get 
\begin{align}\label{3.7777}
	|(\f,\Phi_{\lambda}(\y_{\lambda}))|\leq\frac{1}{1-\mu\varsigma}\|\f\|_{\H}^{2}+\frac{1-\mu\varsigma}{4}\|\Phi_{\lambda}(\y_{\lambda})\|_{\H}^{2}. 
\end{align}
Using the estimates \eqref{e1.2} and \eqref{m}-\eqref{3.7777} in  \eqref{329}, we arrive at

\begin{align}\label{3.5959}
\varsigma\|\Phi_{\lambda}(\y_{\lambda})\|_{\H}^{2}\leq\frac{\mu}{4}\|\A\y\|_{\H}^2+
\begin{cases}
	C(1+\|\f\|_{\H}^2)^r, \  &\text{for} \  d=2 \text{ with } r\in(3,\infty),\\
	C(1+\|\f\|_{\H}^2)^\frac{r+3}{5-r}, \  &\text{for} \  d=3 \text{ with } r\in(3,5),\\
	C(1+\|\f\|_{\H}^2)^3, \  &\text{for} \  d=r=3.
\end{cases}
\end{align}
\textsl{Uniform boundedness of sequences.} It implies from \eqref{phi1} and \eqref{3.5959} that
\begin{align}\label{3.6060}
		\frac{\mu}{4}\|\mathrm{A\y_{\lambda}}\|_{\H}^{2}+\frac{\beta}{2}\||\nabla \y_{\lambda}||\y_{\lambda}|^{\frac{r-1}{2}}\|_{\H}^{2}\leq
		\begin{cases}
				C(1+\|\f\|_{\H}^2)^r, \  &\text{for} \  d=2 \ \text{with} \ r\in(3,\infty),\\		C(1+\|\f\|_{\H}^2)^\frac{r+3}{5-r}, \  &\text{for} \  d=3 \ \text{with} \ r\in(3,5),\\
				C(1+\|\f\|_{\H}^{2}), \  &\text{for} \  d=3 \ \text{with} \ r\in(5,\infty).\\
		\end{cases}
\end{align}
 For $d=r=3$ with $2\beta\mu\geq1$, using \eqref{3.5959} in \eqref{r3e}, we obtain
 	\begin{align}\label{r3e.}
 		&\frac{\mu}{8}\|\A\y_\lambda\|_{\H}^2+\left(\beta-\frac{1}{2\mu}\right)\||\y_\lambda||\nabla\y_\lambda|\|_{\H}^2\leq C(1+\|\f\|_{\H})^{3}.
 	\end{align}
 
Thus under Hypothesis \ref{hyp1} (condition (H.3)),  we have
	\begin{align}\label{336}
		\|\mathrm{A\y_{\lambda}}\|_{\H}\leq C  \ \text{ and } \  \||\nabla\y_{\lambda}||\y_{\lambda}|^{\frac{r-1}{2}}\|_{\H}\leq C
	\end{align}
	for $d=2,3$ with $r\in(3,\infty)$ and $d=r=3$ with $2\beta\mu\geq1$ for all $\y_{\lambda}\in\mathrm{D}(\mathrm{A}).$ Using interpolation inequality and estimates \eqref{P.11} and \eqref{P.22}, we have 
	\begin{align}\label{359}
		\|\mathcal{C}(\y_{\lambda})\|_{\H}\leq\|\y_{\lambda}\|_{\widetilde{\L}^{2r}}^{r}\leq\|\y_{\lambda}\|_{\widetilde{\L}^{r+1}}^{\frac{r+3}{4}}\|\y_{\lambda}\|_{\widetilde{\L}^{3(r+1)}}^{\frac{3(r-1)}{4}}\leq C, 
	\end{align}
for all $\y_\lambda\in\D(\A)$. Also,  by using H\"older's and Agmon's inequalities, we obtain 
	\begin{align*}
		\|\mathcal{B}(\y_{\lambda})\|_{\H}\leq\|(\y_{\lambda}\cdot\nabla) \y_{\lambda}\|_{\H}\leq\|\nabla\y_{\lambda}\|_{\H}
			\|\y_{\lambda}\|_{\H}^{1-\frac{d}{4}}\|\y_{\lambda}\|_{\H_{\mathrm{p}}^2}^{\frac{d}{4}}\leq C, 
	\end{align*}
for all $\y_\lambda\in\D(\A)$. Now, the equation \eqref{Y1} can be rewritten as 
	\begin{equation}\label{Y3}
		\y_{\lambda}+\mathcal{F}(\y_{\lambda})+\Phi_{\lambda}(\y_{\lambda})=\f, 
	\end{equation} 
where  $\mathcal{F}(\cdot)=\mu\mathrm{A} +\mathcal{B}(\cdot)+\beta\mathcal{C}(\cdot)+\kappa\I$. Hence from \eqref{336}, we conclude that 
	\begin{align}\label{338}
		\|\mathcal{F}(\y_{\lambda})\|_{\H}\leq C\ \text{ and } \ \|\Phi_{\lambda}(\y_{\lambda})\|_{\H}\leq C ,   \ \text{ for all }\ \y_{\lambda}\in\mathrm{D}(\mathrm{A}).
	\end{align}

\vskip 0.2 cm
\noindent 
\textbf{Step IV:} \textsl{Convergence of $\y_{\lambda}$ and proof of \eqref{316}.} 
The estimates \eqref{e1.2}, \eqref{336} and  \eqref{338}, and the Banach-Alaoglu theorem  guarantee the existence of a   weakly convergent subsequence $\{\y_{\lambda_j}\}$ of $\{\y_{\lambda}\}$ such that as $j\to\infty$
	\begin{equation}\label{3p4}
		\left\{
		\begin{aligned}
			\y_{\lambda_{j}}&\rightharpoonup \y, \ \text{ in } \  \V, \\
			\A\y_{\lambda_{j}}&\rightharpoonup \A\y, \ \text{ in }\ \H , 
		\end{aligned}	
	\right. \ \
		\left\{
	\begin{aligned}
 \Phi_{\lambda}(\y_{\lambda_{j}})&\rightharpoonup \f_{1}, \ \text{ in }\ \H ,  \\ 
		\mathcal{F}(\y_{\lambda_{j}})&\rightharpoonup \f_{2}, \ \text{ in }\ \H .
	\end{aligned}	
	\right.
	\end{equation}
	Since the embedding $\mathrm{D}(\mathrm{A})\hookrightarrow\V$ is compact,  we get the following strong convergence also:  \begin{align}\label{3p5}\y_{\lambda_{j}}\to \y\ \text{  in }\ \V.\end{align}
	Passing weak limit in \eqref{Y3}, we get
	\begin{align}
		\y + \f_{1}+\f_{2} = \f \ \text{ in }\ \H. 
	\end{align}
	In order to prove \eqref{316}, we need to show that $\f_{2}=\mathcal{F}(\y)$ and $\f_{1}\in\Phi(\y)$. For this, we rewrite equation \eqref{Y3} for $\lambda$ and $\widetilde{\lambda}$,  subtract and then take the inner product with $\y_{\lambda}-\y_{\widetilde{\lambda}}$  to find 
	\begin{align}\label{340}
		(\Phi_{\lambda}(\y_{\lambda})-\Phi_{\widetilde{\lambda}}(\y_{\widetilde{\lambda}}),\y_{\lambda}-\y_{\widetilde{\lambda}})+
		((\mathcal{F}+\mathrm{I})(\y_{\lambda})-(\mathcal{F}+\mathrm{I})(\y_{\widetilde{\lambda}}),\y_{\lambda}-\y_{\widetilde{\lambda}})=0.
	\end{align}
	By the monotonicity of $\mathcal{F}+\mathrm{I}$ (cf. Proposition \ref{prop33}), we conclude that 
	\begin{align}\label{3p6}
		(\Phi_{\lambda}(\y_{\lambda})-\Phi_{\widetilde{\lambda}}(\y_{\widetilde{\lambda}}),\y_{\lambda}-\y_{\widetilde{\lambda}})\leq 0,
	\end{align}
	for all $\lambda, \widetilde{\lambda}>0.$ By \cite[Proposition 1.3, part (iv), pp. 49]{VB2} (see \eqref{3p4}-\eqref{3p5} and \eqref{3p6}), we conclude that $(\y,\f_{1})\in\Phi$ and 
$\lim\limits_{\lambda,\widetilde{\lambda}\to 0} (\Phi_{\lambda}(\y_{\lambda})-\Phi_{\widetilde{\lambda}}(\y_{\widetilde{\lambda}}),\y_{\lambda}-\y_{\widetilde{\lambda}})=0.$ This also implies from \eqref{340} that 
	\begin{align}\label{345}
		\lim\limits_{\lambda,\widetilde{\lambda}\to 0} ((\mathcal{F}+\mathrm{I})(\y_{\lambda})-(\mathcal{F}+\mathrm{I})(\y_{\widetilde{\lambda}}),\y_{\lambda}-\y_{\widetilde{\lambda}})=0.
	\end{align}
	Since  $\y_{\lambda}\to \y, \  \mathcal{F}(\y_{\lambda})\rightharpoonup \f_{2}$ in $\H$ (see \eqref{3p4}-\eqref{3p5}), $\mathcal{F}+\mathrm{I}$ is maximal monotone (cf. Proposition \ref{prop33}) and \eqref{345} holds, then  by \cite[Lemma 1.3, pp. 49]
	{VB2}, we deduce that  $(\y,\y+\f_{2})\in\mathcal{F}+\mathrm{I},$ and this implies that $\mathcal{F}(\y)=\f_{2}.$ Hence it follows that $\f\in \y+\mathcal{F}(\y)+\Phi(\y),$ as claimed in \eqref{316}. It also follows that $\y\in\mathrm{D}(\mathcal{F})\cap\mathrm{D}(\Phi)=\mathrm{D}(\mathrm{A})\cap\mathrm{D}(\Phi)$ and hence $\mathrm{D}(\mathfrak{A})=\mathrm{D}(\mathrm{A})\cap\mathrm{D}(\Phi).$
	\vskip 0.2 cm
	\noindent 
	\textbf{Step V:} \textsl{Proof of \eqref{3.3838} and \eqref{3.3939}.} 
		From \eqref{3.6060}, it implies that 
	\begin{align}\label{a3}
		\|\mathrm{A\y_{\lambda}}\|_{\H}^{2}\leq C(1+\|\f\|_{\H}^{2})^\vartheta,
	\end{align}
where $\vartheta$ is given as in \eqref{vartheta}. For a fixed $\lambda>0$ and $\w\in\mathrm{D}(\mathrm{A})$, let
\begin{align}\label{3.725}
\g_{\lambda}=\mu\mathrm{A}\w+\mathcal{B}(\w)+\beta\mathcal{C}(\w)+\Phi_{\lambda}(\w)+\widetilde{\kappa}\w.
\end{align}
Then
\begin{align}\label{3.7171}
	\|\g_\lambda\|_{\H}^2\leq
	2 \wi\kappa^2\|\w\|_{\H}^2+ 2\|\mu\mathrm{A}\w+\mathcal{B}(\w)+\beta\mathcal{C}(\w)+\Phi_{\lambda}(\w)\|_{\H}^2.
\end{align}
Analogous to \eqref{a3} (for the solution $\y_{\lambda}$ of \eqref{Y1} with $\f\in\H$), it yields from \eqref{3.7171} that the solution $\w$ of \eqref{3.725} with $\g_{\lambda}\in\H$ satiesfies \eqref{3.3838}.
Now, for $\w\in\mathrm{D}(\mathrm{A})\cap\mathrm{D}(\Phi)$ and $\xi\in\Phi(\w)$, let  $$\g=\mu\mathrm{A}\w+\mathcal{B}(\w)+\beta\mathcal{C}(\w)+\xi+\widetilde{\kappa}\w.$$ Since $\g\in\H$, we obtain a sequence $\{\w_{\lambda}\}_{\lambda>0}\subset \H$ such that  $\w_{\lambda}$ is a solution of 
\begin{align*}
\mu\mathrm{A}\w_{\lambda}+\mathcal{B}(\w_{\lambda})+\beta\mathcal{C}(\w_{\lambda})+\Phi_{\lambda}(\w_{\lambda})+\widetilde{\kappa}\w_{\lambda}=\g, \ \text{ for all }\ \lambda>0. 
\end{align*}
Then as similar to Step III, we get $\w_{\lambda}\to\w$ in $\V$ and $\mathrm{A}\w_{\lambda}\rightharpoonup\mathrm{A}\w$ in $\H$. Now we calculate the estimate $\|\A\w_\lambda\|_{\H}^2$ as we calculate above and then passing the limit as $\lambda\to 0$, we obtain
\begin{align*}
\|\mathrm{A}\w\|_{\H}^{2}\leq
C(1+\|\w\|_{\H}^2+\|\mu\mathrm{A}\w+\mathcal{B}(\w)+\beta\mathcal{C}(\w)+\xi\|_{\H}^2)^\vartheta,
\end{align*}
where $\vartheta$ is defined as in \eqref{vartheta} and this completes the proof of \eqref{3.3939}. 
\end{proof}

\begin{remark}
{	It can be easily seen from \eqref{359} that $\|\mathcal{C}(\y_{\lambda})\|_{\H}$ is bounded uniformly. 
 Now, if we take the inner product in \eqref{Y2} with $\mathcal{C}(\y_\lambda)$, then we have 
\begin{align}\label{rev1}
	&\mu(\mathrm{A}\y_{\lambda},\mathcal{C}(\y_{\lambda}))+\beta\|\mathcal{C}(\y_{\lambda})\|_{\H}^2 +\widetilde{\kappa}(\y_{\lambda},\mathcal{C}(\y_{\lambda}))\nonumber\\&=(\f,\mathcal{C}(\y_{\lambda})) 	-(\mathcal{B}(\y_{\lambda}),\mathcal{C}(\y_{\lambda}))-(\Phi_{\lambda}(\y_{\lambda}),\mathcal{C}(\y_{\lambda})).
\end{align}
By performing similar calculations as above,  one can deduce for some $\rho>0$
\begin{align}\label{rev2}
	\frac{\mu}{2}\||\y_{\lambda}|^{\frac{r-1}{2}}|\nabla\y_{\lambda}|\|_{\H}^{2}+\frac{\beta}{2}\|\mathcal{C}(\y_{\lambda})\|_{\H}^{2}+\widetilde{\kappa}\|\y_\lambda\|_{\wi\L^{r+1}}^{r+1}\leq \frac{2}{\beta}\|\f\|_{\H}^{2}+\frac{2}{\beta}\|\Phi_{\lambda}(\y_{\lambda})\|_{\H}^{2}+\rho\|\nabla\y_\lambda\|_{\H}^2.
\end{align}
We are not gaining any additional regularity results from the above relation,  as the estimate depends on the bounds  of  $\|\Phi_{\lambda}(\y_{\lambda})\|_{\H}$. 
}
\end{remark}

From Proposition \ref{prop33}-\ref{prop3.3} and \cite[Theorems 1.4-1.6, pp. 214-216]{VB2}, we have following immediate result:
	\begin{proposition}\label{thm1.1}
		Let $T>0$ and assume that $\Phi\subset\H\times\H$ satisfies Hypothesis \ref{hyp1}. Let $\y_0\in\D(\A)\cap\D(\Phi)$ and $\f \in \W^{1,1}(0,T;\H)$. For $d=2$ with $r\in[3,\infty)$ and $d=3$ with $r\in[3,\infty)$ ($2\beta\mu\geq1$ for $r=3$), there exists a unique strong solution 
		\begin{align}\label{1p.4}
			\y\in \W^{1,\infty}(0,T;\H)\cap\mathrm{L}^{\infty}(0,T;\D(\A))\cap \C([0,T];\V),
		\end{align} such that in $\H$
		\begin{equation}\label{1p4}
			\left\{
			\begin{aligned}
				\frac{\d \y(t)}{\d t}+\mu\A\y(t)+\mathcal{B}(\y(t))+\beta\mathcal{C}(\y(t))+\Phi(\y(t))&\ni \f(t), \ \text{ a.e. } \ t\in[0,T], \\
				\y(0)&=\y_0. 
			\end{aligned}
			\right.
		\end{equation}
		Furthermore, $\y$ is right differentiable, $\frac{\d^+\y}{\d t} $ is right continuous, and 
		\begin{equation}\label{1p5}
			\frac{\d^+ \y(t)}{\d t}+\left(\mu\A\y(t)+\mathcal{B}(\y(t))+\beta\mathcal{C}(\y(t))+\Phi(\y(t))-\f(t)\right)^0= \mathbf{0}, \ \text{ for all  } \ t\in[0,T),
		\end{equation}
		{where $\left(-\f+\mu\A\y+\mathcal{B}(\y)+\beta\mathcal{C}(\y)+\Phi(\y)\right)^0$ denotes the minimal selection  of the multi-valued map $\y\mapsto(-\f+\mu\A\y+\mathcal{B}(\y)+\beta\mathcal{C}(\y)+\Phi(\y))$, that is, $\left(-\f+\mu\A\y+\mathcal{B}(\y)+\beta\mathcal{C}(\y)+\Phi(\y)\right)^0$ is the element of minimum norm in $\left(-\f+\mu\A\y+\mathcal{B}(\y)+\beta\mathcal{C}(\y)+\Phi(\y)\right)$ (see \cite[Reamrk, pp. 76]{VB2}).} 
\end{proposition}

A similar result holds for the system \eqref{1p4}, when one replaces $\Phi$ with the Yosida approximation $\Phi_\lambda$. 

\begin{proposition}\label{prop3.5}
	Let $\Phi\subset \H\times \H$ satisfy Hypothesis \ref{hyp1}. Let $\f\in\W^{1,1}(0,T;\H)$ and $\y_0\in\D(\A)\cap\D(\Phi)$. Then there exists a unique strong solution 
	\begin{align}\label{355}
		\y_\lambda\in \W^{1,\infty}(0,T;\H)\cap\mathrm{L}^{\infty}(0,T;\D(\A))\cap \C([0,T];\V)
	\end{align}
	to the problem 
	\begin{equation}\label{3.29}
		\left\{
		\begin{aligned}
			\frac{\d\y_\lambda(t)}{\d t}+\mu\A\y_\lambda(t)+\mathcal{B}(\y_\lambda(t))+\beta\mathcal{C}(\y_\lambda(t))+\Phi_{\lambda}(\y_\lambda(t))&=\f(t), \ \text{ a.e. } \ t\in[0,T], \\
			\y_\lambda(0)&=\y_0. 
		\end{aligned}
		\right.
	\end{equation}
	
	Furthermore, $\y_\lambda$ is right differentiable, $\frac{\d^+\y_\lambda}{\d t} $ is right continuous, and 
	\begin{equation}\label{3p55}
		\frac{\d^+\y_\lambda(t)}{\d t}+\mu\A\y_\lambda(t)+\mathcal{B}(\y_\lambda(t))+\beta\mathcal{C}(\y_\lambda(t))+\Phi_{\lambda}(\y_\lambda(t))= \f(t), \ \text{ for all  } \ t\in[0,T). 
	\end{equation}
\end{proposition}


\section{Proof of Theorem \ref{thm1.2}} \label{sec44}\setcounter{equation}{0}
The aim of this section is to prove the Theorem \ref{thm1.2} using the solvability results obtained in Proposition \ref{thm1.1}. We first provide some uniform energy estimates for the solutions of the problem \eqref{3.29}. 
\subsection{Energy estimates for the solution of the problem \eqref{3.29}} From the abstarct result Proposition \ref{thm1.1}, we infer that the problem \eqref{1p4} has a unique strong solution with the regularity given in \eqref{1p.4}. Our aim in this subsection is to obtain some  energy estimates for the solution of the problem \eqref{1p7} so that we can pass to the limit. In order to do this, we first  obtain suitable energy estimates for   the solution $\y_{\lambda}(\cdot)$ for the approximate problem \eqref{3.29}, which also has a unique strong solution with the regularity given in \eqref{1p.4}. 
\begin{proposition}\label{prop4.1}
	Let $\y_\lambda(\cdot)$  be the unique strong solution of the problem \eqref{3.29} obtained in Proposition \ref{prop3.5}. Then for $\f\in\W^{1,1}(0,T;\H)$ and $\y_0\in\D(\A)\cap\D(\Phi)$, the solution $\y_\lambda(\cdot)$   satisfies the following energy estimates: 
	\begin{align}\label{356}
&	\sup_{t\in[0,T]}\|\y_\lambda(t)\|_{\H}^2+\mu\int_0^T\|\nabla\y_\lambda(t)\|_{\H}^2\d t+\beta\int_0^T\|\y_\lambda(t)\|_{\wi\L^{r+1}}^{r+1}\d t\nonumber\\&\leq C\left(T, \|\y_0\|_{\H}, \|\f\|_{\mathrm{L}^2(0,T;\H)},\|\Phi(\boldsymbol{0})\|_{\H}\right), 
	\end{align}
where $C$ is independent of $\lambda$. Furthermore, we have 
	\begin{align}\label{357}
	&	\sup_{t\in[0,T]}	\|\nabla\y_\lambda(t)\|_{\H}^2+\mu\int_0^T\|\A\y_\lambda(t)\|_{\H}^2\d t+\beta\int_0^T\||\nabla \y_{\lambda}(t)||\y_{\lambda}(t)|^{\frac{r-1}{2}}\|_{\H}\d t\nonumber\\&\leq C\left(\mu,\beta,T,\|\A\y_0\|_{\H},\varphi(\y_0),\|\Phi(\y_0)\|_{\H},\|\Phi(\boldsymbol{0})\|_{\H},\|\f(0)\|_{\H},\|\f\|_{\mathrm{L}^2(0,T;\H)}\right), 
\end{align}
where $C$ is independent of $\lambda$.
\end{proposition}
\begin{remark}
	Since $\f\in\mathrm{W}^{1,1}(0,T;\H)$, $\f$ is absolutely continuous and hence $\f\in\C([0,T];\H)$.
\end{remark}
\begin{proof}[Proof of Proposition \ref{prop4.1}] We prove \eqref{356} and \eqref{357} in the following steps:
	\vskip 0.2 cm
	\noindent 
	\textbf{Step I:} \textsl{Proof of \eqref{356}.}
 Taking the inner product with $\y_\lambda$ in \eqref{3.29}, we obtain 
\begin{align}\label{e.1}
	\frac{1}{2}\frac{\d}{\d t}\|\y_\lambda(t)\|_{\H}^{2}+\mu\|\nabla\y_\lambda(t)\|_{\H} ^{2} +\beta\|\y_\lambda(t)\|_{\wi\L^{r+1}}^{r+1} + (\Phi_{\lambda}(\y_\lambda(t)),\y_\lambda(t))  &= (\f(t),\y_\lambda(t)),
\end{align}
for a.e. $t\in[0,T]$, where we have used the fact that $(\mathcal{B}(\y_\lambda),\y_\lambda)=0$. {Using an estimate similar to \eqref{m}} and Cauchy-Schwarz, Young's and Gronwall's inequalities in \eqref{e.1}, we deduce 
\begin{align}\label{e.3}
	&	\|\y_\lambda(t)\|_{\H}^{2}+\mu\int_0^t\|\y_\lambda(s)\|_{\V} ^{2} \d s+2\beta\int_0^t\|\y_\lambda(s)\|_{\wi\L^{r+1}}^{r+1}\d s \leq C,
\end{align}	
for all $t\in[0,T]$, where
\begin{align*}
	C=\|\y_0\|_{\H}^2+
	2\left(\int_0^T\|\f(s)\|_{\H}^{2}\d s+T\|\Phi(\boldsymbol{0})\|_{\H}^{2}\right) e^T.
\end{align*}
	\vskip 0.2 cm
\noindent 
\textbf{Step II:} \textsl{Regularity estimates.} In order to obtain the energy estimate \eqref{357}, we first need further regularity estimates on the solution.  This is due to the (H.3) assumption in Hypothesis \ref{hyp1}. We observe that $\y_\lambda(\cdot)$ satisfies  for any $h>0$
\begin{align*}
\frac{\d\y_\lambda(t+h)}{\d t}+\mu\A\y_\lambda(t+h)+\mathcal{B}(\y_\lambda(t+h))+\beta\mathcal{C}(\y_\lambda(t+h))+\Phi_{\lambda}(\y_\lambda(t+h))&=\f(t+h), 	
\end{align*}
for a.e. $t\in [0,T]$. Then subtracting above equation  from \eqref{3.29}, and taking the inner product with $\y_\lambda(\cdot+h)-\y_{\lambda}(\cdot)$ and then using \eqref{2.23}, we get for a.e. $t\in [0,T]$
\begin{align}\label{d.1}
&\frac{1}{2}\frac{\d}{\d t}\|\y_\lambda(t+h)-\y_\lambda(t)\|_{\H}^2+\mu\|\nabla(\y_\lambda(t+h)-\y_\lambda(t))\|_{\H}^2+\frac{\beta}{2}\||\y(t)|^{\frac{r-1}{2}}(\y(t+h)-\y(t))\|_{\H}^2\nonumber\\&\leq (\f(t+h)-\f(t),\y_\lambda(t+h)-\y_\lambda(t))-(\mathcal{B}(\y_\lambda(t+h))-\mathcal{B}(\y_\lambda(t)),\y_\lambda(t+h)-\y_\lambda(t)), 
\end{align} 
where we have used the monotonicity property of the Yosida approximation $\Phi_{\lambda}(\cdot)$. We consider the follwing cases: 

\vskip 0.2cm
\textsl{For $r>3.$} From \eqref{2.30} and the Cauchy-Schwarz inequality, \eqref{d.1} yields 
\begin{align}\label{mr1}
&(\|\y_\lambda(t+h)-\y_\lambda(t)\|_{\H})\frac{\d}{\d t} \|\y_\lambda(t+h)-\y_\lambda(t)\|_{\H} + \frac{\mu}{2}\|\nabla(\y_\lambda(t+h)-\y_\lambda(t))\|_{\H}^2\nonumber\\&\leq
\|\f(t+h)-\f(t)\|_{\H}\|\y_\lambda(t+h)-\y_\lambda(t)\|_{\H}+\varrho\|\y_\lambda(t+h)-\y_\lambda(t)\|_{\H}^2,
\end{align}
or we can write 
\begin{align*}
\frac{\d}{\d t}\|\y_\lambda(t+h)-\y_\lambda(t)\|_{\H}\leq 
\|\f(t+h)-\f(t)\|_{\H}+\varrho\|\y_\lambda(t+h)-\y_\lambda(t)\|_{\H},
\end{align*}
for a.e. $t\in[0,T]$. By Gronwall's inequality, we have 
\begin{align*}
\|\y_\lambda(t+h)-\y_\lambda(t)\|_{\H}\leq e^{\varrho t}	\left(\|\y_\lambda(h)-\y_\lambda(0)\|_{\H}+\int_0^t \|\f(s+h)-\f(s)\|_{\H} \d s\right),
\end{align*}
for all $t\in[0,T]$. On dividing by $h$ and then taking limit as $h\to 0$, we obtain  for all $t\in[0,T]$
\begin{align}\label{mr2}
	\left\|\frac{\d^+\y_\lambda(t)}{\d t}\right\|_{\H}&\leq e^{\varrho T}	\left(\left\|\frac{\d^+\y_\lambda(0)}{\d t}\right\|_{\H}+\int_0^T\left\|\frac{\d\f}{\d t}(t)\right\|_{\H} \d t \right)\nonumber\\&\leq e^{\varrho T}	\left(\mu\|\A\y_\lambda(0)\|_{\H}+\|\mathcal{B}(\y_\lambda(0))\|_{\H}+\beta\|\mathcal{C}(\y_\lambda(0))\|_{\H}+\|\Phi_{\lambda}(\y_\lambda(0))\|_{\H}\right.\nonumber\\&\quad\left.+\|\f(0)\|_{\H}+\|\f\|_{\W^{1,1}(0,T;\H)}\right)\nonumber\\&\leq C\left(\mu,\beta,T,\|\A\y_0\|_{\H},\|\Phi(\y_0)\|_{\H},\|\f(0)\|_{\H},\|\f\|_{\W^{1,1}(0,T;\H)}\right),
\end{align}
where we have used \eqref{3p55} and the fact that $\f\in\W^{1,1}(0,T;\H)$ implies $\f\in\C([0,T];\H)$.  Now on integrating \eqref{mr1}, we obtain
\begin{align*}
	\mu\int_0^t \|\nabla(\y_\lambda(s+h)-\y_\lambda(s))\|_{\H}^2 \d s&\leq \|\y_\lambda(h)-\y_\lambda(0)\|_{\H}^2+2\varrho\int_0^t \|\y_\lambda(s+h)-\y_\lambda(s)\|_{\H}^2 \d s\nonumber\\&\quad+2\int_0^t \|\f(s+h)-\f(s)\|_{\H}\|\y_\lambda(s+h)-\y_\lambda(s)\|_{\H}\d s.
\end{align*}
On dividing both sides by $h^2$ and then passing limit as $h\to 0$, we get
\begin{align}\label{3.8989}
\mu\int_0^t \left\|\frac{\d(\nabla\y_\lambda(s))}{\d s}\right\|_{\H}^2 \d s \leq\left\|\frac{\d^+\y_\lambda(0)}{\d t}\right\|_{\H}^2+2\varrho\int_0^t\left\|\frac{\d^+\y_\lambda(s)}{\d t}\right\|_{\H}^2\d s+ \int_0^t\left\|\frac{\d\f}{\d t}(s)\right\|_{\H}\left\|\frac{\d^+\y_\lambda(s)}{\d t}\right\|_{\H} \d s,
\end{align}
for all $t\in[0,T]$. By using \eqref{mr2}, it implies that 
\begin{align*}
\int_0^T \left\|\frac{\d(\nabla\y_\lambda(s))}{\d s}\right\|_{\H}^2 \d s\leq C\left(\mu,\beta,T,\|\A\y_0\|_{\H},\|\Phi(\y_0)\|_{\H},\|\f(0)\|_{\H},\|\f\|_{\W^{1,1}(0,T;\H)}\right).
\end{align*}

\vskip 0.2cm
\textsl{For $r=3$ with $2\beta\mu\geq1.$}  Once again by using Cauchy-Schwarz and Young's inequalities, one can obtain the following:
\begin{align*}
	&\frac{\d}{\d t}\|\y_\lambda(t+h)-\y_\lambda(t)\|_{\H}^2+2\left(\mu-\frac{1}{2\beta}\right)\|\nabla(\y_\lambda(t+h)-\y_\lambda(t))\|_{\H}^2\nonumber\\&\leq
	\|\f(t+h)-\f(t)\|_{\H}^2+\|\y_\lambda(t+h)-\y_\lambda(t)\|_{\H}^2.	
\end{align*}
Using the similar calculations as we have done in the case $r>3$, we get the similar estimates as in \eqref{3.8989}. Taking the inner product with $\frac{\d \y_\lambda}{\d t}$ in \eqref{3.29} and then using the Cauchy-Schwarz and Young's inequalities, we get for a.e. $t\in[0,T]$
\begin{align}\label{389}
	\left\|\frac{\d \y_\lambda(t)}{\d t}\right\|_{\H}^2&+\frac{\mu}{2}\frac{\d}{\d t} \|\nabla\y_\lambda(t)\|_{\H}^2+\frac{\beta}{r+1}\frac{\d}{\d t}\|\y_\lambda(t)\|_{\wi\L^{r+1}}^{r+1}+\left(\frac{\d \y_\lambda(t)}{\d t},\Phi_{\lambda}(\y_\lambda(t))\right)\nonumber\\&=\left(\f(t),\frac{\d \y_\lambda(t)}{\d t}\right)+\left(\mathcal{B}(\y_\lambda(t)),\frac{\d \y_\lambda(t)}{\d t}\right).
\end{align}
We calculate $\left(\mathcal{B}(\y_\lambda),\frac{\d \y_\lambda}{\d t}\right)$ by using H\"older's, Young's and interpolation inequalities as
\begin{align*}
\left(\mathcal{B}(\y_\lambda),\frac{\d \y_\lambda}{\d t}\right)&=-b\left(\y_\lambda,\frac{\d \y_\lambda}{\d t},\y_\lambda\right)\leq
\|\y_\lambda\|_{\wi\L^4}^2\left\|\frac{\d (\nabla\y_\lambda)}{\d t}\right\|_{\H} \nonumber\\&\leq \frac{1}{2}\left\|\frac{\d(\nabla\y_\lambda)}{\d t}\right\|_{\H}^2+ \frac{1}{2}\|\y_\lambda\|_{\wi\L^{r+1}}^{\frac{2(r+1)}{r-1}}\|\y_\lambda\|_{\H}^{\frac{2(r-3)}{r-1}}. 
\end{align*}
Therefore from \eqref{389}, it is immediate that 
\begin{align}
&\frac{\mu}{2}\|\nabla\y_\lambda(t)\|_{\H}^2+\frac{\beta}{r+1}\|\y_\lambda(t)\|_{\wi\L^{r+1}}^{r+1}+\frac{1}{2}	\int_0^t	\left\|\frac{\d \y_\lambda(s)}{\d t}\right\|_{\H}^2\d s+\int_0^t\left(\frac{\d \y_\lambda(s)}{\d t},\Phi_{\lambda}(\y_\lambda(s))\right)\d s\nonumber\\&\leq \frac{\mu}{2}\|\nabla\y_0\|_{\H}^2+\frac{\beta}{r+1}\|\y_0\|_{\wi\L^{r+1}}^{r+1}+\frac{1}{2}\int_0^t\|\f(s)\|_{\H}^2\d s+\frac{1}{2}\int_0^t\left\|\frac{\d(\nabla\y_\lambda(s))}{\d t}\right\|_{\H}^2\d s \nonumber\\&\quad+\frac{1}{2}t^{\frac{r-3}{r-1}}\sup_{s\in[0,t]}\|\y_\lambda(s)\|_{\H}^{\frac{2(r-3)}{r-1}}\left(\int_0^t\|\y_\lambda(s)\|_{\wi\L^{r+1}}^{r+1}\d s\right)^{\frac{2}{r-1}},
\end{align}
for all $t\in[0,T]$. From Hypothesis \ref{hyp1}, we know that $\Phi=\partial\varphi,$ where $\varphi:\H\to\overline{\R}$ is a lower semicontinuous proper convex function. Then by an application of \cite[Chapter 2, Theorem 2.2]{VB2} yields that the Yosida approximation $\Phi_{\lambda}$ is the Gateaux derivative of $\varphi_{\lambda}$, for all $\lambda>0$, that is, $\Phi_{\lambda}=\nabla\varphi_{\lambda}$, where $\varphi_{\lambda}$ is the regularization of $\varphi$ \cite[pp. 64]{VB2}, given by 
\begin{align}\label{e.9}
	\varphi_{\lambda}(\y)=\inf\left\{\frac{\|\y-\z\|_{\H}^{2}}{2\lambda}+\varphi(\z): \z\in\H\right\}, \  \ \text{for all} \ \y\in\H.
\end{align}
Moreover by a standard calculation, we have 
\begin{align}\label{e.10}
	\frac{\d}{\d t}\left[\varphi_{\lambda}(\y_\lambda(\cdot))\right]	= \left(\frac{\d \y_\lambda(\cdot)}{\d t},(\nabla\varphi_{\lambda})(\y_\lambda(\cdot))\right), 
\end{align} 
and 
\begin{align}\label{e..10}
	\int_0^t\left(\frac{\d \y_\lambda(s)}{\d t},\Phi_{\lambda}(\y_\lambda(s))\right)\d s=\varphi_{\lambda}(\y_\lambda(t))-\varphi_{\lambda}(\y_{0}),
\end{align}
for all $t\in[0,T]$. From \cite[Chapter 2, Proposition 1.3]{VB2},  we infer that $\J_{\lambda}:=(\I+\lambda\Phi)^{-1}$ is bounded on bounded subsets of $\H$. Furthermore, from \cite[Chapter 2, Theorem 2.2]{VB2} (see also \cite[Proposotion 4.33, Chapter 4]{Hu}) we also have 
\begin{align}\label{e.12}
	\varphi(\J_{\lambda}(\y))\leq\varphi_{\lambda}(\y)\leq\varphi(\y),  \  \text{ for all } \ \lambda>0,\  \y\in\H.
\end{align}
From \cite[Chapter 2, Proposition 2.1]{VB2},  we know that any proper lower semicontinuous convex function is bounded from below by an affine function. Therefore, there exists $\w\in\H$ and $q\in\R$ such that 
\begin{align}\label{e.13}
	\varphi(\y)\geq(\y,\w)+q, \  \  \text{for all} \ \y\in\H.
\end{align}
{From \cite[proposition 1.3 (iii)]{VB2}, we know that $\J_\lambda:\H\to\H$ is bounded on bounded subsets of $\H$. Moreover, from \eqref{356}, we know that $\|\y_\lambda\|_{\H}\leq C$, where $C$ is constant independent of $\lambda$. Therefore, we  conclude that 
\begin{align*}
	\|\J_\lambda(\y_\lambda)\|_{\H}\leq C,
\end{align*}
where constant $C$ is independent of $\lambda$.} Thus, From \eqref{e.12}, \eqref{e.13} and application of the Cauchy-Schwarz inequality yield
\begin{align}\label{e.14}
	-\varphi_{\lambda}(\y_\lambda)\leq-\varphi(\J_{\lambda}(\y_\lambda))&\leq-(\J_{\lambda}(\y_\lambda),\w)-q\leq	\|\J_{\lambda}(\y_\lambda)\|_{\H}\|\w\|_{\H}+|q|\leq C,
\end{align}
where $C$ is independent of $\lambda$. Thus using  $\varphi_{\lambda}(\y_0)\leq\varphi(\y_0)$ in \eqref{e..10}, we deduce 
\begin{align}\label{e..14}
-\int_0^t\left(\frac{\d \y_\lambda(s)}{\d t},\Phi_{\lambda}(\y_\lambda(s))\right)\d s\leq C,
\end{align} 
where the constant $C$ depends on $\varphi(\y_0).$ Thus from \eqref{356}, \eqref{3.8989} and \eqref{e..14}, we get for all $t\in[0,T]$
\begin{align}\label{e.15.}
\frac{\mu}{2}\|\nabla\y_\lambda(t)\|_{\H}^2+\frac{\beta}{r+1}\|\y_\lambda(t)\|_{\wi\L^{r+1}}^{r+1}+\frac{1}{2}	\int_0^t	\left\|\frac{\d \y_\lambda(s)}{\d s}\right\|_{\H}^2\d s\leq C,
\end{align} where
$C=C\left(\mu,\beta,T,\|\A\y_0\|_{\H},\varphi(\y_0),\|\Phi(\y_0)\|_{\H},\|\Phi(\boldsymbol{0})\|_{\H},\|\f(0)\|_{\H},\|\f\|_{\W^{1,1}(0,T;\H)}\right)$. 

	\vskip 0.2 cm
\noindent 
\textbf{Step III:} \textsl{Proof of \eqref{357}.} 
We take inner product with $\A\y_\lambda$ in \eqref{3.29} to obtain 
\begin{align}\label{e.4}
	&\frac{1}{2}\frac{\d}{\d t} \|\nabla\y_\lambda\|_{\H}^{2}+\mu\|\A\y_\lambda(t)\|_{\H}^{2}	+\beta(\mathcal{C}(\y_\lambda(t)),\A\y_\lambda(t))\nonumber\\& =(\f(t),\A\y_\lambda(t))-(\mathcal{B}(\y_\lambda(t)),\A\y_\lambda(t))-(\Phi_{\lambda}(\y_\lambda(t)),\A\y_\lambda(t)),
\end{align}
for a.e. $t\in[0,T]$. This yields 
\begin{align}\label{e.55.e}
&\|\nabla\y_\lambda(t)\|_{\H} ^{2} + 2\mu\int_0^t\|\A\y_\lambda(s)\|_{\H}^2\d s+2\beta\int_0^t (\mathcal{C}(\y_\lambda(s)),\A\y_\lambda(s)) \d s\nonumber\\&=\|\nabla\y_0\|_{\H}^2+2
\int_0^t (\f(s),\A\y_\lambda(s)) \d s-2\int_0^t (\mathcal{B}(\y_\lambda(s)),\A\y_\lambda(s)) \d s \nonumber\\&\quad-2\int_0^t (\Phi_{\lambda}(\y_\lambda(s)),\A\y_\lambda(s)) \d s,
\end{align}
for all $t\in[0,T]$. Then, using the similar calculations as in \eqref{phi1} and using the fact that $\f\in\W^{1,1}(0,T;\H)$ implies $\f\in\C([0,T];\H)$, one can derive the following estimate:
\begin{align}\label{3104}
	&\|\nabla\y_\lambda(t)\|_{\H} ^{2}+\frac{\mu}{2}\int_0^t\|\A\y_\lambda(s)\|_{\H}^2\d s+\frac{3\beta}{2}\int_0^t \||\nabla \y_\lambda(s)||\y_\lambda(s)|^{\frac{r-1}{2}}\|_{\H}^{2} \d s\nonumber\\&\leq C+
	\begin{cases}
		\varsigma\int_0^t \|\Phi_{\lambda}(\y_\lambda(s))\|_{\H}^2 \d s, \ &\text{for } d=2 \text{ with } r\in(3,\infty)\text{ and } d=3 \text{ with } r\in(3,5),\\
		0, &\text{for} \ d=3 \ \text{with} \ r\in[5,\infty),
	\end{cases}
\end{align}
where $C=C(\mu,\beta,T,\|\y_0\|_{\H},\|\f\|_{\mathrm{L}^2(0,T;\H)}, \|\Phi(\boldsymbol{0})\|_{\H})$. This completes the proof of \eqref{357} for $d=3$ with $r\in[5,\infty).$ Moreover, calculations similar to \eqref{r3e} yields for $d=3$ (with $2\beta\mu\geq1$)
\begin{align}\label{r31.}
	&\|\nabla\y_\lambda(t)\|_{\H} ^{2}+\frac{3\mu}{4}\int_0^t\|\A\y_\lambda(s)\|_{\H}^2\d s+2\left(\beta-\frac{1}{2\mu}\right)\int_0^t \||\nabla \y_\lambda(s)||\y_\lambda(s)|\|_{\H}^{2} \d s\nonumber\\&\leq C+
	\varsigma\int_0^t \|\Phi_{\lambda}(\y_\lambda(s))\|_{\H}^2 \d s,
\end{align}
where $C=C(\mu,\beta,T,\|\y_0\|_{\H},\|\f\|_{\mathrm{L}^2(0,T;\H)}, \|\Phi(\boldsymbol{0})\|_{\H})$.

	\vskip 0.2 cm
\noindent 
\textbf{Step IV:} \textsl{An estimate for $\int_{0}^{t}\|\Phi_{\lambda}(\y_\lambda(s))\|_{\H}^{2} \d s$.}
Let us now find a bound for $\int_{0}^{t}\|\Phi_{\lambda}(\y_\lambda(s))\|_{\H}^{2} \d s$. For this, taking the inner product of \eqref{3.29} with $\Phi_{\lambda}(\y_\lambda(\cdot))$, we get for a.e. $t\in[0,T]$
\begin{align}\label{e.8}
&	\left(\frac{\d \y_\lambda(t)}{\d s},\Phi_{\lambda}(\y_\lambda(t))\right)+	\mu(\Phi_{\lambda}(\y_\lambda(t)),\A\y_\lambda(t))+(\mathcal{B}(\y_\lambda(t)),\Phi_{\lambda}(\y_\lambda(t))\nonumber\\&\qquad+\beta(\mathcal{C}(\y_\lambda(t)),\Phi_{\lambda}(\y_\lambda(t))) +\|\Phi_{\lambda}(\y_\lambda(t))\|_{\H}^{2}=(\f(t),\Phi_{\lambda}(\y_\lambda(t))).
\end{align}
Integrating \eqref{e.8}, and using the condition (H3) of Hypothesis \ref{hyp1} and \eqref{e..14}, we obtain 
\begin{align}\label{3.105.}
(1-\mu\varsigma)\int_{0}^{t}\|\Phi_{\lambda}(\y_\lambda(s))\|_{\H}^{2}	\d s&\leq C+
	\int_{0}^{t} (\f(s),\Phi_{\lambda}(\y_\lambda(s)))\d s-\int_{0}^{t}(\mathcal{B}(\y_\lambda(s)),\Phi_{\lambda}(\y_\lambda(s)) \d s\nonumber\\&\quad-\beta\int_{0}^{t} (\mathcal{C}(\y_\lambda(s)),\Phi_{\lambda}(\y_\lambda(s)))\d s,
\end{align}
for all $t\in[0,T]$. A calculation similar to  \eqref{e1.4}  yields
\begin{align}\label{e.15}
\left|\int_{0}^{t}(\mathcal{B}(\y_\lambda(s)),\Phi_{\lambda}(\y_\lambda(s)) \d s\right|&\leq C+\frac{1-\mu\varsigma}{8}\int_0^t \|\Phi_{\lambda}(\y_{\lambda}(s))\|_{\H}^{2} \d s \nonumber\\&\quad+\frac{\mu(1-\mu\varsigma)}{8\varsigma}\int_0^t \|\mathrm{A\y_{\lambda}}(s)\|_{\H}^{2} \d s,
\end{align} 
where we have used \eqref{e.15.} also. Using the estimates \eqref{e.15.} and \eqref{3104}, and performing similar calculations as in \eqref{3.5555}-\eqref{3.6666}, we  find 
\begin{align}\label{e.16.}
\left|\int_0^t(\mathcal{C}(\y_\lambda(s)),\Phi_{\lambda}(\y_\lambda(s)))\d s\right|\leq 
 C+\frac{1-\mu\varsigma}{8\beta}\int_0^t\|\Phi_{\lambda}(\y_{\lambda}(s))\|_{\H}^{2}\d s,
\end{align}
for $d=3$ and $r\in(3,5)$ and $d=2$ with $r\in(3,\infty)$.
Using the Cauchy-Schwarz and Young's inequailities, we further have 
\begin{align}\label{e.18.}
\left|\int_{0}^{t} (\f(s),\Phi_{\lambda}(\y_\lambda(s)))\d s\right|\leq C\int_0^t \|\f(s)\|_{\H}^2\d s+\frac{1-\mu\varsigma}{4}\int_0^t\|\Phi_{\lambda}(\y_{\lambda}(s))\|_{\H}^{2}\d s.
\end{align}
Combining \eqref{3.105.}-\eqref{e.18.} in \eqref{3104}, we conclude for $d=2,3$ with $r>3$
\begin{align}\label{3.109.}
	\|\nabla\y_\lambda(t)\|_{\H} ^{2}+\frac{\mu}{4}\int_0^t\|\A\y_\lambda(s)\|_{\H}^2\d s+\frac{3\beta}{2}\int_0^t \||\nabla \y_\lambda(s)||\y_\lambda(s)|^{\frac{r-1}{2}}\|_{\H}^{2} \d s\leq C,
\end{align}
for all $t\in[0,T]$. Also from \eqref{r31.}, for $d=r=3$ with $2\beta\mu\geq1$, we find
\begin{align}\label{3.109..}
&\|\nabla\y_\lambda(t)\|_{\H} ^{2}+\frac{\mu}{2}\int_0^t\|\A\y_\lambda(s)\|_{\H}^2\d s+2\left(\beta-\frac{1}{2\mu}\right)\int_0^t \||\nabla \y_\lambda(s)||\y_\lambda(s)|\|_{\H}^{2} \d s\leq C,
\end{align}
for all $t\in[0,T]$. Therfore, we deduce that
\begin{align}\label{3.110.}
	\int_0^t \|\Phi_{\lambda}(\y_\lambda(s))\|_{\H}^2 \d s\leq C,
\end{align}
for all $t\in[0,T]$, which completes the proof. 

\end{proof}

\subsection{Passing to the limit as $\lambda\to 0$} Let us now pass $\lambda\to 0$ and obtain the energy estimates for the solution of the problem \eqref{1p7}. 
\begin{proposition}\label{soln}
The limit of the sequence $(\y_\lambda)_{\lambda>0}$ satisfies the problem \eqref{1p7} for a.e. $t\in[0,T]$ in $\H$. 
\end{proposition}

\begin{proof}
	The proof of this proposition depends on a Minty-Browder technique. From the Proposition \ref{prop3.5}, we have 
\begin{align}\label{c.1}
	\y_\lambda\in \W^{1,\infty}(0,T;\H)\cap \mathrm{L}^{\infty}(0,T;\D(\A))\cap \C([0,T];\V).
\end{align}
From \eqref{3.109.}-\eqref{3.110.}, we have uniform bounds for the sequences 
\begin{align}\label{alin}
(\A\y_\lambda)_{\lambda>0} \ \text{and} \ (\Phi_{\lambda}(\y_\lambda))_{\lambda>0 }\ \text{ in } \ \mathrm{L}^2(0,T;\H).
\end{align}
Thus, by making use of the Banach-Alaoglu theorem, we infer from \eqref{356}-\eqref{357}, \eqref{e.15.} and \eqref{alin}, the following weak and weak-star convergences:
\begin{equation}\label{c2}
	\left\{
	\begin{aligned}
			\y_\lambda&\stackrel{\ast}{\rightharpoonup}\ \y \ &&\text{ in } \ \mathrm{L}^{\infty}(0,T;\V\cap\wi\L^{r+1}),\\
			\y_\lambda&\rightharpoonup\y \ &&\text{ in } \ \mathrm{L}^{r+1}(0,T;\wi\L^{3(r+1)}),\\
	\frac{\d\y_\lambda}{\d t}&\rightharpoonup \ \frac{\d\y}{\d t} \ &&\text{ in } \ \mathrm{L}^2(0,T;\H),
\end{aligned}\right.
	\left\{
\begin{aligned}
	\A\y_\lambda&\rightharpoonup\ \A\y  \ &&\text{ in } \ \mathrm{L}^2(0,T;\H),\\
	\Phi_{\lambda}(\y_\lambda)&\rightharpoonup\ \phi  \ &&\text{ in } \ \mathrm{L}^2(0,T;\H).
\end{aligned}\right.
\end{equation}
Since   $\V\hookrightarrow\H\hookrightarrow\V'$, the embedding of $\V\hookrightarrow\H$ is compact, and the fact that $\y \in\mathrm{L}^{\infty}(0,T;\V)$, $\frac{\d\y}{\d t} \in\mathrm{L}^2(0,T;\H)\hookrightarrow\mathrm{L}^2(0,T;\V')$ imply 
	\begin{align}\label{c3}
		\y_\lambda\to\y \ \text{ in } \ \C([0,T];\H),
	\end{align}
where we have used the Aubin-Lions compactness lemma. Since $\D(\A)\hookrightarrow\V\hookrightarrow\H$,  $(\y_\lambda)_{\lambda>0}$ is bounded in $\mathrm{L}^2(0,T;\D(\A))$ and $\left(\frac{\d\y_\lambda}{\d t}\right)_{\lambda>0}$ is bounded in $\mathrm{L}^2(0,T;\H)$, and the embedding $\D(\A)\hookrightarrow\V$ is compact, it implies once again from Aubin-Lions compactness lemma that 
\begin{align}\label{c3.}
	\y_\lambda\to\y \ \text{ in } \ \mathrm{L}^2(0,T;\V).
\end{align} 
From \cite[Chapter 2, Proposition 1.4, part(i)]{VB2}, we know that $(\I+\lambda\Phi)^{-1}$ is nonexpansive (that is, Lipschitz with Lipschitz constant $1$) and from \cite[Chapter 2, Proposition 1.3, part (iii)]{VB2}, we have $(\I+\lambda\Phi)^{-1}(\y)\to \y$ as $\lambda\to0$ in $\H$. Therefore, we conclude 
\begin{align*}
 & \int_0^T \|(\I+\lambda\Phi)^{-1}\y_\lambda(t)-\y(t)\|_{\H}^2\d t\nonumber\\&\leq 2\int_0^T\|(\I+\lambda\Phi)^{-1}(\y_\lambda(t))-(\I+\lambda\Phi)^{-1}\y(t)\|_{\H}^2\d t+2\int_0^T\|(\I+\lambda\Phi)^{-1}\y(t)-\y(t)\|_{\H}^2\d t\nonumber\\&\leq 2\int_0^T\|\y_\lambda(t)-\y(t)\|_{\H}^2\d t+2\int_0^T\|(\I+\lambda\Phi)^{-1}\y(t)-\y(t)\|_{\H}^2\d t\nonumber\\&\to 0\ \text{ as }\ \lambda\to 0,
\end{align*}
so  that $(\I+\lambda\Phi)^{-1}(\y_\lambda)\to \y$ in $\mathrm{L}^2(0,T;\H)$ and $(\I+\lambda\Phi)^{-1}(\y_\lambda(t))\to \y(t)$, for a.e. $t\in[0,T]$ in $\H$ (along a subsequence, which is still denoted by the same).   From \cite[Chapter 2, Proposition 1.1, part (i)]{VB2} and  \cite[Proposition 1.7]{JPSS},  we know that  the maximal monotone operator $\Phi$ is weak-strong and strong-weak closed in $\H\times\H,$ that is, if  $\Phi_{\lambda}(\y_\lambda)\in\Phi(\I+\lambda\Phi)^{-1}(\y_\lambda), $  $ (\I+\lambda\Phi)^{-1}(\y_\lambda)\to \y$  in $\mathrm{L}^2(0,T;\H)$   and $\Phi_{\lambda}(\y_\lambda)\rightharpoonup\ \phi  \ \text{ in } \ \mathrm{L}^2(0,T;\H), $ then $\phi\in\Phi(\y)$ for a.e. $t\in[0,T]$  in $\H$. 

{Let us rewrite \eqref{3.29} as 
	\begin{align}\label{mb}
		\frac{\d\y_\lambda(t)}{\d t}+\mathcal{G}(\y_\lambda(t))+\Phi_{\lambda}(\y_\lambda(t))=\f(t), \ \  \text{for a.e.} \ t\in[0,T] \ \text{ in } \ \H,
	\end{align}
	where $\mathcal{G}(\cdot)=\mu\A+\mathcal{B}(\cdot)+\mathcal{C}(\cdot)$ is a quasi $m$-accretive operator as we have shown in Proposition \ref{prop33}. From \eqref{mb}, we have following energy equality:
	\begin{align}\label{mb1}
		\|\y_\lambda(t)\|_{\H}^2+2\int_0^t \left(\mathcal{G}(\y_\lambda(s)) -\f(s),\y_\lambda(s)\right)\d s +\int_0^t (\Phi_\lambda(\y_\lambda)(s),\y_\lambda(s))\d s=\|\y_0\|_{\H}^2,
	\end{align}
	for all $t\in[0,T]$. 
	From \eqref{e.15.} and \eqref{3.110.}, we have following bound:
	\begin{align*}
		\int_0^T \|\mathcal{G}(\y_\lambda)\|_{\H}^2\d t\leq C,
	\end{align*}
	which ensures that
	\begin{align}\label{mb3}
		\mathcal{G}(\y_\lambda)\rightharpoonup\mathcal{G}_0  \  \ \text{in} \ \mathrm{L}^2(0,T;\H).
	\end{align}
We now calculate by using \eqref{c3}-\eqref{c3.} and the fact that $\Phi$ is weak-strong and strong-weak closed that 
\begin{align*}
	&\left|\int_0^t (\Phi_{\lambda}(\y_\lambda(s)),\y_\lambda(s))\d s-\int_0^t (\phi(s),\y(s))\d s\right|
	\nonumber\\&\leq\left|\int_0^t (\Phi_{\lambda}(\y_\lambda(s))-\phi(s),\y_\lambda(s)-\y(s))\d s \right|+\left|\int_0^t (\Phi_{\lambda}(\y_\lambda(s))-\phi(s),\y(s))\d s \right|\nonumber\\&\quad+\left|\int_0^t (\phi(s),\y_\lambda(s)-\y(s))\d s\right|\nonumber\\&\to0  \ \ \text{as} \  \lambda\to0,
\end{align*}
which implies that 
\begin{align}\label{mb4}
	\int_0^t (\Phi_{\lambda}(\y_\lambda(s)),\y_\lambda(s))\d s\to\int_0^t (\phi(s),\y(s))\d s, \  \  \text{as} \ \lambda\to0,
\end{align}
for all $t\in(0,T)$ and $\phi\in\Phi(\y)$. Finally, on passing the limit $\lambda\to0$ in \eqref{mb} gives
\begin{equation*}
	\left\{
	\begin{aligned}
		\frac{\d\y(t)}{\d t}+\mathcal{G}_0(t)+\Phi(\y(t))&\ni\f(t), \ \text{ a.e. } \ t\in[0,T]\ \text{ in }\ \H, \\
		\y(0)&=\y_0. 
	\end{aligned}
	\right.
\end{equation*}
In order to complete the proof, it is enough to show that $\mathcal{G}_0=\mathcal{G}(\y)$. Since $\y\in\mathrm{L}^2(0,T;\V)$ and $\frac{\d\y}{\d t}\in\mathrm{L}^2(0,T;\H)\subset\mathrm{L}^2(0,T;\V')$, the following energy equality  is satisfied: 
\begin{align}\label{mb5}
	\|\y(t)\|_{\H}^2+2\int_0^t \left(\mathcal{G}_0(s)-\f(s),\y(s)\right)\d s +\int_0^t (\phi(s),\y(s)) \d s = \|\y_0\|_{\H}^2,
\end{align}
for all $t\in[0,T]$. Taking the limit supremum in \eqref{mb1} and using weakly  lower semicontinuity property of the $\|\cdot\|_{\H}$-norm, \eqref{c3}, \eqref{mb3} - \eqref{mb4}, we obtain
	\begin{align}\label{mb6}
		&\limsup\limits_{\lambda\to0} \int_0^t (\mathcal{G}(\y_\lambda(s)),\y_\lambda(s))\d s \nonumber\\&=
		\limsup\limits_{\lambda\to0}\left[\frac{1}{2}\|\y_0\|_{\H}^2-\frac{1}{2}\|\y_\lambda(t)\|_{\H}^2+\int_0^t(\f(s),\y_\lambda(s)) \d s+\int_0^t (\Phi_\lambda(\y_\lambda)(s),\y_\lambda(s))\d s\right] \nonumber\\&= \frac{1}{2}\left[\|\y_0\|_{\H}^2-\liminf\limits_{\lambda\to0}\|\y_\lambda(t)\|_{\H}^2\right]+
		\limsup\limits_{\lambda\to0}\left[\int_0^t(\f(s),\y_\lambda(s))\d s\right]\nonumber\\&\quad +\limsup\limits_{\lambda\to0} \int_0^t (\Phi_\lambda(\y_\lambda)(s),\y_\lambda(s))\d s\nonumber\\&\leq
		\frac{1}{2}\left[\|\y_0\|_{\H}^2-\|\y(t)\|_{\H}^2\right]+
		\left[\int_0^t(\f(s),\y(s))\d s\right]+\int_0^t (\phi(s),\y(s))\d s\nonumber\\&=
		\int_0^t \left(\mathcal{G}_0(s) ,\y(s)\right)\d s,
	\end{align}
	for all $t\in[0,T]$. From Proposition \ref{prop33}, we know that the nonlinear operator $\mathcal{G}(\cdot)$ is quasi $m$-accretive, so we can write
	\begin{align*}
		\int_0^T (\mathcal{G}(\v(t))-\mathcal{G}(\y_\lambda(t)),\v(t)-\y_\lambda(t))\d t+\kappa\int_0^T  
		\|\v(t)-\y_\lambda(t)\|_{\H}^2\d t\geq0.
	\end{align*}
	On taking the limit supremum and using \eqref{c3} and \eqref{mb6}, we obtain
	\begin{align*}
		\int_0^T (\mathcal{G}(\v(t))-\mathcal{G}_0(t),\v(t)-\y(t))\d t+ \kappa\limsup\limits_{\lambda\to0}\int_0^T  
		\|\v(t)-\y_\lambda(t)\|_{\H}^2\d t\geq0.
	\end{align*}
	Using the limit \eqref{c3}, one can pass the limit in the last term and this gives
	\begin{align*}
		\int_0^T (\mathcal{G}_0(t)-\mathcal{G}(\v(t)),\y(t)-\v(t))\d t+ \kappa\int_0^T  
		\|\y(t)-\v(t)\|_{\H}^2\d t\geq0,
	\end{align*}
	for all $\v\in\mathrm{L}^2(0,T;\H)$. Let us take $\v=\y+\lambda\w$ and substituting in the above expression  and dividing by $\lambda$, we get
	\begin{align}\label{444}
		\int_0^T (\mathcal{G}_0(t)-\mathcal{G}(\y(t)+\lambda\w(t)),\w(t))\d t+ \kappa\lambda\int_0^T \|\w(t)\|_{\H}^2\d t\geq0.
	\end{align}
	Thus, by taking limit $\lambda\to0$ in \eqref{444} and using the fact that $\mathcal{G}(\cdot)$ is hemicontinuous, we finally conclude that $\mathcal{G}_0=\mathcal{G}(\y).$ }
\end{proof}


\subsection{Uniqueness of solution to the problem \eqref{1p7}}
Let us now prove that the solution obtained by passing to the limit with $\lambda\to 0$  is unique.
\begin{proposition}\label{unique}
	The solution for  the problem \eqref{1p7} is unique. 
\end{proposition}
\begin{proof}
	
	Let $\y_1(\cdot)$ and $\y_2(\cdot)$ be two solutions of \eqref{1p7} satisfying \eqref{356}-\eqref{357}. Then we have  $ \text{for a.e.}  \  t\in[0,T],$
	\begin{align*}
		&\frac{1}{2}\frac{\d}{\d t} \|\y_1(t)-\y_2(t)\|_{\H}^2+\mu\|\nabla(\y_1(t)-\y_2(t))\|_{\H}^2 + (\mathcal{B}(\y_1(t))-\mathcal{B}(\y_2(t)),\y_1(t)-\y_2(t))\nonumber\\&\quad+\beta(\mathcal{C}(\y_1(t))-\mathcal{C}(\y_2(t)),\y_1(t)-\y_2(t))+(\xi_1(t)-\xi_2(t),\y_1(t)-\y_2(t))=0, 
	\end{align*}
	where $\xi_j(\cdot)\in\Phi(\y_j(\cdot)),$ for $j=1,2.$ Integrating the  above equality  and using the monotonicity of $\Phi$, we can write
	\begin{align}\label{u1}
		&\|\y_1(t)-\y_2(t)\|_{\H}^2+2\mu\int_0^t\|\nabla(\y_1(s)-\y_2(s))\|_{\H}^2 \d s\nonumber\\&\leq\|\y_1(0)-\y_2(0)\|_{\H}^2-2\int_0^t (\mathcal{B}(\y_1(s))-\mathcal{B}(\y_2(s)),\y_1(s)-\y_2(s))\d s\nonumber\\&\quad-2\beta\int_0^t (\mathcal{C}(\y_1(s))-\mathcal{C}(\y_2(s)),\y_1(s)-\y_2(s))\d s.
	\end{align}
	Using calculations similar to \eqref{2.30}-\eqref{2.27} in \eqref{u1}, we find 
	\begin{align}\label{u4}
	&\|\y_1(t)-\y_2(t)\|_{\H}^2+\mu\int_0^t\|\nabla(\y_1(s)-\y_2(s))\|_{\H}^2 \d s \nonumber\\&\leq\|\y_1(0)-\y_2(0)\|_{\H}^2+\varrho\int_0^t \|\y_1(s)-\y_2(s)\|_{\H}^2\d s,
	\end{align}
for all $t\in[0,T]$. Applying Gronwall's inequality in \eqref{u4}, we obtain for all $t\in[0,T]$
	\begin{align*}
		\|\y_1(t)-\y_2(t)\|_{\H}^2\leq\|\y_1(0)-\y_2(0)\|_{\H}^2e^{\varrho T},
	\end{align*}
which  proves the uniqueness.
\end{proof}
\subsection{Proof of Theorem \ref{thm1.2}} 
From Proposition \ref{prop3.3}, we know that the operator $\mathfrak{A}(\cdot)$ is $m$-accretive in $\H\times\H$ for sufficiently large $\kappa\geq\varrho$. So, by using the abstract theory, we obtain a unique strong solution of the problem \eqref{1p7} with the regularity \eqref{1p.4} given in Proposition \ref{thm1.1}.  While from Proposition \ref{unique}, we infer that the problem \eqref{1p7} has a unique solution satisfying the regularity 
$$\y\in\mathrm{L}^{2}(0,T;\D(\A))\cap\mathrm{L}^{r+1}(0,T;\wi\L^{3(r+1)})\cap\mathrm{W}^{1,2}(0,T;\V).$$
 Thus by the uniqueness of strong solutions, both the solution must coincide and  the solution  satisfies \eqref{1p7} in $\H$ for a.e. in $t\in[0,T]$  with the regularity given in \eqref{regu}.

\section{Applications}\label{sec5}\setcounter{equation}{0}
We discuss some applications of the results obtained in Theorem \ref{thm1.2} and Proposition \ref{thm1.1}. These include flow invariance preserving feedback controllers, a time optimal control problem and stabilizing feedback controllers for 2D and 3D CBF equations, etc. 
\subsection{Flow invariance preserving feedback controllers (\cite{VBSS})}
Let us consider the following controlled CBF equations:
\begin{equation}\label{appl1.1}
	\left\{
	\begin{aligned}
		\frac{\d \y(t)}{\d t}+\mu\A\y(t)+\mathcal{B}(\y(t))+\beta\mathcal{C}(\y(t))&= \f(t)+\mathbf{U}(t),  \ t\in(0,T],\\ 
		\y(0)&=\y_0,
	\end{aligned}
	\right.
\end{equation}	
where $\mathbf{U}(\cdot)$ is distributed control acting on the system, $\f\in\W^{1,1}(0,T;\H)$ and $\y_0\in\D(\A)$. Consider a closed and convex set $\mathcal{K}\subset\H$ such that $\boldsymbol{0}\in\mathcal{K}$ and 
\begin{align}\label{appl1..1}
	(\I+\lambda\A)^{-1}\mathcal{K}\subset\mathcal{K},  \  \text{for all} \  \lambda>0.
\end{align}
Our aim is to search for a feedback control $\mathbf{U}=\Psi(\y)$ such that $\y(t)\in\mathcal{K}, $ for all $t\in[0,T],$ if $\y_0\in\mathcal{K}.$ That is, we have to find a feedback controller for which the	set $\mathcal{K}$  is invariant with respect to CBF flow. We establish this by solving the following CBF inclusion problem: 
\begin{align}\label{appl1.2}
	\frac{\d\y(t)}{\d t}+\mu\A\y(t)+\B(\y(t))+\beta\mathcal{C}(\y(t))-\f(t)+N_{\mathcal{K}}(\y(t))\ni \mathbf{0}, \ t\in(0,T],
\end{align}
where $N_{\mathcal{K}}(\y)=\{\w\in\H:(\w,\y-\z)\geq0,\ \text{for all} \ \z\in\mathcal{K}\}$ is the well-known \textsl{Clark's normal cone} to $\mathcal{K}$ at $\y$. {We know that the multi-valued map $\Phi:=\partial\I_{\mathcal{K}}$ is a maximal monotone subset of $\H\times\H$ satisfying the Hypothesis \ref{hyp1} (see example \ref{exm})}. Therefore we can apply  Proposition \ref{thm1.1} to the  problem \eqref{appl1.2} to determine a feedback controller $\mathbf{U}\in\mathrm{L}^{\infty}(0,T;\H)$ with $$\mathbf{U}(t)\in-N_{\mathcal{K}}(\y(t))\ \mbox{ for a.e. $t\in[0,T]$},$$  which is given by 
\begin{align}\label{appl1.3}
\mathbf{U}(t)=&-\f(t)+\mu\A\y(t)+\B(\y(t))+\beta\mathcal{C}(\y(t))\nonumber\\&-(-\f(t)+\mu\A\y(t)+\B(\y(t))+\beta\mathcal{C}(\y(t))+N_{\mathcal{K}}(\y(t)))^0, \ \ \ \  \text{for a.e.} \ t\in[0,T],
\end{align} 
where $N_{\mathcal{K}}(\y)$ is the $\H$-valued normal cone to $\mathcal{K}$ at $\y.$ 

\textsl{Flow invariance for the {enstrophy} of the system.} {The enstrophy of the flow is an important quantity which determines the rate of dissipation of kinetic energy (see \cite{aanse,FMRT}). It is defined as 
\begin{align*}
	\mathcal{E}(\y):=\int_{\mathbb{T}^d} |\nabla\y(x)|^2\d x.
\end{align*}
For the incompressible flow, one can express the enstrophy in the following form:
\begin{align*}
	\mathcal{E}(\y):=\int_{\mathbb{T}^d} |\boldsymbol{\omega}(x)|^2\d x,
\end{align*}
 where $\boldsymbol{\omega}:=\nabla\times\y$ is the vorticity vector.} We consider the constraint set
\begin{align*}
	\mathcal{K}=\{\y\in\V:\|\nabla\times\y\|_{\H}=\|\nabla\y\|_{\H}\leq \varpi \}.
\end{align*}
Let $\g$ be any arbitrary element of $\mathcal{K}$ such that $\y+\lambda\A\y=\g,$ which has a unique strong solution for all $\g\in\V$.  Taking the inner product with $\y$ and using the Cauchy-Schwarz and Young's inequalities, we find
\begin{align*}
	\|\y\|_{\H}^2+\lambda\|\nabla\y\|_{\H}^2\leq\frac{1}{2}\|\y\|_{\H}^2+\frac{1}{2}\|\g\|_{\H}^2\Rightarrow
	\|\y\|_{\H}\leq\|\g\|_{\H},
	\end{align*}
for all $\lambda>0$. Taking the inner product with $\A\y$,  we obtain 
\begin{align*}
\|\nabla\y\|_{\H}^2+\lambda\|\A\y\|_{\H}^2&\leq\frac{1}{2}\|\nabla\g\|_{\H}^2+\frac{1}{2}\|\nabla\y\|_{\H}^2\Rightarrow
\|\nabla\y\|_{\H}\leq\|\nabla\g\|_{\H},
\end{align*} 
 for all $\lambda>0.$ Thus from the definition of $\mathcal{K},$ we have $\y\in\mathcal{K}$ and this imply $(\I+\lambda\A)^{-1}\mathcal{K}\subset\mathcal{K}.$ Our aim is to find a feedback control so that enstrophy of the system can be  kept inside this constraint set $\mathcal{K}.$  The normal cone corresponding to the convex set $\mathcal{K}$ is
\begin{align*}
	N_{\mathcal{K}}(\y)=
\begin{cases}
	\boldsymbol{0}, &\text{ if } \ \|\nabla\y\|_{\H}<\varpi ,\\
	\bigcup\limits_{\lambda>0} \lambda\A\y, &\text{ if } \ \|\nabla\y\|_{\H}=\varpi .
\end{cases}
\end{align*}
The feedback control is given by
$\mathbf{U}(t)\in-N_{\mathcal{K}}(\y(t))$ for a.e. $t\in[0,T].$ For $\|\nabla\y\|_{\H}<\varpi$,
that is, when the flow remain inside the constraint set $\mathcal{K}$, we have  $\mathbf{U}(t)=\boldsymbol{0}.$
For  $\|\nabla\y\|_{\H}=\varpi$,  
\begin{align}\label{appl1.4}
\mathbf{U}(t)=-\lambda_0\A\y(t), \  \text{for a.e.} \ t\in[0,T],
\end{align}
 for some $\lambda_0>0$.  Then from \eqref{appl1.3}, we have for a.e. $t\in[0,T]$
\begin{align}\label{56}
\mathbf{U}(t)=&-\f(t)+\mu\A\y(t)+\B(\y(t))+\beta\mathcal{C}(\y(t))+	\frac{\d^+ \y(t)}{\d t}.
\end{align}
{Since $\y\in\C([0,T];\V)$, by an application of the Lebesgue dominated convergence theorem, we calculate for all $t\in[0,T]$ (cf. \eqref{1p5})
	\begin{align}\label{ap2}
		\frac{\d^+}{\d t}\|\nabla\y(t)\|_{\H}^2&= \lim\limits_{h\to0}\frac{\|\nabla\y(t+h)\|_{\H}^2-\|\nabla\y(t)\|_{\H}^2}{h}\nonumber\\&=
		\lim\limits_{h\to0}\frac{(\nabla\y(t+h)-\nabla\y(t),\nabla\y(t+h))+(\nabla\y(t),\nabla\y(t+h)-\nabla\y(t))}{h}\nonumber\\&=\left(\lim\limits_{h\to0}\frac{(\nabla\y(t+h)-\nabla\y(t))}{h},\nabla\y(t)\right)+\left(\nabla\y(t),\lim\limits_{h\to0}\frac{(\nabla\y(t+h)-\nabla\y(t))}{h}\right)\nonumber\\&=2\left(\frac{\d^+(\nabla\y(t))}{\d t},\nabla\y(t)\right).
	\end{align}
Since $\y\in\mathrm{L}^{\infty}(0,T;\D(\A))$, in a similar way, we estimate for a.e. $t\in[0,T]$
\begin{align}\label{ap1}
	\left(\frac{\d^+ \y(t)}{\d t},\A\y(t)\right)&= \left(\lim\limits_{h\to0}\frac{\y(t+h)-\y(t)}{h},\A\y(t)\right)=
	\lim\limits_{h\to0}\left(\frac{\y(t+h)-\y(t)}{h},\A\y(t)\right)\nonumber\\&=\lim\limits_{h\to0}\left(\frac{\nabla\y(t+h)-\nabla\y(t)}{h},\nabla\y(t)\right)\nonumber\\&=\left(\lim\limits_{h\to0}\frac{\nabla\y(t+h)-\nabla\y(t)}{h},\nabla\y(t)\right)\nonumber\\&=\left(\frac{\d^+(\nabla\y(t))}{\d t},\nabla\y(t)\right).
\end{align}
Thus, from \eqref{ap1}-\eqref{ap2}, we get $	\left(\frac{\d^+ \y(t)}{\d t},\A\y(t)\right)=
\frac{1}{2}\frac{\d^+}{\d t}\|\nabla\y(t)\|_{\H}^2$, for a.e. $t\in[0,T]$ and  for $\|\nabla\y(t)\|_{\H}=\varpi,$ we have 
$
	\left(\frac{\d^+ \y(t)}{\d t},\A\y(t)\right)=0,
$
for a.e. $t\in[0,T]$.   Taking the inner product with $\A\y$ in \eqref{56} and using \eqref{appl1.4}, we obtain
\begin{align*}
	\lambda_0=\frac{-1}{\|\A\y\|_{\H}^2}\left[(\f,\A\y)-\mu\|\A\y\|_{\H}^2-b(\y,\y,\A\y)-(\mathcal{C}(\y),\A\y)\right]. 
\end{align*}
Therefore the feedback control becomes
\begin{align*}
	\mathbf{U}(t)=\frac{-\A\y(t)}{\|\A\y(t)\|_{\H}^2}\left\{(\f(t),\A\y(t))-\mu\|\A\y(t)\|_{\H}^2-b(\y(t),\y(t),\A\y(t))-(\mathcal{C}(\y(t)),\A\y(t))\right\},
\end{align*} 
for a.e. $t\in[0,T]$.} Thus, for $\y_0\in\D(\A)\cap\mathcal{K}$ and $\f\in\W^{1,1}(0,T;\H)$, and the feedback control given above, the closed loop problem \eqref{appl1.1} has a unique strong solution $\y\in\W^{1,\infty}(0,T;\H)\cap\C([0,T];\V)\cap\mathrm{L}^{\infty}(0,T;\D(\A))$ which satisfies (Corollary 2.54, \cite{VB22}) $$\y(t)\in\mathcal{K},\ \text{ for all }\ t\in[0,T].$$ We refer the interested readers to \cite{VBSS} for a discussion on some other important flow invariance problems like localized dissipation, pointwise velocity constraints, pointwise vorticity contraint, helicity invariance, etc.

\subsection{A time optimal control problem (\cite{TiOp1,TiOp2})}\label{TO} Let us discuss the following time optimal control for  CBF equations 
\begin{equation}\label{appl2.1}
	\left\{
	\begin{aligned}
		\frac{\d \y(t)}{\d t}+\mu\A\y(t)+\mathcal{B}(\y(t))+\beta\mathcal{C}(\y(t))&= \mathbf{U}(t),  \ \text{for a.e.} \ t>0, \ \text{in} \ \H,\\ 
		\y(0)&=\y_0,
	\end{aligned}
	\right.
\end{equation}	
Let $\kappa>0$ and we define the \textsl{class of controls}
\begin{align*}
	\mathcal{U}_\kappa=\{\mathbf{U}(\cdot)\in\mathrm{L}^{\infty}(\R^+;\H):\|\mathbf{U}(t)\|_{\H}\leq\kappa, \  \text{a.e.} \ t>0\}.
\end{align*} 
Let $\y_0,\y_1\in\D(\A)$ be arbitrary but fixed. A control $\mathbf{U}(\cdot)\in\mathcal{U}_\kappa$ is said to be admissible if it steers from the (initial state) $\y_0$ to the (target) $\y_1$ in a finite time $T$ along the trajectory $\y(t;\y_0,\mathbf{U}(\cdot))$ of \eqref{appl2.1} which starts from $\y_0.$ We assume that the class of all such controls (\textsl{admissible class}) is nonempty. Let $T(\y_0,\y_1)$ be the infimum of all such times and it is called \textsl{minimal time}, that is,
$	T(\y_0,\y_1):= \inf\limits_{T\in\R^+}\{T:\y(T;\y_0,\mathbf{U}(\cdot))=\y_1,\mathbf{U}(\cdot)\in\mathcal{U}_\kappa\}.$

A control $\mathbf{U}^*(\cdot)$ such that $\y(T(\y_0,\y_1);\y_0,\mathbf{U}^*(\cdot))=\y_1$ is called \textsl{time optimal control} and the time $T(\y_0,\y_1)$ is said to be \textsl{optimal time.} The pair $(\y^*,\mathbf{U}^*)$ is called the \textsl{time optimal pair}, where $\y^*=\y(t;\y_0,\mathbf{U}^*).$ 
We define a multi-valued operator $\mathrm{sgn}:\H\to\H$ by
\begin{align*}
	\mathrm{sgn}(\y)=
	\begin{cases}
		\frac{\y}{\|\y\|_{\H}}, &\text{ if } \ \y\neq\mathbf{0},\\
		\{\z\in\H:\|\z\|_{\H}\leq 1\}, &\text{ if } \ \y=\mathbf{0},
	\end{cases}
\end{align*}
which is the subdifferential of $\|\y\|_{\H}$ and hence it is maximal monotone in $\H\times\H$ (\cite[Theorem 2.1, Chapter 2]{VB2}). From \cite[Proposition 2.4, Chapter 4]{VB2}, the Yosida approximation of $\Theta:=\kappa\ \mathrm{sgn}(\cdot)$ is given by 
\begin{align*}
	\Theta_\lambda(\y)=\frac{1}{\lambda}\left(\y-\left(\I+\lambda\Theta\right)^{-1}\y\right)=
	\begin{cases}
		\frac{\kappa\y}{\|\y\|_{\H}}, &\text{ if } \  \|y\|_{\H}\geq\lambda,\\
		\frac{\kappa}{\lambda}\y, &\text{ if } \  \|y\|_{\H}<\lambda.
	\end{cases}
\end{align*}
From the above definition, we conclude that 
\begin{align*}
	(\A\y,\Theta_\lambda(\y-\y_1))\geq0,   \  \text{ for all }  \  \y\in\D(\A), \ \lambda>0,
\end{align*}
and therefore all the assumptions of Hypothesis \ref{hyp1} are satisfied. Thus we can apply Proposition \ref{thm1.1} for the system
\begin{equation}\label{appl2.2}
	\left\{
	\begin{aligned}
		\frac{\d \y(t)}{\d t}+\mu\A\y(t)+\mathcal{B}(\y(t))+\beta\mathcal{C}(\y(t))+\kappa(\mathrm{sgn}(\y(t)-\y_1))&\ni\boldsymbol{0},  \ \text{for a.e.} \ t>0, \\ 
		\y(0)&=\y_0.
	\end{aligned}
	\right.
\end{equation}	
Then the feedback law $\mathbf{U}(t)\in-\kappa(\mathrm{sgn}(\y(t)-\y_1)),$ for $t>0,$ ensures the existence of an admissible control $\mathbf{U}(\cdot)\in\mathcal{U}_\kappa$ for the system \eqref{appl2.1}, under the assumption that 
\begin{align}\label{appl2.3*}
	\|\mu\A\y_1+\mathcal{B}(\y_1)+\beta\mathcal{C}(\y_1)\|_{\H}<\kappa, 
\end{align}
and 
\begin{align}\label{appl2.3}
	\|\y_0-\y_1\|_{\H}\leq\frac{\kappa-\|\mu\A\y_1+\mathcal{B}(\y_1)+\beta\mathcal{C}(\y_1)\|_{\H}}{\varrho},
\end{align} 
for $\y_0,\y_1\in\D(\A),$ where $\varrho$ is given as in $\eqref{prop33}.$ In order to prove this, we  show that the system \eqref{appl2.2} has finite extinction property in $\H,$ that is, $\y(T)=\y_1$ for some $T>0$ (see \cite[section 5.3, Chapter 5]{VB2}). 
Let us set $\z(\cdot)=\y(\cdot)-\y_1.$ Then $\z(\cdot)$ satisfies 
\begin{equation*}
	\left\{
	\begin{aligned}
		\frac{\d \z(t)}{\d t} +\mu\A\z(t) &+\mathcal{B}(\z(t)+\y_1)-\mathcal{B}(\y_1) +\beta(\mathcal{C}(\z(t)+\y_1)-\mathcal{C}(\y_1))\nonumber\\+ \kappa\ \mathrm{sgn}(\z(t))&\ni-(\mu\A\y_1+\mathcal{B}(\y_1)+\beta\mathcal{C}(\y_1)),\ \text{for a.e.} \ t>0, \\ 
		\z(0)&=\y_0-\y_1.
	\end{aligned}
	\right.
\end{equation*}	
We assume that there exists no $T$ such that $\z(T)=\boldsymbol{0}.$ 
Taking the inner product with $\mathrm{sgn}(\z(\cdot))$ (using a smooth approximation of $\mathrm{sgn}(\z(\cdot))$ \cite{VB2}, one can justify), we get
\begin{align*}
	&\frac{1}{2}\frac{\d }{\d t}\|\z(t)\|_{\H}^2+\mu\|\nabla\z(t)\|_{\H}^2+\kappa\|\z(t)\|_{\H}+(\mathcal{C}(\z(t)+\y_1)-\mathcal{C}(\y_1),\z(t))\nonumber\\&=(\mathcal{B}(\y_1)-\mathcal{B}(\z(t)+\y_1),\z(t))-
	(\mu\A\y_1+\mathcal{B}(\y_1)+\beta\mathcal{C}(\y_1),\z(t)).
\end{align*}
From \eqref{2.30}, \eqref{2.23} and using a calculation similar to \eqref{Z6..}, we obtain
\begin{align*}
	\frac{1}{2}\frac{\d }{\d t}\|\z(t)\|_{\H}^2+\frac{\mu}{2}\|\nabla\z(t)\|_{\H}^2+\kappa\|\z(t)\|_{\H}\leq
	\varrho\|\z(t)\|_{\H}^2+\|\mu\A\y_1+\mathcal{B}(\y_1)+\beta\mathcal{C}(\y_1)\|_{\H}\|\z(t)\|_{\H},
\end{align*}
and we can rewrite $\frac{\d }{\d t}\|\z(t)\|_{\H}+\eta\leq\varrho\|\z(t)\|_{\H},$ where $\eta=\kappa-\|\mu\A\y_1+\mathcal{B}(\y_1)+\beta\mathcal{C}(\y_1)\|_{\H}>0$ and $\varrho$ is given in $\eqref{prop33}.$
By using variation of constant formula, we get 
\begin{align*}
	e^{-\varrho t}\|\z(t)\|_{\H}\leq\left(\|\z(0)\|_{\H}-\frac{\eta}{\varrho}\right)+\frac{\eta}{\varrho}	e^{-\varrho t}.
\end{align*}
This shows as $t\to\infty$ we are getting contradiction to the assumption \eqref{appl2.3}. This implies that $\z=\z(t)$ has finite extinction property in time $T>0$ and this proves the existence of an admissible control $\mathbf{U}(\cdot)\in\mathcal{U}_{\kappa}.$ The first order necessary and second order necessary and sufficient conditions of optimality will be discussed in a future work.

\subsection{Stabilizing feedback controllers}\label{St}
Let us consider the following  controlled CBF equations:
\begin{equation}\label{appl3.1}
	\left\{
	\begin{aligned}
		\frac{\d \y(t)}{\d t}+\mu\A\y(t)+\mathcal{B}(\y(t))+\beta\mathcal{C}(\y(t))&= \f_e+\mathbf{U}(t),  \ \text{for a.e.} \ t>0,\\ 
		\y(0)&=\y_0.
	\end{aligned}
	\right.
\end{equation}	
Let $\y_e\in\D(\A)$ be the steady-state (equilibrium) solution of \eqref{appl3.1}, that is, $\y_e$ satisfies
\begin{equation}\label{appl3.2}
	\mu\A\y_e+\mathcal{B}(\y_e)+\beta\mathcal{C}(\y_e)= \f_e \  \text{ in } \ \mathbb{T}^{d},
\end{equation}	
whose solvability results are available in \cite[Theorem 4.1]{MT4}. Let $\mathcal{K}\subset\H$ be a closed and convex set with $\boldsymbol{0}\in\mathcal{K}$ such that \eqref{appl1..1} is satisfied. We set $\z(\cdot)=\y(\cdot)-\y_e,$ then \eqref{appl3.1} becomes $\text{for a.e.} \ t>0,$ in $\H$
\begin{equation}\label{appl3.3}
	\left\{
	\begin{aligned}
		\frac{\d \z(t)}{\d t}+\mu\A\z(t)+\mathcal{B}(\z(t)+\y_e)-\mathcal{B}(\y_e)+\beta\mathcal{C}(\z(t)+\y_e)-\mathcal{C}(\y_e)&= \mathbf{U}(t),\\ 
		\z(0)&=\y_0-\y_e.
	\end{aligned}
	\right.
\end{equation}	
Let $\wi{\mathcal{B}}(\z(\cdot)):=\mathcal{B}(\z(\cdot)+\y_e)-\mathcal{B}(\y_e)$ and $\wi{\mathcal{C}}(\z(\cdot)):=\mathcal{C}(\z(\cdot)+\y_e)-\mathcal{C}(\y_e).$ Then \eqref{appl3.3} becomes
\begin{equation}\label{appl3.3.}
	\left\{
	\begin{aligned}
		\frac{\d \z(t)}{\d t} +\mu\A\z(t)+\wi{\mathcal{B}}(\z(t))+\wi{\mathcal{C}}(\z(t))&= \mathbf{U}(t),  \ \mbox{for a.e. } \ t>0 \ \mbox{ in }\ \H, \\ 
		\z(0)&=\y_0-\y_e.
	\end{aligned}
	\right.
\end{equation}	
Using Step IV in the proof of Proposition \ref{prop33}, it is clear that  $\wi{\mathcal{B}}(\cdot)$ and $\wi{\mathcal{C}}(\cdot)$ map from $\D(\A)$ to $\H$.  Therefore the operator $\mu\A+\wi{\mathcal{B}}(\cdot)+\beta\wi{\mathcal{C}}(\cdot)+\theta\I+\partial\I_{\mathcal{K}}(\cdot)$ is $m$-accretive in $\H\times\H$ for $\theta>0$ is sufficiently large. Let $\Phi(\w):=\theta\w+\partial\I_{\mathcal{K}}(\w),$ for all $\w\in\H.$ Since $\D(\partial\I_{\mathcal{K}})=\mathcal{K}$ and $\mathcal{K}\subset\H$, then from \cite[Chapter 2, Theorem 2.3]{VB2}, we can write for all $\w\in\H$  $$\Phi(\w)=\partial\left(\frac{\theta}{2}\|\w\|_{\H}^2+\I_{\mathcal{K}}(\w)\right).$$ Clearly $\mathbf{0}\in\D(\Phi)=\mathcal{K}$.  Also, from \cite[Chapter IV, Proposition 1.1, part (iv)]{VB1}, we have  $$(\Phi(\w),\A\w)\geq0.$$ Thus the Hypothesis \ref{hyp1} is satisfied.
Therefore, we can apply the existence and uniqueness result (see Proposition \ref{thm1.1}) for the inclusion problem in $\H\times\H$ 
\begin{equation}\label{appl3.4}
	\left\{
	\begin{aligned}
		\frac{\d \z(t)}{\d t} +\mu\A\z(t)+\wi{\mathcal{B}}(\z(t))+\wi{\mathcal{C}}(\z(t))+\theta\z(t)+\partial\I_{\mathcal{K}}(\z(t))&\ni \boldsymbol{0}, \ \text{for a.e.} \ t>0,\\ 
		\z(0)&=\y_0-\y_e.
	\end{aligned}
	\right.
\end{equation}	
 We intend to find a feedback controller $\mathbf{U}(\cdot)$ given by $\mathbf{U}(t)\in-\theta\z(t)-\partial\I_{\mathcal{K}}(\z(t)$ which statbilizes the equilibrium solution $\y_e$ exponentially under the invariance condition that $\y_0-\y_e\in\mathcal{K},$ then $\y(t)-\y_e\in\mathcal{K}$ for all $t\geq0.$ The stability part will be discussed in a future work.

\begin{appendix}
	\renewcommand{\thesection}{\Alph{section}}
	\numberwithin{equation}{section}
 \section{The case of $d=2,3$ and $r\in[1,3]$} \label{Appen.}\setcounter{equation}{0}
 
 The case of $d=2,3$ and $r\in[1,3]$ is considered in this section. We quantize the Navier-Stokes nonlinearity $\mathcal{B}(\cdot)$  and prove  monotonicity property. The authors in \cite{VBSS} took a $\V$-ball for quantization, while we are taking an $\wi\L^4$-ball.    Define the quantized nonlinearity as
\begin{equation}\label{Q1}
	\mathcal{B}_{\mathrm{N}}(\y) =
	\begin{cases}
		\mathcal{B}(\y), \   & \text{ if } \  \|\y\|_{\widetilde{\L}^{4}} \leq \mathrm{N},\\
		\biggl(\frac{\mathrm{N}}{\|\y\|_{\widetilde{\L}^{4}}}\biggr)^{4}\mathcal{B}(\y), \   &\text{ if } \  \|\y\|_{\widetilde{\L}^{4}} > \mathrm{N},
	\end{cases}
\end{equation}
where $\mathrm{N}\in\N^*:=\mathbb{N}\cup\{0\}.$ 
\begin{lemma}\label{lem3.1}
	The operator $\mathcal{B}_{\mathrm{N}}(\cdot):\V\to\V'$ satisfies
	\begin{align}\label{3.3}
		|\langle \mathcal{B}_{\mathrm{N}}(\y)-\mathcal{B}_{\mathrm{N}}(\z), \y-\z \rangle|\leq\frac{\mu}{2}\|\nabla(\y-\z)\|_{\H}^{2}+C_\mathrm{N}\|\y-\z\|_{\H}^{2}, \ \text{ for all } \ \y, \z\in\V,
	\end{align}
	where $\mu>0$ is the same as in the system \eqref{1}.
\end{lemma}
\begin{proof}
	\vskip 2mm
	\noindent
	Without loss of generality, one may assume that $\|\y\|_{\wi\L^4}\leq \|\z\|_{\wi\L^4}. $	Therefore, we need to consider the following three cases: 
	\vskip 2mm
	\noindent
	\textsl{Case I}: \textit{$\|\y\|_{\widetilde{\L}^{4}},\|\z\|_{\widetilde{\L}^{4}} \leq \mathrm{N}.$} Using \eqref{b0},   H\"older's, Ladyzhenskaya's and Young's inequalities, we have 
	\begin{align}\label{35}
		|\langle \mathcal{B}_{\mathrm{N}}(\y)-\mathcal{B}_{\mathrm{N}}(\z), \y-\z \rangle| &=|\langle\mathcal{B}(\y)-\mathcal{B}(\z),\y-\z\rangle|= |\langle\mathcal{B}(\y-\z),\z\rangle|\nonumber\\&\leq C\|\y-\z\|_{\widetilde{\L}^{4}} \|\nabla(\y-\z)\|_{\H} \|\z\|_{\widetilde{\L}^{4}}\leq \frac{\mu}{2}\|\nabla(\y-\z)\|_{\H}^{2}+C_\mathrm{N}\|\y-\z\|_{\H}^{2}.
	\end{align}
	\vskip 2mm
	\noindent
	\textsl{Case II: $\|\y\|_{\widetilde{\L}^{4}},\|\z\|_{\widetilde{\L}^{4}} > \mathrm{N}.$} Let us first consider 
	\begin{align}\label{e0}
	&	\langle \mathcal{B}_{\mathrm{N}}(\y)-\mathcal{B}_{\mathrm{N}}(\z), \y-\z \rangle\nonumber\\& = \biggl \langle\biggl(\frac{\mathrm{N}}{\|\y\|_{\widetilde{\L}^{4}}}\biggr)^{4}\mathcal{B}(\y)-\biggl(\frac{\mathrm{N}}{\|\z\|_{\widetilde{\L}^{4}}}\biggr)^{4}\mathcal{B}(\z),\y-\z\biggl\rangle\nonumber \\&=\biggl[\biggl(\frac{\mathrm{N}}{\|\y\|_{\widetilde{\L}^{4}}}\biggr)^{4}-\biggl(\frac{\mathrm{N}}{\|\z\|_{\widetilde{\L}^{4}}}\biggr)^{4}\biggr]\langle\mathcal{B}(\y),\y-\z\rangle+\biggl(\frac{\mathrm{N}}{\|\z\|_{\widetilde{\L}^{4}}}\biggr)^{4} \langle \mathcal{B}(\y)-\mathcal{B}(\z) , \y-\z \rangle. 
	\end{align}
	From  \eqref{b0}, Taylor's formula, H\"older's, Ladyzhenskaya's and Young's inequalities, we obtain 
	\begin{align*}
		&	\biggl|\biggl(\frac{\mathrm{N}}{\|\y\|_{\widetilde{\L}^{4}}}\biggr)^{4}-	\biggl(\frac{\mathrm{N}}{\|\z\|_{\widetilde{\L}^{4}}}\biggr)^{4}\biggr| |\langle \mathcal{B}(\y),\y-\z \rangle|\nonumber\\&\leq 4\left(\biggl(\frac{\mathrm{N}}{\|\y\|_{\widetilde{\L}^{4}}}\biggr)+\biggl(\frac{\mathrm{N}}{\|\z\|_{\widetilde{\L}^{4}}}\biggr)\right)^3\left|\frac{\mathrm{N}}{\|\y\|_{\widetilde{\L}^{4}}}-\frac{\mathrm{N}}{\|\z\|_{\widetilde{\L}^{4}}}\right||\langle\B(\y,\y-\z),\z\rangle|\nonumber\\&\leq C\mathrm{N}\|\y-\z\|_{\wi\L^4}\|\nabla(\y-\z)\|_{\H}\leq \frac{\mu}{4}\|\nabla(\y-\z)\|_{\H}^{2}+C_\mathrm{N}\|\y-\z\|_{\H}^{2}.
	\end{align*}
	A calculation similar to \eqref{35} yields 
	\begin{align}
		&	\left|\biggl(\frac{\mathrm{N}}{\|\z\|_{\widetilde{\L}^{4}}}\biggr)^{4} \langle \mathcal{B}(\y)-\mathcal{B}(\z) , \y-\z \rangle\right|\leq \frac{\mu}{4}\|\nabla(\y-\z)\|_{\H}^{2}+C_\mathrm{N}\|\y-\z\|_{\H}^{2}.
	\end{align}
	Combining the above estimates imply \eqref{3.3}. 
	
	\vskip 2mm
	\noindent
	\textsl{Case III: $\|\y\|_{\widetilde{\L}^{4}}\leq\mathrm{N}$ and $\|\z\|_{\widetilde{\L}^{4}} > \mathrm{N}.$} One can rewrite 
	\begin{align*}
		\langle \mathcal{B}_{\mathrm{N}}(\y)-\mathcal{B}_{\mathrm{N}}(\z), \y-\z \rangle& = \biggl \langle\mathcal{B}(\y)-\biggl(\frac{\mathrm{N}}{\|\z\|_{\widetilde{\L}^{4}}}\biggr)^{4}\mathcal{B}(\z),\y-\z\biggr\rangle\nonumber \\&=\biggl[1-\biggl(\frac{\mathrm{N}}{\|\z\|_{\widetilde{\L}^{4}}}\biggr)^{4}\biggr]\langle\mathcal{B}(\y),\y-\z\rangle+\biggl(\frac{\mathrm{N}}{\|\z\|_{\widetilde{\L}^{4}}}\biggr)^{4} \langle \mathcal{B}(\y)-\mathcal{B}(\z) , \y-\z \rangle.
	\end{align*}
	As  $1-\biggl(\frac{\mathrm{N}}{\|\z\|_{\widetilde{\L}^{4}}}\biggr)^{4}=\frac{\|\z\|_{\widetilde{\L}^{4}}^{4}-\mathrm{N}^{4}}{\|\z\|_{\widetilde{\L}^{4}}^{4}}\leq\frac{\|\z\|_{\widetilde{\L}^{4}}^{4}-\|\y\|_{\widetilde{\L}^{4}}^{4}}{\|\z\|_{\widetilde{\L}^{4}}^{4}}$, one can use the estimates in the previous cases to conclude \eqref{3.3}. 
\end{proof}
The following results regarding $m$-accretivity of the operators can be proved in a similar fashion as we prove Propositions \ref{prop33}-\ref{prop3.3}.
\begin{proposition}\label{prop3.2}
	For $d=2,3$ and $1\leq r\leq 3$, define the  operator $\Upsilon_{\mathrm{N}}:\mathrm{D}(\Upsilon_{\mathrm{N}})\to\H$  by $$\Upsilon_{\mathrm{N}}(\cdot):= \mu\mathrm{A}+\mathcal{B}_{\mathrm{N}}(\cdot)+\beta\mathcal{C}(\cdot), \  \text{with} \  \ \mathrm{D}(\Upsilon_{\mathrm{N}})=\mathrm{D}(\mathrm{A}).$$ Then there exists $\eta_{\mathrm{N}}>0$ such that $\Upsilon_{\mathrm{N}}+\eta_{\mathrm{N}}\mathrm{I}$ is $m$-accretive in $\H\times\H.$
\end{proposition}

\begin{proposition}\label{propA3}
	Let $\mathrm{N}\in\N^{*}$ be fixed. Let $\Phi\subset\H\times\H$ be a maximal monotone operator satisfying Hypothesis \ref{hyp1}. Define the multi-valued operator  $\mathfrak{A}_\mathrm{N}:\mathrm{D}(\mathfrak{A}_\mathrm{N})\to\H$ by 
	\begin{equation*}
		\mathfrak{A}_\mathrm{N}(\cdot) = \mu\mathrm{A} +\mathcal{B}_\mathrm{N}(\cdot)+\beta\mathcal{C}(\cdot)+\Phi(\cdot)+\eta_\mathrm{N}\I
	\end{equation*}
	with domain $\mathrm{D}(\mathfrak{A}_\mathrm{N})=\{y\in\H: \varnothing \neq \mathfrak{A}_\mathrm{N}(\y)\subset\H\}$. Then $\mathrm{D}(\mathfrak{A}_\mathrm{N}) = \mathrm{D}(\mathrm{A})\cap\mathrm{D}(\Phi)$ and $\mathfrak{A}_\mathrm{N}$ is a maximal monotone operator in $\H\times\H,$ where $\eta_{\mathrm{N}}$ is as in Proposition \ref{prop3.2}. 
	Moreover, there exists a constant $C$ such that 
	\begin{align}\label{111}
		\|\mathrm{A}\w\|_{\H}^{2}&\leq C(1+\|\w\|_{\H}^{2}+\|\mu\mathrm{A}\w +\mathcal{B}_{\mathrm{N}}(\w)+\beta\mathcal{C}(\w)+\Phi_{\lambda}(\w)\|_{\H}^{2})^{3},
	\end{align}
	for every $\w\in\mathrm{D}(\mathrm{A}) \ \text{and for every} \ \lambda>0.$ Furthermore, we have 
	\begin{align}\label{222}
		\|\mathrm{A}\w\|_{\H}^{2}&\leq C(1+\|\w\|_{\H}^{2}+\|\mu\mathrm{A}\w +\mathcal{B}_{\mathrm{N}}(\w)+\beta\mathcal{C}(\w)+\xi\|_{\H}^{2})^{3},
	\end{align}
	for every $\w\in\mathrm{D}(\mathrm{A})\cap\mathrm{D}(\Phi)$ and for every $\xi\in\Phi(\w).$
\end{proposition}


Let us now consider the following approximate equation: 
\begin{equation}\label{apq}
	\left\{
	\begin{aligned}
		\frac{\d \y_{\mathrm{N}}(t)}{\d t}+\mu\A\y_{\mathrm{N}}(t)+\B_{\mathrm{N}}(\y_{\mathrm{N}}(t))+\beta\mathcal{C}(\y_{\mathrm{N}}(t))+\Phi(\y_{\mathrm{N}}(t))&\ni \f(t), \ \text{ a.e. } \ t\in[0,T], \\
		\y_{\mathrm{N}}(0)&=\y_0. 
	\end{aligned}
	\right.
\end{equation}

Using Proposition \ref{propA3}, one can establish the following result similar to Proposition \ref{prop3.5}. 

\begin{proposition}
	Let $\Phi\subset \H\times \H$ satisfy Hypothesis \ref{hyp1}. Let $\f\in\W^{1,1}(0,T;\H)$ and $\y_0\in\D(\A)\cap\D(\Phi)$. Then there exists a unique strong solution 
	\begin{align*}\y_{\mathrm{N}}\in \W^{1,\infty}(0,T;\H)\cap \mathrm{L}^{\infty}(0,T;\D(\A))\cap \C([0,T];\V)\end{align*}to the problem \eqref{apq}. Furthermore, $\y_{\mathrm{N}}$ is right differentiable, $\frac{\d^+\y_{\mathrm{N}}}{\d t} $ is right continuous, and 
	\begin{equation*}
		\frac{\d^+ \y_{\mathrm{N}}(t)}{\d t}+\left(\mu\A\y_{\mathrm{N}}(t)+\B_{\mathrm{N}}(\y_{\mathrm{N}}(t))+\beta\mathcal{C}(\y_{\mathrm{N}}(t))+\Phi(\y_{\mathrm{N}}(t))-\f(t)\right)^0= \mathbf{0}, \ \text{ for all  } \ t\in[0,T]. 
	\end{equation*}
\end{proposition}
\begin{proposition}
	Let $\Phi\subset \H\times \H$ satisfy Hypothesis \ref{hyp1}. Let $\f\in\W^{1,1}(0,T;\H)$ and $\y_0\in\D(\A)\cap\D(\Phi)$. Then there exists a unique strong solution 
	\begin{align*}
		\y_{\mathrm{N}}^{\lambda}\in \W^{1,\infty}(0,T;\H)\cap \mathrm{L}^{\infty}(0,T;\D(\A))\cap \C([0,T];\V)
	\end{align*}
	to the problem 
	\begin{equation*}
		\left\{
		\begin{aligned}
			\frac{\d\y_{\mathrm{N}}^{\lambda}(t)}{\d t}+\mu\A\y_{\mathrm{N}}^{\lambda}(t)+\B_{\mathrm{N}}(\y_{\mathrm{N}}^{\lambda}(t))+\beta\mathcal{C}(\y_{\mathrm{N}}^{\lambda}(t))+\Phi_{\lambda}(\y_{\mathrm{N}}^{\lambda}(t))&=\f(t), \ \text{ a.e. } \ t\in[0,T], \\
			\y_{\mathrm{N}}^{\lambda}(0)&=\y_0. 
		\end{aligned}
		\right.
	\end{equation*}
	
	Furthermore, $\y_{\mathrm{N}}^{\lambda}$ is right differentiable, $\frac{\d^+\y_{\mathrm{N}}^{\lambda}}{\d t} $ is right continuous, and 
	\begin{equation*}
		\frac{\d^+\y_{\mathrm{N}}^{\lambda}(t)}{\d t}+\mu\A\y_{\mathrm{N}}^{\lambda}(t)+\B_{\mathrm{N}}(\y_{\mathrm{N}}^{\lambda}(t))+\beta\mathcal{C}(\y_{\mathrm{N}}^{\lambda}(t))+\Phi_{\lambda}(\y_{\mathrm{N}}^{\lambda}(t))= \f(t), \ \text{ for all  } \ t\in[0,T). 
	\end{equation*}
\end{proposition}

\begin{theorem}\label{thm12}
	Let $T>0$ and assume that $\Phi\subset\H\times\H$ satisfies Hypothesis \ref{hyp1}. Let $\y_0\in\V\cap\D(\Phi)$ and $\f \in \mathrm{L}^2(0,T;\H)$. For $d=2$ with $r\in[1,3]$, there exists a unique strong solution $$\y\in \C([0,T];\V)\cap\mathrm{L}^{2}(0,T;\D(\A))\cap\mathrm{W}^{1,2}(0,T;\H)$$   such that in $\H$
	\begin{equation}\label{17}
		\left\{
		\begin{aligned}
			\frac{\d \y(t)}{\d t}+\mu\A\y(t)+\mathcal{B}(\y(t))+\beta\mathcal{C}(\y(t))+\Phi(\y(t))&\ni \f(t), \ \text{ a.e. } \ t\in[0,T], \\
			\y(0)&=\y_0. 
		\end{aligned}
		\right.
	\end{equation}
	For $d=3$ with $r\in[1,3]$ and $2\beta\mu<1$ for $r=3$, there exists a time $$T_0=T_0\left(\|\y_0\|_{\V},\|\f\|_{\mathrm{L}^2(0,T;\H)}\right)\leq T$$ such that   the solution $\y$ exists on some interval $[0,T_0)$. 
\end{theorem}
\begin{proof}[Proofs of Theorem \ref{thm12}]
	 For the case $d=2,3$ and $r\in[1,3]$, calculations similar to the energy estimates \eqref{356} yields
 \begin{align}
	&\|\nabla\y_{\mathrm{N}}^{\lambda}(t)\|_{\H}^{2}+\mu\int_{0}^{t}\|\A\y_{\mathrm{N}}^{\lambda}(s)\|_{\H}^{2}+2\beta\int_{0}^{t}\||\nabla \y_{\mathrm{N}}^{\lambda}(s)||\y_{\mathrm{N}}^{\lambda}(s)|^{\frac{r-1}{2}}\|_{\H}^{2}\d s\nonumber\\&\leq C+\begin{cases}
		C\sup\limits_{s\in[0,t]}\|\y^{\lambda}_{\mathrm{N}}(s)\|_{\H}^{2}\int_0^t\|\nabla\y^{\lambda}_{\mathrm{N}}(s)\|_{\H}^{2(r-1)}\d s, &\ \text{for} \ d=2, \\
		C\sup\limits_{s\in[0,t]}\|\y^{\lambda}_{\mathrm{N}}(s)\|_{\H}^{3-r}\int_0^t\|\nabla\y^{\lambda}_{\mathrm{N}}(s)\|_{\H}^{3(r-1)}\d s+\int_{0}^{t}\|\nabla\y_{\mathrm{N}}^{\lambda}(s)\|_{\H}^{6}\d s, &\ \text{for} \ d=3, 
	\end{cases}
\end{align}
where $C= C\left(\|\y_0\|_{\V}, \|\f\|_{\mathrm{L}^2(0,T;\H)},\|\Phi(\boldsymbol{0})\|_{\H}\right).$ Then one can use the similar techniques as in the proof \cite[Theorem 2.1]{AIL} to pass $\lambda\to 0$ and then use Gronwall's inequality to obtain 
\begin{align*}
	\|\y_{\mathrm{N}}(t)\|_{\V}\leq C ,  \  \text{ for all } \ t\in[0,T_0],
\end{align*}
where constant $C$ is independent of $\mathrm{N}$ and $T_0=T$ for $d=2$ and $T_0<T$ for $d=3.$ Hence for large $\mathrm{N}$, we can choose $\mathrm{N}\geq C$ so that $\mathcal{B}_{\mathrm{N}}(\y_\mathrm{N})=\mathcal{B}(\y_{\mathrm{N}})$ and therefore $\y_\mathrm{N}=\y$ is a solution of \eqref{1p4} with the  regularity properties given in Theorem \ref{thm1.2}. So, $\y_\mathrm{N}$ satisfies \eqref{1p4} on the set 
$\mathrm{E}_\mathrm{N}=\{t\in[0,T]:\|\y_\mathrm{N}(t)\|_{\V}\leq\mathrm{N}\}.$
 By using Markov's inequality, we have 
\begin{align}\label{Q4}
	m([0,T]/\mathrm{E}_\mathrm{N})\leq\frac{C}{\mathrm{N}^2},
\end{align}
where $m$ is the Lebesgue measure. Since 
$m([0,T])=m(\mathrm{E}_\mathrm{N})+m([0,T]/\mathrm{E}_\mathrm{N}),$
 for large $\mathrm{N}$, from \eqref{Q4}, we conclude that $m([0,T])=m(\mathrm{E_\mathrm{N}})$ and $\y(\cdot)$ satisfies \eqref{1p4} for a.e. $t\in[0,T]$.    The case of $\y_0\in\V\cap\D(\Phi)$ and $\f \in \mathrm{L}^2(0,T;\H)$ in Theorem \ref{thm12} can be completed by a density argument as in the proof of \cite[Theorems 2.2 and 2.3]{AIL}. 
\end{proof}
	\end{appendix}

\medskip\noindent
\textbf{Acknowledgments:} The first author would like to thank Ministry of Education, Government of India - MHRD for financial assistance. K. Kinra would like to thank the Council of Scientific $\&$ Industrial Research (CSIR), India for financial assistance (File No. 09/143(0938)/2019-EMR-I).  M. T. Mohan would  like to thank the Department of Science and Technology (DST), India for Innovation in Science Pursuit for Inspired Research (INSPIRE) Faculty Award (IFA17-MA110).  The authors would like to thank Prof. S. S. Sritharan for helpful discussions. The authors would like to thank the reviewers for their valuable comments and suggestions.


\begin{thebibliography}	{99}
	
	
	
	
	
	
	
	
	\bibitem{F.T.} F. Abergel and R. Temam, On some control problems in fluid mechanics, \emph{Theor. Comput. Fluid Dyn.}, \textbf{1} (1990), 303--325.
		 
	\bibitem{CTATMN}  C. T.  Anh and T. M.  Nguyet, Time optimal control of the unsteady 3D Navier-Stokes-Voigt equations, \emph{Appl. Math. Optim.}, {\bf 79}(2) (2019), 397--426.
	
		 
	\bibitem{SNA}	S. N. Antontsev and H. B. de Oliveira, The Navier-Stokes problem modified by an absorption term, \emph{Appl. Anal.}, {\bf 89}(12) (2010), 1805--1825. 
	
	 \bibitem{JBPCK} J. Babutzka and P. C. Kunstmann, $L^q$-Helmholtz decomposition on periodic domains and applications to Navier-Stokes equations, \emph{J. Math. Fluid Mech.}, \textbf{20}(3) (2018), 1093--1121.
	  \bibitem{VB1} V. Barbu, \emph{Nonlinear Semigroups and Differential Equations in Banach Spaces}, Noordhoff International Publishing, 1976.
	  \bibitem{VB2} V. Barbu, \emph{Analysis and Control of Nonlinear Infinite Dimensional Systems}, Academic Press, 1993.
	   \bibitem{TiOp1} V. Barbu, The time optimal control of Navier-Stokes equations, \emph{Systems Control Lett.}, \textbf{30}(2-3) (1997), 93--100.
	  
	  
	  \bibitem{VB6} V. Barbu, \emph{Stabilization of Navier-Stokes Flows}, Communications and Control Engineering Series, Springer, London, 2011.
	  \bibitem{VB7} V. Barbu, \emph{Controllability and Stabilization of Parabolic Equations},
	  Progress in Nonlinear Differential Equations and their Applications  Subseries in Control 90, Birkhäuser/Springer, Cham, 2018. 
	  
	   \bibitem{VBLT} V. Barbu, I. Lasiecka and R. Triggiani, \emph{Tangential Boundary Stabilization of Navier-Stokes Equations}, Memoirs of the American Mathematical Society, 2006.
	   \bibitem{VBL} V. Barbu and C. Lefter, Internal stabilizability of the Navier-Stokes equations, \emph{System Control Lett.,} \textbf{48}(3-4)  (2003), 161--167.
	   
	     \bibitem{VBNH}  V. Barbu and N. H. Pavel,  Flow-invariant closed sets with respect to nonlinear semigroup flows, \emph{NoDEA Nonlinear Differential Equations Appl.}, {\bf 10}(1) (2003), 57--72.
	   
	     \bibitem{VBSS}  V. Barbu and S. S.  Sritharan,   Flow invariance preserving feedback controllers for the Navier-Stokes equation,   \emph{J. Math. Anal. Appl.}, {\bf 255}(1) (2001),  281--307.
	 
	  \bibitem{VBT} V. Barbu and R. Triggiani, Internal stabilization of Navier-Stokes equations with finite-dimensional controllers, \emph{Indiana Univ. Math. J.}, \textbf{53}(5) (2004), 1443--1494.
	  
	   \bibitem{VB22} V. Barbu, Semigroup approach to nonlinear diffusion equations, \emph{World Scientific Publishing Co. Pte. Ltd.}, Hackensack, 2022.
	
	 
	   \bibitem{OPHB} H. Brezis, \emph{Operateurs Maximaux et Semi-groupes de Contractions das les Espaces de Hilbert,} North Holland, New York, 1973.
	   
	   
	   \bibitem{ZCQJ} Z. Cai and Q. Jiu, Weak and strong solutions for the incompressible Navier-Stokes equations with damping, \emph{J. Math. Anal. Appl.}, {\bf 343}(2) (2008), 799--809.
	   
	   
	   
	   
     \bibitem{PGC} P. G. Ciarlet, \emph{Linear and Nonlinear Functional Analysis with Applications}, SIAM Philadelphia, Philadelphia (2013).	
     
    \bibitem{ZdSmP} Z. Denkowski, S. Migórski and N. S. Papageorgiou, \emph{An Introduction to Nonlinear Analysis: Applications}, Kluwer Academic Publishers, Boston, MA, 2003. 
    
    \bibitem{aanse} R. C. Doering and J. D. Gibbon, \emph{Applied Analysis of the Navier-Stokes Equations}, Cambridge University Press, Cambridge, 1995.
    
     
      \bibitem{FKS} R. Farwig, H. Kozono and H. Sohr,
     An $L^q$-approach to Stokes and Navier-Stokes equations in general domains,
     \emph{Acta Math.}, \textbf{195} (2005), 21--53.
     
    
      
     
     \bibitem{BPWFSS}  B. P. W. Fernando, S. S. Sritharan and M. Xu, A simple proof of global solvability of 2-D Navier-Stokes equations in unbounded domains, \emph{Differential Integral Equations}, {\bf  23}(3--4) (2010),   223--235.
     
     
     
     
    \bibitem{FMRT} C. Foias, O. Manley, R. Rosa, and R. Temam, \emph{Navier-Stokes Equations and Turbulence}, Cambridge University Press, 2008.
     
   \bibitem{DFHM} D. Fujiwara and H. Morimoto, An $L^r$-theorem of the Helmholtz decomposition of vector fields, \emph{J. Fac. Sci. Univ. Tokyo Sect. IA Math.}, \textbf{24}(3) (1977), 685--700.
   
     \bibitem{Fu}  A. V. Fursikov, \emph{Optimal Control of Distributed Systems: Theory and Applications}, American Mathematical Society, Providence, Rhode Island, 2000. 
     
    
    
     \bibitem{TFu}  T. Fukao, Variational inequality for the Stokes equations with constraint, \emph{Discrete Contin. Dyn. Syst.},  Dynamical systems, differential equations and applications,
     8th AIMS Conference. Suppl. Vol. I (2011), 437--446.
    
    
    
   \bibitem{qu2} J. Glimmand and A. Jaffe, \emph{Quantum Physics: A Functional Integral Point of View}, Springer-Verlag, Berlin, 1982.
    
    \bibitem{MGNK}  M. Gokieli, N. Kenmochi and M. Niezg\`odka, 
   Variational inequalities of Navier-Stokes type with time dependent constraints, \emph{J. Math. Anal. Appl.}, {\bf 449}(2) (2017), 1229--1247.
    
    
    
    \bibitem{G} M. D. Gunzburger, \emph{Perspectives in Flow Control and Optimization}, SIAM's Advances in Design and Control Philadelphia, 2003.
    
   \bibitem{KWH} K. W. Hajduk and J. C. Robinson, Energy equality for the 3D critical convective Brinkman-Forchheimer equations, \emph{J. Differential Equations}, \textbf{263}(11) (2017), 7141--7161.
   
   
   
  
 \bibitem{Hu} S. Hu and N. S. Papageorgiou, \emph{Handbook of Multivalued Analysis, Vol I: Theory,
 Mathematics and its Applications}, \textbf{419}, Kluwer Academic Publishers, Dordrecht, 1997.

   \bibitem{KT2} V. K. Kalantarov and S. Zelik, Smooth attractors for the Brinkman-Forchheimer equations with fast growing nonlinearities, \emph{Commun. Pure Appl. Anal.}, \textbf{11}(5) (2012), 2037--2054.
   
    
  
  
    \bibitem{AIL} A. I. Lefter, Navier-Stokes equations with potentials, \emph{Abstr. Appl. Anal.}, (2007), Art. ID 79406, 30 pp.
    \bibitem{AIL2} A. I. Lefter, Nonlinear feedback controllers for the Navier-Stokes equations, \emph{Nonlinear Anal.}, \textbf{71}(1-2) (2009), 301--316.
    
    
   
    
    \bibitem{AIL3} A. I. Lefter,  Strong solutions for Boussinesq equations with potentials, \emph{An. Stiint. Univ. Al. I. Cuza Iasi. Mat. (N.S.)}, {\bf  55}(2) (2009),  295--328. 
    
   
    
   
  
     \bibitem{SlGw} S. Li and G. Wang, The time optimal control of the Boussinesq equations, \emph{Numer. Funct. Anal. Optim.}, \textbf{24}(1–2) (2007), 163--180. 
     
     \bibitem{pL} P. Lindqvist, Notes on the Stationary $p$-Laplace equation, Springer Briefs in Mathematics, Springer, Cham, 2019. 
     
    \bibitem{MTT} P. A. Markowich, E. S. Titi and S. Trabelsi, Continuous data assimilation for the three dimensional Brinkman-Forchheimer-extended Darcy model, \emph{Nonlinearity}, \textbf{29}(4) (2016), 1292--1328. 
  
  \bibitem{JMSS} J. L. Menaldi and S. S. Sritharan, Stochastic 2-D Navier-Stokes equation, \emph{Appl. Math. Optim.}, \textbf{46} (2002), no. 1, 31--53.
  
    \bibitem{MT1} M. T. Mohan, On the convective Brinkman-Forchheimer equations, \emph{Submitted}.
    
     \bibitem{MT4} M. T. Mohan, Stochastic convective Brinkman-Forchheimer equations, \emph{Submitted}. \url{https://arxiv.org/pdf/2007.09376.pdf}.
     
    \bibitem{MT2}  M. T. Mohan, Well-posedness and asymptotic behavior of stochastic convective Brinkman–Forchheimer equations perturbed by pure jump noise, \emph{Stoch PDE: Anal. Comp.}, {\bf 10}(2) (2022), 614--690.
     
     
   
  
   \bibitem{TiOp2} M. T. Mohan, The time optimal control of two dimensional convective Brinkman-Forchheimer equations, \emph{Appl. Math. Optim.}, \textbf{84}(3)  (2021), 3295-- 3338.
   
   
   
   
    \bibitem{DAN}  D. A. Nield, The limitations for Brinkman–Forchheimer equation in modeling flow in a saturated porous medium   and at an interface, \emph{Int. J. Heat Fluid Flow}, {\bf  12}(3) (1991), 269--272.
    
   \bibitem{DAN1}  D. A. Nield, Modelling high speed flow of an incompressible fluid in a saturated porous medium, \emph{Transp. Porous   Media}, {\bf 14} (1) (1994), 85--88.
   
   \bibitem{JPSS} J. Peypouquet  and S. Sorin, Evolution equations for maximal monotone operators: asymptotic analysis in continuous and discrete time,     \emph{J. Convex Anal.}, {\bf 17}(3-4) (2010), 1113--1163.
   
    
    
    \bibitem{JCR} J. C. Robinson, \emph{Infinite-Dimensional Dynamical Systems: An Introduction to Dissipatives Parabolic PDEs and the Theory of Global Attractors}, Cambridge University Press, 2001.
     
   \bibitem{LSM}  J. C. Robinson and  W. Sadowski, A local smoothness criterion for solutions of the 3D Navier–Stokes equations, \emph{Rend. Semin. Mat. Univ. Padova}, \textbf{131} (2014), 159--178.
   
    \bibitem{JCR4} J. C. Robinson, J. L. Rodrigo and W. Sadowski, \emph{The Three-Dimensional Navier--Stokes equations, classical theory}, Cambridge University Press, Cambridge, UK, 2016.
     
     \bibitem{MRXZ}	M. R\"ockner and X. Zhang, Tamed 3D Navier-Stokes equation: existence, uniqueness and regularity, \emph{Infin. Dimens. Anal. Quantum Probab. Relat. Top.}, {\bf 12}(4) (2009), 525--549.
     
      \bibitem{qu1} B. Simon, \emph{Functional Integration and Quantum Physics}, Academic Press, New York, 1979.
       
     \bibitem{DsLt} D. T. Son and L. T. Thuy, Time optimal control problem of the 3D Navier-Stokes-$\alpha$ equations, \emph{Numer. Funct. Anal. Optim.}, \textbf{43}(6) (2022), 667--697.
    \bibitem{OpVf} S. S. Sritharan, An introduction to deterministic and stochastic control of viscous flow, in \emph{Optimal Control of Viscous Flow}, SIAM, Philadelphia, 1998.
    
    \bibitem{DsSz} D. Stone and S. Zelik, The non-autonomous Navier-Stokes-Brinkmann-Forchhiemer equation with dirichlet boundary conditions: dissipativity, regularity and attractors, \url{https://arxiv.org/pdf/2210.05580.pdf}. 
    
      \bibitem{Te} R. Temam, \emph{Navier-Stokes Equations: Theory and Numerical Analysis,} North-Holland, Amsterdam, 1984.
      
   
    \bibitem{RT1}  R. Temam, \emph{Infinite-Dimensional Dynamical Systems in Mechanics and Physics}, Second edition, Vol. 68, Applied Mathematical Sciences, Springer, 1997.
 
   
   
   \bibitem{wwxz} G. Wang, L. Wang, Y. Xu and Y.  Zhang, \emph{Time Optimal Control of Evolution Equations}, Springer, New York (2018).
   
  \bibitem{TCY}    T. C.  Yeh, R.  Khaleel and K. C. Carroll, \emph{Flow through Heterogeneous Geologic Media}, Cambridge University Press, 2015. 
   
   
   \bibitem{ZZXW}	Z. Zhang, X. Wu and M. Lu, On the uniqueness of strong solution to the incompressible Navier-Stokes equations with damping, \emph{J. Math. Anal. Appl.}, {\bf 377}(1) (2011), 414--419.
   
   
\end{thebibliography}
\end{document}